\documentclass[11 pt,reqno]{amsart}

\usepackage{mathrsfs,amsmath,amsthm,amssymb,amscd,graphicx,color,cite}

\newcommand{\beq}{\begin{equation}}
\newcommand{\eeq}{\end{equation}}
\def\epsilon{\varepsilon}

\newtheorem{thm}{Theorem}[section]
 \newtheorem{cor}[thm]{Corollary}
 \newtheorem{lem}[thm]{Lemma}
 
 \newtheorem{prop}[thm]{Proposition}
 
 \theoremstyle{definition}
 
 \theoremstyle{remark}

 \numberwithin{equation}{section}

\def\be#1 {\begin{equation} \label{#1}}
\newcommand{\ee}{\end{equation}}

\newcommand\deux{\operatorname{II}}

\begin{document}

\author{Pierre Germain}
\address{Courant Institute of Mathematical Sciences, New York University, 251 Mercer Street, New York, N.Y. 10012, USA}
\curraddr{}
\email{pgermain@cims.nyu.edu}

\author{Fr\'ed\'eric Rousset}
\address{Laboratoire de Math\'ematiques d'Orsay (UMR 8628), Universit\'e Paris-Sud, 91405 Orsay Cedex France et Institut Universitaire de France}
\curraddr{}
\email{frederic.rousset@math.u-psud.fr}

\title{Long wave limit for  Schr\"odinger maps}



\thanks{P. Germain is partially supported by NSF grant DMS-1101269, a start-up grant from the Courant Institute,
and a Sloan fellowship. The second author is partially supported by the ANR projects BoND (ANR-13-BS01-0009-01)
 and  DYFICOLTY (ANR-13-BS01-0003-01)}

\begin{abstract} 
We study  long wave limits for general Schr\"odinger maps systems into K\"ahler manifolds with a constraining potential
vanishing on a Lagrangian submanifold. We obtain  KdV type systems set on the tangent space of the submanifold.
Our general theory is applied to study the  long wave limits of the Gross-Pitaevskii equation and of the Landau-Lifshitz
systems for ferromagnetic and anti-ferromagnetic chains.
\end{abstract}

\maketitle

\begin{quote}
\footnotesize\tableofcontents
\end{quote}

\section{Introduction}

In this paper, we shall analyze  a  long wave limit  problem for general Schr\"odinger map systems into K\"ahler 
manifolds. More precisely, we shall study the  long wave limit in the presence of a nonlinear confining potential vanishing on 
a Lagrangian submanifold $\mathcal{L}$.  The "wave" regime, that is to say long time dynamics for data close to 
$\mathcal{L}$  was studied in \cite{ShatahZeng} where it was proven that the limit system is given by the wave map 
valued in the Lagrangian submanifold $\mathcal{L}$. The aim of the present paper is to study the asymptotic behaviour  on a 
longer time scale for data close to a point (that we denote by $0$ without loss of generality)  of $\mathcal{L}$. 
We shall derive a system of KdV-type equations taking values in the tangent space $T_{0}\mathcal{L}$.

Our motivation comes from the  study of  the Landau-Lifshitz systems and we shall apply our theory to these models. 
Our general equation is also linked to the study of long wave limits for Gross-Pitaevskii-type equations, 
that has received a lot of interest recently \cite{Chiron-Rousset, Bethuel-Gravejat-Saut-Smets1, Bethuel-Gravejat-Saut-Smets2, Chiron} 
(the set up considered in these papers can be seen as a special case of the point of view adopted in the present article: 
namely choosing $\mathcal{M}$ to be the Euclidean space $\mathbb{R}^2$, usually identified to $\mathbb{C}$, 
and the Lagrangian submanifold $\mathcal{L}$ to be the unit circle). These authors derive the KdV equation, which was first proposed in~\cite{KdV}. Rigorous derivations of KdV-type limits have been established for various other equations, 
 for example: general hyperbolic systems~\cite{BenYoussef-Colin, BenYoussef-Lannes},
water-waves \cite{Alvarez-Lannes, Craig},  the Euler-Poisson \cite{LLS,Guo-Pu} or Vlasov-Poisson systems \cite{HK} for plasma, the Euler-Korteweg model for capillary fluids \cite{Benzoni-Chiron}.

While all the results that were mentioned above concern scalar-valued KdV equation, our general framework leads to vector-valued KdV 
equations. Instances of vector-valued KdV equation are known: the Hirota-Satsuma system~\cite{Hirota-Satsuma} was introduced based on 
its interesting integrability properties, see also~\cite{WGHZ} for its generalization; and the Gear-Grimshaw system~\cite{Gear-Grimshaw, BPST} to model the interaction of internal waves in a fluid. The model which we derive turns out to be related to the Gear-Grimshaw system, this point will be discussed in Section~\ref{ROTLKS}.

\subsection{The geometric picture}

We consider $\mathcal{M}$ a  $2d$ dimensional-K\"ahler manifold, and denote by $\nabla$ the connection compatible with the Riemannian metric (that we denote $\langle \;,  \; \rangle$ or simply $\cdot$), and  by $i$ the complex structure.
The system that we shall study is of the form
\begin{equation}
\label{SM1}
\boxed{
\partial_{s} \Gamma = i \big({1 \over 2 } \nabla_{y} \partial_{y} \Gamma + B(\Gamma) \partial_{y} \Gamma -  V'(\Gamma) \big)
}
\end{equation}
where   the unknown $\Gamma$ is a  map,    $\Gamma: \, \mathbb{R}_{+} \times \mathbb{R} \rightarrow \mathcal{M}$. We  use the notation
$\nabla_{y}$ for the riemannian connection  on $\Gamma^{-1} TM$, the pull-back of $T\mathcal{M}$ by $\Gamma$; it is uniquely defined by the condition 
that for a vector field $Y(y)$ along  the curve $y\mapsto \Gamma(s,y)$, we have 
$\nabla_{y} Y(y)= (\nabla_{\partial_{y}  \Gamma} X)_{|\Gamma(s,y)}$ for any vector field $X$ on $\mathcal{M}$ such that $X(\Gamma(s,y))= Y(y)$, $\forall y$.  
In the zero order  term, we use the notation $V'(\Gamma)$ for the Riemannian gradient of the map $V: \, \mathcal{M}\rightarrow \mathbb{R}$. 
 For the first order term, $ B(\Gamma): \, T_{\Gamma} \mathcal{M} \rightarrow  T_{\Gamma }\mathcal{M}$ is a smooth skew-symmetric tensor ($B^*= -B$).
  When there exists  a vector field $\mathcal{W}(\Gamma)$ such that
  $$ B(\Gamma)=  \nabla \mathcal{W}(\Gamma)^* - \nabla \mathcal{W}(\Gamma)$$
  (hence in particular when $B=0$) 
the above system  is formally the Hamiltonian flow of the energy functional
\begin{equation}
\label{energieGamma} E(\Gamma) = { 1 \over 4} \int_{\mathbb{R}} | \partial_{y} \Gamma |^2 \, dy - \int_{\mathbb{R}}  \langle W(\Gamma), \partial_{y} \Gamma \rangle \,dy+  \int_{\mathbb{R}} V(\Gamma) \, dy.
\end{equation}
given the symplectic form on $L^2(\Gamma^{-1} TM)$
$$
\omega(X,Y) = \int_{\mathbb{R}} \langle i X\,,\, Y \rangle \,dy.
$$
  
As in \cite{ShatahZeng}, we shall assume that the potential $V$  is smooth and confining to a Lagrangian submanifold $\mathcal{L}$:
\begin{equation}
\tag{H1}
\label{H1}
\begin{split}
\mbox{V smooth, $V(p) \geq 0$ on $\mathcal{M}$,  $V(\Gamma)=0$ on $\mathcal{L}$
and  $V''(p)_{|N_{p}\mathcal{L}\times N_{p}\mathcal{L}} =  2 \lambda I,$ for some  $\lambda>0$}
\end{split}
\end{equation}
where we use the notation 
$T_{p}\mathcal{L}$ and $N_{p}\mathcal{L}$ for the tangent and normal  spaces  to $\mathcal{L}$ at $p$, which are subspaces
of $T_{p}\mathcal{M}$. We also use the notation $V''$ as a short hand for  $\nabla V'$. Compared to \cite{ShatahZeng}, we make the slightly stronger assumption that $V''(p)$ is constant, this is sufficient to handle all our examples.

For the skew symmetric tensor $B$ on $T\mathcal{M}$, we shall make the following additional assumptions
\begin{equation}
\tag{H2}
\label{H2}
\begin{split}
& \mbox{if $p \in \mathcal{L}, \: \nabla B(p)= 0, \: iB(p):\, {N}_{p} \mathcal{L} \rightarrow {N}_{p} \mathcal{L},  \;  (Bi)(p)= -(iB)(p),$}\\
& \qquad \qquad \mbox{$B^2(p)= - \mu I$ for a constant $\mu < \lambda$}
\end{split}
\end{equation}
(note that if the properties $\nabla B = 0$, $Bi = -iB$, $B^2 = -\mu I$ were true on $\mathcal{M}$ instead of only $\mathcal{L}$, this would turn $\mathcal{M}$ into a hyperk\"ahler manifold). It is easy to see that $\mu$ is necessarily nonnegative. Again, we make slightly stronger assumptions than  in \cite{ShatahZeng} in order to handle our asymptotic regime. Note that the case $B=0$ is already
interesting and will  allow to handle many physical examples. Since  $iB$ is skew symmetric, the first order term in \eqref{SM1}, must be dominated  by the two other terms
 in order to get a well-posed problem. This is the reason for the assumption $\mu <\lambda$ (see the energy estimates in the proof for more details).

It will be convenient to introduce a number $c>0$ defined  by 
\begin{equation}
\label{cdef} c^2= \lambda - \mu
\end{equation}
which will play the role of a sound speed.

To define the long wave regime that we will study, it is  convenient to use the geodesic normal coordinate system 
in the vicinity of  $\mathcal{L}$ (we shall study precisely this coordinate system and its properties in section \ref{GPII}). 
Let us choose an arbitrary  point of  $\mathcal{L}$ that we denote by $0$, we can then define a coordinate system for  any  $\Gamma$ in a vicinity of zero  by  
$\Gamma= \Psi(p, N)$ with $p \in \mathcal{L}$, $N \in N_{p}\mathcal{L}$
and  $\Psi (p, N)=  \exp_{p}^\mathcal{M}( N)$ with $\exp^\mathcal{M}$ the exponential map on $\mathcal{M}$.
In this coordinate system,  we obtain  that $V$ is of the form
\begin{equation}
\label{Vdefbis}
V(\Psi(p,N))= \lambda |N|^2 + V_{1}(p)(N,N,N) + V_{2}(p,N),
\end{equation}
where $V_1$ is smooth in $p$ and trilinear and symmetric in $N$, and $V_2$ is a smooth function of size $O(|N|^4)$, which will not contribute in the analysis.

\subsection{Relevant scalings} To get the wave regime that was studied in \cite{ShatahZeng}, we look for solutions of  \eqref{SM1} under the form
$$
\Gamma(s,y)= u^\epsilon(\epsilon s, \epsilon y), \quad u^\epsilon(t,x)= \Psi(p^\epsilon(t,x), \epsilon  n^\epsilon(t,x)) .$$
This yields the rescaled equation for $u^\epsilon$
\begin{equation}
\label{wave1} \partial_{t} u^\epsilon= i \big(  {\epsilon \over 2 } \nabla_{x} \partial_{x} u^\epsilon + B \partial_{x} u^\epsilon  - {1 \over \epsilon }   V'(u^\epsilon) \big), \quad u^\epsilon= \Psi(p^\epsilon, \epsilon n^\epsilon).
\end{equation} 
With our assumption on the shape of the  nonlinear term  $V(u)$, we get  at leading order in $\epsilon$ the system  
\begin{equation}
\label{wave1'} \partial_{t} p - i B \partial_{x}p =  - 2  i \lambda  n, \quad \nabla_{t}^\perp n - i B \nabla_{x}^\perp n = {1 \over 2} i (\nabla_{x}^\top)\partial_{x} p.
\end{equation}
Here $\nabla^\top $ and $\nabla^\perp$ stand for the covariant derivatives respectively  on the tangent and normal bundles
of $\mathcal{L}$, and the two equalities above hold in $T_p \mathcal L$ and $N_p \mathcal L$ respectively; more details on these objects and on the way 
to perform these computations will be given below. Note that since we assumed that $\mathcal{L}$ is Lagrangian, $i$ maps $T_{p}\mathcal{L}$ on $\mathcal{N}_{p}\mathcal{L}$
 and vice-versa and that  thanks to~\eqref{H2},  we have that on $\mathcal{L}$,  $T_{p} \mathcal{L}$ and $\mathcal{N}_{p} \mathcal{L}$
 are stable subspaces of $iB$. By applying $\nabla_{t}^\top -  Bi \nabla_{x}^\top$ to the first equation of \eqref{wave1'}, we get (since $\nabla i= 0$) that
 $$ (\nabla_{t}^\top -  Bi \nabla_{x}^\top) ( \partial_{t} - iB \partial_{x})p = - 2 i \big( \nabla_{t}^\perp - i B \nabla_{x}^\perp \big)n= \lambda \nabla_{x}^\top \partial_{x}p.$$
Next,
$$  (\nabla_{t}^\top -  Bi \nabla_{x}^\top) ( \partial_{t} - iB \partial_{x})p= \nabla_{t}^\top \partial_{t} p  - ( Bi+ i B) \nabla_{x}^\top \partial_{t} p  - B^2 \nabla_{x}^\top \partial_{x} p =
 \nabla_{t}^\top\partial_{t} p  + \mu \nabla_{x}^\top \partial_{x} p$$
  thanks to~\eqref{H2}. Combining these two equalities, we obtain that $p \in \mathcal{L}$  solves
  $$ \nabla_{t}^\top \partial_{t}p= c^2 \nabla_{x}^\top \partial_{x} p$$
   with $c$ defined by \eqref{cdef} which is a wave map system on $\mathcal{L}$.
 
The rigorous justification of this regime was performed in \cite{ShatahZeng}.
In particular, it is proven that in this regime smooth  solutions of \eqref{wave1} exist on $[0,T]$ with $T$ independent of $\epsilon$
and converge to the solutions of the wave map system.
   
Such a  one-dimensional wave equation gives rise to one wave moving
to the left, and one moving to the right; our aim here will be to study the dynamics of one of these waves 
on longer times  in the case where the amplitude is smaller. 
We shall study the wave moving to the right close to some  point of $\mathcal{L}$ that we denote by zero. We use the ansatz
\begin{equation}
\label{KdVscaling}
\Gamma(s,y) = u^\epsilon( \epsilon^3 s, \epsilon(y- c s)), \quad u^\epsilon(t,x)= \Psi(p^\epsilon, \epsilon^2n^\epsilon), \quad p= \Phi(\epsilon \phi^\epsilon)
\end{equation}
where  $\Phi$ is the exponential map at $0$ on $\mathcal{L}$: $\Phi=exp_{0}^\mathcal{L}.$ We are thus studying the wave
moving to the right in a smaller amplitude regime ($\epsilon^2$ instead of $\epsilon$ in the direction normal to $\mathcal{L}$) and on a much longer time scale ($ 1/\epsilon^3$ against $1/\epsilon$).
We get for $u^\epsilon$ the system
\begin{equation}
\label{SMKdV}
\epsilon^2 \partial_{t} u^\epsilon  - c\partial_{x} u^\epsilon -i B \partial_{x} u^\epsilon = i \big( { \epsilon  \over 2} \nabla_{x}\partial_{x} u^\epsilon - {1 \over \epsilon} V'(u^\epsilon) \big).
\end{equation}    

\subsection{Uniform well-posedness}
We shall first prove  that for appropriate initial data,   smooth  solutions of \eqref{SMKdV} exist on an interval of time $[0,T]$ with $T$ independent of $\epsilon$ 
and satisfy uniform estimates.  More precisely, we will  also justify that on $[0,T]$, the representation
$$u^\epsilon(t,x)= \Psi(p^\epsilon(t,x), \epsilon^2n^\epsilon(t,x)), \quad p^\epsilon(t,x)= \Phi(\epsilon \phi^\epsilon(t,x))
$$
holds true and that $\phi^\epsilon$ and $n^\epsilon$ also satisfy  uniform estimates.
The following statement involves the energy functional $\mathcal{E}_s((\phi^\epsilon, n^\epsilon), t)$, which will be defined in 
Section~\ref{sectionprelim} (see \eqref{Esdef1}, \eqref{Esdef2}).  By a slight abuse of notation, 
 we shall often use the notation $ \mathcal{E}_{s}(u^\epsilon, t)$ in place of
$ \mathcal{E}_s((\phi^\epsilon, n^\epsilon), t)$.  Without going into details for the moment, it satisfies
$$ \| \partial_{x} \phi^\epsilon(t)\|_{H^{s}} + \|n^\epsilon(t)\|_{H^s}   \leq \mathcal{E}_{s}((\phi^\epsilon, n^\epsilon),t),$$
and the reader can heuristically think of $\mathcal{E}_s$ as being equivalent to $\| \partial_{x} \phi^\epsilon(t)\|_{H^{s}} + \|n^\epsilon(t)\|_{H^s}$.

We shall fix $r>0$ such that $\Phi= \exp_{0}^\mathcal{L}$ is well defined on $B(0, r)$ (the ball in $\mathbb{R}^d$)  and a diffeomorphism  on its image.

\begin{thm}[Uniform existence]
\label{theounif}
Assume that~\eqref{H1} and~\eqref{H2} hold, and let $s \geq 1$, $c_{0}>0$ such that $2c_{0}<r$.
 For every $M>0$    such that    
 $$
\mathcal{E}_{s}((\phi^\epsilon, n^\epsilon),t=0) \leq M, \quad  \epsilon \| \phi^\epsilon_{|t=0} \|_{L^\infty} \leq  c_{0}, \quad \forall \epsilon \in (0, 1]
$$
 there exists 
   $T >0$,  $E>0$ and  $\epsilon_{0}>0$ such that for every $\epsilon \in (0, \epsilon_{0}]$, we have a unique solution 
   $$u^\epsilon(t,x) =  \Psi(p^\epsilon, \epsilon^2n^\epsilon), \quad p^\epsilon= \Phi(\epsilon \phi^\epsilon)$$
    of \eqref{SMKdV} on $[0, T]$ that satisfies the uniform estimates
$$ \mathcal{E}_{s}((\phi^\epsilon, n^\epsilon),t) \leq E, \quad  \epsilon \| \phi^\epsilon (t)  \|_{L^\infty} \leq  2 c_{0} \quad \forall t \in [0,T].$$
\end{thm}
The estimate on the $L^\infty$  norm of $\epsilon \phi^\epsilon$ in the previous statement is needed in order to use the  exponential coordinates  in a vicinity of the reference point $0$ on $\mathcal{L}$.

The most difficult step in the proof of  the above result  is to provide a priori uniform estimates on  
$\phi^\epsilon$ and $n^\epsilon$. When dealing with the Gross-Pitaevskii equation, as in \cite{Chiron-Rousset} for example, these estimates can be obtained
by using the hydrodynamical form of the equation sometimes called quantum Euler equation which follows from the Madelung transform (or its modified version due to Grenier \cite{Grenier}).

 In our more general framework, the representation $u^\epsilon = \Psi(p^\epsilon, \epsilon^2 n^\epsilon)$
 can be thought of
as a generalized Madelung transform. The first step in the analysis is thus to derive
a suitable hydrodynamical system on $(p^\epsilon, n^\epsilon)$ (we can drop the subscript $\epsilon$ for the sake of clarity for the moment) and to study its properties. It turns out there are several new difficulties that  do not occur in the  study of the standard hydrodynamical system derived from Gross-Pitaevskii, the main one being that  our hydrodynamical system suffers
from a lack of symmetry away from $\mathcal{L}$. 
As a consequence, 
we are only able to prove an estimate on the hydrodynamical system with a loss of derivatives, but a gain in $\epsilon$.
In order to compensate for this loss of derivatives, we add to the natural energy for the hydrodynamical system, $\mathcal{E}_{s,2}$,
a correction $\mathcal{E}_{s,1}$, and define
$$\mathcal{E}_{s}(u^\epsilon, t)\overset{def}{=}  \mathcal{E}_{s, 1}(u^\epsilon,t) + \mathcal{E}_{s, 2}(u^\epsilon,t).$$
Our plan is as follows:
\begin{itemize}
\item The control of the energy $\mathcal{E}_{s, 1}$ is obtained  by working directly on the Schr\"odinger map system in the spirit
of \cite{ShatahZeng}. This quantity controls higher order derivatives but has a singular behaviour in $\epsilon$.
\item Then we derive the hydrodynamical system and use it to estimate $\mathcal{E}_{{s, 2}}$
 that gives a control of $\|\partial_{x} \phi^\epsilon \|_{s} + \| n^\epsilon \|_{s}$. The terms with loss of derivatives that
occur are controlled thanks to $\mathcal{E}_{s, 1}$.
\end{itemize}
All these estimates are valid if 
  $\|\epsilon \phi^\epsilon \|_{L^\infty}<r$ so that the exponential on $\mathcal{L}$ is also well defined. 
   In order to close the argument here, we thus have  to estimate  $\|\epsilon \phi^\epsilon \|_{L^\infty}$. 
Note that this estimate was not needed in the analysis for  Gross-Pitaevskii: in this case, the hydrodynamical system does not involve
   $\phi^\epsilon$ (only $\partial_{x} \phi^\epsilon$ which plays the role of the fluid velocity  and higher derivatives) and, since  $\mathcal{L}$ is the unit circle, the exponential is globally defined. 
   The
    estimate does not follow from the control of $\mathcal{E}_{s}$ and Sobolev embeddings since
      $\mathcal{E}_{s}$ only controls $\|\partial_{x} \phi^\epsilon \|_{s}$.
    Also,  note that
    by direct time integration of the equation for $\phi$ in the hydrodynamical system,  we only easily get an estimate of  $\|\epsilon^2 \phi^\epsilon \|_{L^\infty}$.
     We will proceed as follows.
     \begin{itemize}
     \item To estimate $\|\epsilon \phi^\epsilon \|_{L^\infty}$, we use the hydrodynamical system to get
      that, up to well controlled terms, 
       $\|\epsilon \phi^\epsilon \|_{L^\infty ([0, T] \times \mathbb{R})}$ can be estimated  by 
       $ \sup_{x} \int_{0}^T { 1 \over \epsilon} |W^\epsilon(s, x)| \, ds$
        where
         $$ W= (c+iB) D\Phi \partial_{x} \phi^\epsilon - 2 i \lambda n^\epsilon$$
          We can control this quantity uniformly in $\epsilon$ by observing that $|W^\epsilon|^2$ solves a transport
           equation at speed $\epsilon^{-2}.$
     \end{itemize}
   
  Let us now explain a little more the derivation of the hydrodynamical system, at least in a simplified
   framework.
When deriving the hydrodynamical system and studying its structure, both the geometry of $\mathcal{L}$ and the geometry of $\mathcal{M}$   play a role, together with the local structure of the potential $V$. In particular, in order to get the hydrodynamical system in the general case, 
we need to use  a connection on $N \mathcal{L}$ and to understand how the 
Schr\"odinger map system can be split into  a tangential and a normal equation away from $\mathcal{L}$,
in such a way that $i$ still exchanges the tangential and the normal directions.

We begin with the study of the simpler case in which
$\mathcal{L}$ is a Lagrangian submanifold of $\mathbb{R}^{2d}$ (for instance with the complex structure of $\mathbb{C}^d$).
In this case, the representation \eqref{KdVscaling} $u= \Psi (p^\epsilon, \epsilon^2n^\epsilon), \, p^\epsilon = \Phi(\epsilon \phi^\epsilon)$ is much simpler
since it reads
$$ u^\epsilon= \Phi(\epsilon \phi^\epsilon) + \epsilon^2 n^\epsilon$$
(indeed, the geodesics are straight lines in $\mathbb{R}^{2d}$ and thus only the exponential on $\mathcal{L}$ is nonlinear).
It is then rather straightforward to get a hydrodynamical system by plugging this  ansatz into \eqref{SMKdV}
and by splitting the system into the tangential part  and the normal part to $\mathcal{L}$. We obtain
the system  
\begin{equation}
\label{eqhydro}
\left\{ 
\begin{array}{l} 
\displaystyle 
S_1 \left( \partial_t \phi^\epsilon - \frac{c}{\epsilon^2} \partial_x \phi^\epsilon \right)  - { 1 \over \epsilon^2 } i B S_{1} \partial_{x} \phi^\epsilon= i \left[ \frac{1}{2} \deux^{\perp}(S_1 \partial_x \phi^\epsilon, D\Phi \partial_{x} \phi^\epsilon) + \frac{1}{2}
(\nabla^{\perp}_x)^2 n^\epsilon  -   2 \lambda \frac{n^\epsilon}{\epsilon^2} \right. \\
 \displaystyle  \mbox{\hspace{10cm}}\left. - F_{1}(p)(n^\epsilon,n^\epsilon)-  {1 \over \epsilon^4}P^\perp R^V(p,\epsilon^2n) \right] \\
\displaystyle
\nabla_t^\perp n^\epsilon - \frac{c}{\epsilon^2} \nabla^{\perp}_x n^\epsilon -  { 1 \over \epsilon^2 } iB  \nabla_{x}^\perp n^\epsilon  = i \left[ \frac{1}{2 \epsilon^2} \nabla_x^\top (S_1 \partial_x \phi^\epsilon)  + \frac{1}{2} \deux^{\top} \left(
 D\Phi \partial_x \phi^\epsilon \, , \,\nabla_x^\perp n^\epsilon  \right)  \right. \\
 \displaystyle  \mbox{\hspace{10cm}}\left.- {1 \over \epsilon^5} P^\top R^V(p^\epsilon, \epsilon^2 n^\epsilon) \right]
\end{array}
\right.
\end{equation}
where we denote
$$
S_1 \overset{def}{=} S_0 D\Phi \qquad \mbox{with} \qquad S_0 \overset{def}{=} \operatorname{Id} + \epsilon^2 \deux^\top \left( \cdot\,,\,n^\epsilon \right)
$$
and refer to the subsection~\ref{diffgeo} for the precise definition of the second fundamental form of the tangent bundle $\deux^{\perp}$; the second fundamental form of the normal bundle $\deux^\perp$; and the connection on the normal bundle $\nabla^\perp$.
In this system these objects are always computed at the point $p = \Phi(\epsilon \phi^\epsilon) \in \mathcal{L}.$ The terms involving $R^V$
are  harmless terms that  come from  the higher order terms in the potential.

The case where $\mathbb{R}^{2d}$ is replaced by a general K\"ahler manifold $\mathcal{M}$ is more complicated
since one needs to define appropriate generalizations of the tangential and normal projections away from
$\mathcal{L}$.

\subsection{Derivation of the vector KdV equation}
The uniform estimates of Theorem \ref{theounif} are the key to the following result that justifies rigorously
 the KdV asymptotics.
 Before stating the result,recall that $ \deux^\perp_{p}$ is the second fondamental form of the tangent bundle of $\mathcal{L}$. Its definition and basic properties are recalled in section \ref{sectionprelim}.
  In view of the expansion \eqref{Vdefbis} of the potential, it is convenient to define
 $F_{1}(p)(N_1,N_2) \in N_{p} \mathcal{L}$ by the  formula: 
\begin{equation}
\label{F1def} F_{1}(p) (N_1,N_2) \cdot N_3 \overset{def}{=} 3 V_{1}(p)(N_1,N_2,N_3), \quad \forall N_1,N_2,N_3 \in N_{p} \mathcal{L}.
\end{equation}

The second main result of this paper is the following.
\begin{thm}[KdV limit]
\label{theoKdV}
Assume~\eqref{H1},~\eqref{H2} and  that, for some $s \geq 3$,  and for some  $M>0$, we have the uniform estimate
$$
\mathcal{E}_{s}(u^\epsilon,t=0) \leq M, \quad   \forall \epsilon\in (0, 1].
$$
Assume furthermore that the initial data
$
u_{0}^\epsilon = \Psi \big(\Phi( \epsilon\phi_{0}^\epsilon), \epsilon^2 n_{0}^\epsilon\big)$
 is such that  $\epsilon \phi_0^\epsilon \rightarrow 0$ in $L^\infty(\mathbb{R})$ and that  there exists $A_{0} \in L^2(\mathbb{R}, T_{0}\mathcal{L})$ such that 
 \begin{equation}
 \label{hypW}  \epsilon^2  \phi^\epsilon_{0}\rightarrow 0, \quad   (c+ (iB)_{\Phi(\epsilon \phi_{0}^\epsilon)})  D\Phi_{\epsilon \phi^\epsilon_{0}} \partial_{x}\phi^\epsilon_{0} \rightarrow A_{0}, \quad 
  2  i \lambda  n^\epsilon_{0} \rightarrow A_{0} 
 \end{equation}
 in $L^2(\mathbb{R})$  when $\epsilon$ tends to zero. Then $ (c+(iB)_{\Phi(\epsilon \phi^\epsilon)}) D\Phi \partial_{x}\phi^\epsilon$ and  $ 2  i  \lambda n^\epsilon $ converge to $A \in T_{0}\mathcal{L}$
in  $\mathcal{C} ([0, T], L^2)$  where $A$ is 
the unique solution of the KdV type system
\begin{equation}
\label{eqKdVlimitintro}  
\boxed{
2  c\partial_{t} A  = \frac{1}{4} \partial_{xxx}  A  + \left( \frac{3}{2}- \frac{2 \mu}{\lambda} - \frac{2c}{\lambda}i_0 B_0 \right) i_0 \deux^\perp_{0} \left(  \partial_{x} A, A\right)  - \frac{i_0}{2\lambda} F_{1,0}(i_0 \partial_{x} A, i_0 A).
}
\end{equation}
with initial data
$$A(t=0) = A_{0}.$$ 
\end{thm}
In the above system,  the subscript $0$ indicates that all the involved tensors are evaluated at $0$. and we have set $ F_{1,0} = F_1(0)$. The convergence in~\eqref{hypW} should be understood in a natural way, either by embedding $\mathcal{M}$ (locally) in an Euclidean space, or by identifying tangent spaces through parallel transport.

Let us now explain how, at least formally, we expect that the limit evolves according to  \eqref{eqKdVlimitintro}; we will 
focus on the simpler framework where the K\"ahler manifold is Euclidean, and the hydrodynamical system is simply given by \eqref{eqhydro}.
  Due to the singular  $\epsilon^{-2}$ terms, we expect that in the limit, if $(\partial_{x} \phi^\epsilon, n^\epsilon)$
   converges to $(\partial_{x} \phi, n)$,  where
  $$(c + iB) D\Phi \partial_x \phi=  2i\lambda n\overset{def}{=} A.$$
 In order to get the equation satisfied by $A$,  we can apply  $\nabla_x^\top$ to the first equation
  of \eqref{eqhydro}. The system then becomes
$$
\left\{
\begin{array}{l}
\displaystyle \nabla_{t}^\top (D\Phi \partial_{x}\phi^\epsilon) =  {1 \over \epsilon^2} \nabla_{x}^\top W^\epsilon  +  {1 \over 2}  (\nabla_{x}^\top)^3 (in^\epsilon)  - 2 i F_{1}(\nabla_{x}^\perp n^\epsilon, n^\epsilon) \\
\displaystyle \qquad \qquad \qquad \qquad +i   \deux^\perp( \nabla_{x}^\top (D\Phi \partial_{x}\phi^\epsilon), D\Phi \partial_{x}\phi^\epsilon) - 4 \lambda i \deux^\perp ( \nabla_{x}^\top (in^\epsilon), i n^\epsilon) +O(\epsilon) \\
\displaystyle \nabla_{t}^\top  ( 2 \lambda in^\epsilon) = - {1 \over \epsilon^2}( c-iB) \nabla_{x}^\top W^\epsilon
   +   2  i  \lambda  \deux^\perp \big(D\Phi\partial_{x}\phi^\epsilon , \nabla_{x}^\top i  n^\epsilon \big)  + i \lambda  \deux^\perp (\nabla_{x}^\top(D\Phi \partial_{x} \phi^\epsilon), i n^\epsilon\big)  + O(\epsilon).
\end{array} \right.
$$
where $W^\epsilon=(c+iB) D\Phi \partial_{x} \phi^\epsilon - 2 i \lambda n^\epsilon. $
Multiplying the first line by $(c-iB)$, and adding it to the second line, the singular $\frac{1}{\epsilon^2}$ terms cancel, giving the equation
\begin{equation}
\label{intropresqueKdV}
\begin{split}
\nabla_{t}^\top \big(  (c-iB) D\Phi \partial_{x} \phi^\epsilon &+ 2 \lambda i n^\epsilon)  = (c-iB) \left[
{1 \over 2}  (\nabla_{x}^\top)^3 (in^\epsilon)  - 2 i F_{1}(\nabla_{x}^\perp n^\epsilon, n^\epsilon) \right. \\ 
& \left. +i   \deux^\perp( \nabla_{x}^\top (D\Phi \partial_{x}\phi^\epsilon), D\Phi \partial_{x}\phi^\epsilon) - 4 \lambda i \deux^\perp ( \nabla_{x}^\top (in^\epsilon), i n^\epsilon) \right] \\
& +  2  i  \lambda  \deux^\perp \big(D\Phi\partial_{x}\phi^\epsilon , \nabla_{x}^\top i  n^\epsilon \big)  + i \lambda  \deux^\perp (\nabla_{x}^\top(D\Phi \partial_{x} \phi^\epsilon), i n^\epsilon\big)  + O(\epsilon).
\end{split}
\end{equation}
If $\epsilon \phi$ tends to zero and  $(\partial_{x} \phi^\epsilon, n^\epsilon)$
   converges to $(\partial_{x} \phi, n)$  sufficiently strongly
with  the constraint
 $(c + iB) D\Phi \partial_x \phi=  2i\lambda n=  A,$
 by using some algebraic manipulations in order to express all the quantities with respect to $A$, 
  we finally 
get
\eqref{eqKdVlimitintro}.

In order to justify rigorously the above derivation, the main difficulty is to prove that  $\epsilon \| \phi^\epsilon \|_{L^\infty}$ tends to 
zero and to get strong  compactness for $n^\epsilon$ and $\partial_{x} \phi^\epsilon$ (the space compactness
 is a direct consequence of the uniform estimates of Theorem \ref{theounif}, the difficulty is the time compactness) in order to pass to the limit in  \eqref{intropresqueKdV}.  We point out again that 
  in our geometric setting all the tensors involved in the hydrodynamical system 
   are taken at $\epsilon \phi^\epsilon$,  hence we really need to prove a strong convergence of $\epsilon \phi$
   in order to pass to the limit. 
 
 To prove that  $\epsilon \| \phi^\epsilon \|_{L^\infty}$ tends to 
zero, we use again the hydrodynamical system and the link with $W^\epsilon= (c+iB) D\Phi \partial_{x} \phi^\epsilon - 2 i \lambda n^\epsilon$. We first prove that it suffices 
 to get that  $\sup_{[0, T ]} \|W^\epsilon \|_{L^2}$ tends to zero.
  To get this we can use the conserved quantities of \eqref{SM1} (at least in the case when 
  $B = \nabla W^\star - \nabla W$).
Note  that $\| W^\epsilon \|_{L^2}^2$ is of order $\sim \epsilon^4$ and, to leading order in $\epsilon$, conserved by the flow of~\eqref{SM1}. Indeed,
$$
\frac{1}{\epsilon^4} \| W^\epsilon \|_{L^2}^2 = 4 \lambda E + 4 \lambda c P + O(\epsilon),
$$
where the energy $E$ and the momentum $P$ 
$$
\left\{ 
\begin{array}{l}
\displaystyle E = \frac{1}{4} \int \left[ \epsilon^2 |\partial_x u|^2 + V(u) + \epsilon W(u) \cdot \partial_x u \right]\,dx \\
\displaystyle P = \int u \cdot i \partial_x u\,dx
\end{array}
\right.
$$
are conserved quantities of~\eqref{SM1}.
By assumption on the data   $(\phi_0,n_0)$, we thus get that
$$  \frac{1}{\epsilon^4} \| W^\epsilon (t) \|_{L^2}^2 = { 1 \over \epsilon^4} \|W^\epsilon (0)\|_{L^2}^2 + O(\epsilon)= o(1).$$
 This provides the desired estimate on $W^\epsilon$ which in turn implies that $\epsilon  \|\phi\|_{L^\infty}$
 tends to zero uniformly on $[0, T]$.
 
 Once we have obtained this estimate, we can use again the hydrodynamical system \eqref{eqhydro}
  to  also obtain strong compactness  for  the quantity
  $ U^\epsilon = (c-iB) D\Phi \partial_{x} \phi^\epsilon + 2 i \lambda n^\epsilon$.
   This is a consequence of the fact  that $\nabla_{t}^\top U^\epsilon = O(1)$. 
    This in turn  allows to justify the convergence towards a solution of \eqref{eqKdVlimitintro}
     from some algebraic manipulations and standard (weak) convergence arguments. 

\subsection{Organization of the article}   
\begin{itemize}
\item We show in Section \ref{sectionexamples} how our general theory can be used to obtain the long wave limit
for the Gross-Pitaevskii equation as well as for the Landau-Lifshitz equations for ferromagnetic and anti-ferromagnetic chains.
\item In Section \ref{sectionprelim}, we recall some useful geometric facts  and define the various
norms and functionals that we will be using.
\item In Section \ref{sectioneuclid}, we prove theorems \ref{theounif} and \ref{theoKdV}
in the case where $\mathcal{M} = \mathbb{R}^{2d}$.
\item Next, the proof of theorems \ref{theounif} and \ref{theoKdV} in the general case is presented in Section \ref{sectionKahler}. 
The additional geometric notions that are needed for the proof are given in subsection \ref{sectiongeom}.
\item Finally, we say in Section~\ref{ROTLKS} a few words about the properties of the KdV system that we derived.
\end{itemize}

\section{Examples}

In this section, we apply our general result to various physical situations

\label{sectionexamples}

\subsection{The Gross-Pitaevskii equation}

\subsubsection{Scalar case}
We consider first the  Gross-Pitaevskii equation which is a classical model for nonlinear optics, superfluids and Bose-Einstein condensates
(see~\cite{RB} for a recent review)
\begin{equation}
\label{GP}
\partial_{t} \Gamma = i \left( { 1 \over 2 } \partial_{xx} \Gamma + \Gamma ( 1 - | \Gamma |^2) \right), \quad t >0, \, x\in \mathbb{R}
\end{equation}
where the unknown  $\Gamma \in \mathbb{C}$.

We are  thus in the case where  $\mathcal{M}$ is  the Euclidean space $\mathcal{M}= \mathbb{R}^2 \sim \mathbb{C}$,  there is no first order terms so that  the
 tensor $B$ is zero here
and the Lagrangian submanifold is the unit circle, $\mathcal{L}= \mathbb{S}^1$.
The potential $V$ is given by
$$ V(\Gamma)=  {1 \over 4} ( 1 - |\Gamma|^2 )^2.$$
Writing $\Gamma = p+ n$  with $p \in \mathbb{S}^1$ and $n \in N_{p} \mathbb{S}^1$, we get
$$ V(p+ n)=   (p\cdot n)^2 +  p\cdot n |n|^2 +  {1 \over 4 }|n|^4$$
and since in this simple situation $p= n/ |n|$, this can also be written
$$ V(p+n)=   |n|^2  +   p\cdot n |n|^2 + {1 \over 4 } |n|^4$$
which is under the form  \eqref{Vdefbis} with
$$
\left\{ 
\begin{array}{l}
\Lambda= c = 1 \\
V_{1}(p)(n,n,n)=   p\cdot n |n|^2 \\
F_1(p)(n_1,n_2) = (p \cdot n_1) n_2 + (p \cdot n_2) n_1 + (n_1 \cdot n_2) p.
\end{array}
\right.
$$
Moreover, the second fundamental form of $\mathcal{L}= \mathbb{S}^1$ is given by
$$ \deux^\perp_{p}(X,Y) = - (X\cdot Y)p.$$
We use  the KdV scaling \eqref{KdVscaling} with base point $1$ so that $\Phi(\epsilon \phi)= e^{i \epsilon \phi}$. We can then write 
$A \in T_{1} \mathbb{S}^1$ as $A= i \rho$, with $\rho \in \mathbb{R}$. This gives
$$ \deux^\perp_{1}(A, \partial_{x} A)= - \rho \partial_{x} \rho, \quad \mbox{and} \quad F_1 (1)(iA,i\partial_x A) = 3 \rho \partial_{x} \rho.$$
By Theorem \ref{theoKdV}, we thus  get that the long wave limit of the Gross-Pitaevskii equation \eqref{GP} is described by the KdV  equation
$$ 2  \partial_{t} \rho = {1 \over 4} \partial_{xxx} \rho  -  3   \rho  \partial_{x} \rho,$$
recovering the result of \cite{Chiron-Rousset}, \cite{Bethuel-Gravejat-Saut-Smets1}.

\subsubsection{Vector case}
We discuss here the case of two coupled Gross-Pitaevskii equations. It arises in Bose-Einstein condensates where two species are present~\cite{KFC} or in nonlinear optics~\cite{agrawal}; the particular case of two coupled equations is sometimes called the Manakov equations. In general, it reads 
$$
\partial_t \Gamma = i \left( \frac{1}{2} \Delta \Gamma - V'(\Gamma) \right),
$$
where $\Gamma$ takes values in $\mathcal{M} = \mathbb{C}^d$, and $V(\Gamma) = G(|\Gamma_1|,\dots,|\Gamma_d|)$, for $G$ a function $[0,+ \infty[^d \to \mathbb{R}$, hence
$$
V'(\Gamma) = \left( [\partial_{1} G] \frac{\Gamma_1}{|\Gamma|}, \dots , [ \partial_{d} G ] \frac{\Gamma_d}{|\Gamma|} \right).
$$
The associated Hamiltonian reads of course
$$
E(\Gamma) = \frac{1}{4} \int_{\mathbb{R}} |\partial_y \Gamma|^2 \,dx + \int_{\mathbb{R}} V(\Gamma)\,dx.
$$
Assume that $G$ has a minimum at the point $U^0 = (U^0_1,\dots,U^0_d) \in (0, + \infty)^d$, with $\operatorname{Hess} G_{|U^0} = 2\lambda \operatorname{Id}$. Then $V$ is minimal on the Lagrangian manifold $\mathcal{L} = \{ |\Gamma_1| = U^0_1,\dots,|\Gamma_d| = U^0_d \}$. Adopt now natural coordinates $(\phi_i,N_i)$ by decomposing $\Gamma_i = U^0_i e^{i \phi_i} + N_i$, where $\phi_i \in \mathbb{R}$, and $N_i$ is colinear to $e^{i \phi_i}$. In these coordinates, $V(\Gamma)$ admits the expansion
$$
V(\Gamma) = \lambda |N|^2 +F_1(N,N)\cdot N,
$$
for a function $F_1$ with the required symmetry.

It remains to describe the second fundamental form at the point $(U^0_1,\dots,U^0_d) \in \mathcal{L}$. The tangent space is $\sim i \mathbb{R} \times \dots \times i \mathbb{R}$. For points $r^1 = (i\rho_1^1,\dots,i\rho_d^1)$ and $r^2 = (i\rho_1^2,\dots,i\rho_d^2)$ in the tangent space, we have as previously
$$
\deux^{\perp}(r^1,r^2) = \left( -\rho_1^1 \rho_1^2,\dots,-\rho_d^1 \rho_d^2 \right).
$$
By using again the notation $A= (i\rho_{1}, \cdots, i\rho_{d})^t$,  we obtain
 $$ 2  c\partial_{t} \rho_{k}= {1 \over 4} \partial_{x}^3 \rho_{k} - { 3 \over 2 }  \rho_{k}  \partial_{x} \rho_{k} - 
 { 1 \over 2 \lambda} F_{1, 0}(\rho, \partial_{x} \rho)_{k}, \quad  k=1, \cdots d$$
  with $c= \sqrt{\lambda}$.

\subsection{The Landau-Lifshitz equations for ferromagnetic chains}

\subsubsection{General case} We quickly present the Landau-Lifshitz equation with only exchange and isotropy energies, referring to the textbook~\cite{HS} for more.
It describes, in the continuum approximation, the magnetic spin in a one-dimensional ferromagnetic chain and reads
\begin{equation}
\label{LL1} \partial_{t}\Gamma= \Gamma \times \big({1 \over 2}\partial_{xx}\Gamma -  V'(\Gamma) \big)\end{equation}
where $\Gamma \in \mathbb{S}^2$ is the spin vector. We identify $\mathbb{S}^2$ with the unit sphere, so that
$\Gamma=(\Gamma_1,\Gamma_2,\Gamma_3) \in \mathbb{S}^2$ if and only if $\Gamma_1^2 + \Gamma_2^2 + \Gamma_3^2 = 1$.
 Again, we have  $B=0$ and 
the conserved energy for this equation is $\int \left[ \frac{1}{2} |\partial_x \Gamma|^2 + V(\Gamma) \right] \,dx$ where
the term $\frac{1}{2} |\partial_x \Gamma|^2$ accounts for the exchange energy (molecular magnetic fields tend to align) and
$V(\Gamma)$ for the anisotropy energy (in a cristal, not all directions of the molecular magnetic fields have the same energy).

This equation fits in our general framework by setting $\mathcal{M}= \mathbb{S}^2$, $\mathcal{L} = \{ V=0 \}$  with the   complex structure defined as  $i(u)= u \times \cdot$.
For simplicity, we focus on the uniaxial case, in which the minimum of $V$ is reached on sets of the form $\{ \Gamma_3 = \gamma_0 \}$.
We will distinguish two models: one with $\gamma_0 = 0$ (``easy plane anisotropy''), and one with $\gamma_0 \neq 0$ (``easy cone anisotropy'').

\subsubsection{Easy plane anisotropy}
We assume here that
$$ V(\Gamma)=  K  \Gamma_{3}^2$$
so that the Lagrangian submanifold is $\mathcal{L}= \mathbb{S}^1 \times \{0\}.$
      
By using spherical coordinates, we have for $p \in \mathbb{S}^1$ and $n \in \mathbb{R}$
$$  \Psi(p,n)= \cos (n) \, p + \sin (n )\, e_{3}$$ 
and thus
$$ V(\Psi(p,n))=  { \alpha^2 } (\sin n )^2= K n^2 + O(n^4).$$
This is under the form \eqref{Vdefbis} with  $c= K$ and $V_{1}= 0$.
In  the KdV regime \eqref{KdVscaling}, we can take $(1, 0, 0)^t$ as our reference point  and set
$p= \Phi(\epsilon \phi)= (\cos \epsilon \phi, \sin \epsilon \phi, 0)^t$.
We note that  the  second fundamental form of  the tangent  bundle of $\mathbb{S}^1 \times\{0\}$ as a submanifold of  $\mathbb{S}^2$ vanishes,
$ \deux^\perp = 0$. Consequently, we get from Theorem \ref{theoKdV} that the asymptotic regime is described by  the linear  KdV equation
(Airy equation)
$$ 2 \partial_{t} A = {1 \over 4} \partial_{x}^3 A.$$
A different scaling allowing larger data on shorter times  has been studied in \cite{Chiron} in order to get nonlinear effects in the limit.

\subsubsection{Easy cone anisotropy}
We assume here that
$$
V(\Gamma) =  V(\Gamma_3)
$$
is nonnegative, equal to zero for $\Gamma_3 = \gamma_0 \in [0,1]$, and admits the following expansion
$$
V(\gamma_0 + s) = \alpha s^2 + \beta s^3 + O(s^4).
$$
Define $\theta_0$ by $\cos \theta_0 = \gamma_0$. For $p \in \mathcal{L} = \{ \Gamma_3 = \gamma_0 \} \cap \mathbb{S}^2 $, define $N_0$ to be the unit vector in $N_p \mathcal{L}$ such that $N_0 \cdot e_3 < 0$. For $s \in \mathbb{R}$, the map $\Psi$ is then given by
$$
\Psi(p,s N_0) = \sin(\theta_0 + s) q + \cos(\theta_0+s) e_3 \quad \mbox{with} \quad q = \frac{p - p\cdot e_3}{|p-p e_3|},
$$
and $V$ can then be expanded as 
$$
V( \Psi(p,s N_0) ) = V( \cos(\theta_0+s) ) = \underbrace{\alpha (\sin \theta_0)^2}_{c^2= \lambda} s^2 + \underbrace{(\alpha \sin \theta_0 \cos \theta_0 + \beta (\sin \theta_0)^3)}_{b} s^3 + O(s^4).
$$
This means that
$$
F_1 (N_1,N_2) = 3b (N_0 \cdot N_1) (N_0 \cdot N_2) N_0.
$$
On the other hand, a computation gives
$$
\deux^\perp (X,Y) = - \operatorname{cotan} \theta_0 (X\cdot Y) N_0.
$$
Therefore, Theorem \ref{theoKdV} gives the following equation in the long-wave limit:
$$
2 c \partial_t A = \frac{1}{4} \partial_x^3 A + \left( {3 \over 2 }  \operatorname{cotan} \theta_0 + \frac{3b}{2 \lambda} \right) A \partial_x A.
$$

\subsection{The Landau-Lifshitz equations for anti-ferromagnetic chains}
The continuum limit for antiferromagnetic chains is described by the Landau-Lifshitz system
\begin{equation}
\label{LL2}
\left\{
\begin{aligned}
&  \partial_{t} u=  u\times \big( -  \frac{1}{2} \partial_{xx} u  - \partial_{x} v  +  2 v\big), \\
& \partial_{t} v = v \times \big( -\frac{1}{2} \partial_{xx} v + \partial_{x} u  + 2  u \big)
\end{aligned}
\right.
\end{equation}
where the two unknowns $u$ and $v$ take values in $\mathbb{S}^2$ (for the derivation of this equation, see~\cite{Papanicolaou}, equation (3.4) taking into account
that $A = - B+O(\epsilon)$).
This system also enters in our general framework  with $\Gamma= (u,v) \in \mathcal{M}= \mathbb{S}^2 \times \mathbb{S}^2$.
 The   complex structure is defined by $i(\Gamma) \cdot (X, Y)=( u \times X, v \times Y)^t$, the tensor $B$ is
$$ B(u,v)(X,Y)=  ( -P (u) Y, P(v) X), \quad  \forall (X, Y) \in T_{u}  S^2 \times T_{v} S^2 $$
where we denote  by $P(u)$ the orthogonal projection on the tangent space $T_{u} \mathbb{S}^2$ (thus $P(u) X=  X - u \cdot X \, u$, \, $\forall X \in \mathbb{R}^3$)
and the potential $V$ is given  by 
$$ V(\Gamma) =  |u+ v|^2$$
so that $\mathcal{L}$ is the anti diagonal $\{ (u,v) \in \mathcal{M}, \, u+ v = 0\}.$
Let us describe $\Psi(p,n)$. We can write $p \in \mathcal{L}$ under the form $p= (\omega, - \omega), $
$\omega \in S^2$  and a normal vector $n$   is under the form  $n= (X, X)$, $X \in T_{\omega}S^2.$
By choosing the axis of  coordinates, we can always consider that $\omega= (0, 0, 1)^t$
and  by using spherical coordinates that  $ X= \rho( \cos \phi, \sin \phi, 0)^t)$.  The geodesic on $S^2$
starting from $\omega$ with initial speed $X$ is thus given by
$$ \gamma (s) =  \left( \sin (\rho s) \cos(\phi), \sin(\rho s) \sin (\phi), \cos (\rho s) \right)^t.$$
We thus get that
$$ \Psi (p,n)=  (\psi(\omega, X),  -  \psi(\omega, -X))^t, \quad \psi(\omega, X)= \left( \sin (\rho) \cos(\phi), \sin(\rho) \sin (\phi), \cos \rho\right)^t.$$         
Therefore,   $V$ can be expressed as
$$ V(\Psi(p,n))  =   4 (\sin  \rho)^2 =  4 \rho^2 + O(\rho^4) = 2   |n|^2 + O(|n|^4).$$
Consequently, we  have $V_{1}= 0$ and $ \lambda=2.$

 We note that on $\mathcal{L}$, we have $B(u,-u) (X, Y)= ( - Y, X)$ and  thus $ \mu = 1$. This yields
 $$ c ^2= \lambda - \mu = 1.$$ 
In  the KdV regime \eqref{KdVscaling}, we can take 
$$p_{0}= \big((1, 0, 0)^t, -(1, 0, 0)^t \big)$$
as our reference point  on $\mathcal{L} $ for example and set
$ p= (\omega , - \omega),$ with $  \omega= \Phi(\epsilon \phi), $
$\Phi$ being  the exponential map on $S^2$ at the point  $(1,0, 0)^t$ (we do not need  the precise expression).
We also  note that $\mathcal{L}$ is a totally geodesic submanifold, therefore,  $\deux^\perp= 0$ on $\mathcal{L}$. Consequently, 
we obtain from Theorem \ref{theoKdV} that  the long wave limit is described by the linear Airy system
$$ 2 \partial_{t} A= {1 \over 4} \partial_{x}^3 A, \quad A \in \mathbb{R}^2.$$

\section{Preliminaries}

\label{sectionprelim}

\subsection{Geometry}         
Consider $\mathcal{M}$ a $2d$-dimensional K\"ahler manifold. We denote its metric by $(X,Y) \mapsto \langle X , Y \rangle$ or simply $X\cdot Y$, its Levi-Civita connection
by $\nabla$, its Riemann curvature tensor by $R$ and its complex structure by $i$. The compatibility of $i$ with the metric implies that
$$
\nabla i = 0 \quad \mbox{and} \quad \langle iX , iY \rangle = \langle X , Y \rangle \quad \mbox{for any $(X,Y)  \in T \mathcal M$}.
$$
We also consider a Lagrangian submanifold $\mathcal{L}$ of $\mathcal{M}$. We denote $T_p \mathcal{L}$, respectively $N_p \mathcal{L}$, for the tangent, 
respectively normal, spaces of $\mathcal{L}$ as a submanifold of $\mathcal{M}$ at $p \in \mathcal{L}$. $\mathcal{L}$ being Lagrangian means that for any $p \in \mathcal{L}$,
$$
i T_p \mathcal{L} = N_p \mathcal{L}.
$$

\subsubsection{Covariant derivatives}
\label{diffgeo}
 We adopt the following notations:
\begin{itemize}
\item $P^{\top}: T_p \mathcal{M} \rightarrow T_p \mathcal{L}$, respectively $P^{\perp}: T_p \mathcal{M} \rightarrow N_p \mathcal{L}$, is the orthogonal projector on the
tangent, respectively normal, space of $\mathcal{L}$.
\item The covariant derivative on the tangent bundle of $\mathcal{L}$ reads $\nabla^\top = P^\top \nabla$ 
\item The covariant derivative on the normal bundle of $\mathcal{L}$  reads $\nabla^\perp = P^\perp \nabla$.
\item We systematically abuse notations by not distinguishing between, say, the tangent space of $\mathcal{L}$ and its pull-back by a map. Assume for instance that 
$f: t \mapsto f(t)$ is a map 
from $\mathbb{R}$ to $\mathcal{L}$, and that $X$ is a section of the pullback of $T\mathcal{L}$ by $f$. In other words, $X$ associates to each $t$ in 
$\mathbb{R}$ an element $X(t)$ of $T_{f(t)} \mathcal{L}$. Denoting by  $\widetilde{X}$ a vector field such that
 $\tilde X (f(t))= X(t)$,  we will  write
$$
\nabla_t X_{|t_0} = \nabla_{\partial_t f (t_{0})} \widetilde{X}_{|f(t_0)}.
$$
\end{itemize}

\subsubsection{Differentiating tensors}
Consider a tensor mapping, say, $(T \mathcal{L})^2$ to $(N \mathcal{L})$. Then set, for $U,V,W$ sections in $T \mathcal{L}$,
\beq
\label{difftensor}
\left[ \nabla_U^\perp A \right] (V,W) \overset{def}{=} \nabla_U^\perp \left[A (V,W)\right] - A (\nabla^\top_U V,W) - A (V,\nabla^\top_U W).
\eeq
This definition can be extended in an obvious way to general tensors. It will be also useful to view the  covariant derivative of a tensor as a tensor
 with  covariant index  augmented by one. In the case of  the above example, this yields
 $$ (\nabla^\perp  A)(U,V,W) \overset{def}{=} \left[ \nabla_{U}^\perp A\right] (V,W).$$
   Again this can be extended in an obvious way to general tensors.

\subsubsection{Second fundamental forms}
The second fundamental form of $T \mathcal{L}$ is given by ($X$ and $Y$ being sections of $T\mathcal{L}$)
\beq
\label{deuxTdef}
\nabla_X Y = \nabla^\top_X Y + \deux^{\perp}(X,Y) \quad \mbox{or} \quad \deux^\perp(X,Y) = - \nabla P^\perp (X, Y).
\eeq
The second fundamental form of $N \mathcal{L}$ is given by ($X$ and $N$ being sections of $T \mathcal{L}$ and $N \mathcal{L}$ respectively)
\beq
\label{deuxperpdef}
\nabla_X N = \nabla^\perp_X N + \deux^{\top}(X,N) \quad \mbox{or} \quad \deux^\top(X,N) = - \nabla P^\top (X, N).
\eeq
\begin{prop}
\label{deuxsym}
Let $p \in \mathcal{L}$, $X, Y \in T_p\mathcal{L}$ and $N \in N_p \mathcal{L}$. Denote simply
$\deux^\top$ and $\deux^\perp$ for the second fundamental forms at $p$. Then
\begin{enumerate}
\item $\deux^\perp (X,Y) = \deux^\perp(Y,X)$.
\item $i \deux^\top(X,N) = \deux^\perp(X,iN)$.
\item $\deux^\top(\cdot,N)$ is symmetric on $T\mathcal{L}$ (for the metric scalar product).
\item $i \deux^\perp(\cdot,X)$ is symmetric on $T\mathcal{L}$ (for the metric scalar product).
\item $\deux^\perp(i\cdot,X)$ is symmetric on $N\mathcal{L}$ (for the metric scalar product).
\item $i\deux^\top (X,\cdot))$ is symmetric on $N\mathcal{L}$ (for the metric scalar product).
\end{enumerate}
\end{prop}
As a corollary of these classical properties, we also obtain that
\begin{cor}
\label{BdeuxT}
Assuming~\eqref{H2}, we have on $\mathcal{L}$
\begin{enumerate}
\item $iB \deux^\top(X,N)= \deux^\top(iBX, N), \quad \forall X \in T\mathcal{L}, \,  \forall N\in N\mathcal{L}$,
\item $\deux^\perp (X, iB Y)= iB \deux^\perp( X, Y), \quad \forall X,Y \in T \mathcal{L}$,
\item $iB i \deux^\perp \big( X, \cdot) \mbox{ is symmetric on } T \mathcal{L} \mbox{ for the metric scalar product } \forall X \in T \mathcal{L}$.
\end{enumerate}
\end{cor}
 \begin{proof}
 Let us start with the first identity.
  We have  thanks to the above properties and~\eqref{H2} that 
  \begin{align*}
    iB \deux^\top(X,N) &  =  -Bi \deux^\top(X,N)= B \deux^\perp(X,iN)= B \deux^\perp(iN, X)=  BP^\perp \nabla_{iN} X \\
  &  = - i P^\perp (iB) \nabla_{iN} X= - i P^\perp \nabla_{iN} (iB X) = -i \deux^\perp(iN, iBX) \\
  & = -i \deux^\perp(iBX, iN)=  \deux^\top(iBX, N).
  \end{align*} 
 For the second identity, it suffices to note that
 $$ \deux^\perp (X, iB Y)= P^\perp \nabla _{X} (iB Y) = P^\perp (iB \nabla_{X} Y)= iB P^\perp (\nabla_{X} Y)= iB \deux^\perp(X, Y).$$
 For the last property, it suffices to combine (4) of Proposition \ref{deuxsym}, the previous identity and~\eqref{H2}.
 \end{proof}
Second fundamental forms can be differentiated as tensors. For instance, if $U,V,W \in T\mathcal L$,
\begin{equation}
\begin{split}
\label{difftensor2}
\nabla_U \left[ \deux^\perp ( V , W ) \right] = & \left[ \nabla_U^\perp \deux^\perp \right] (V,W) +  \deux^\perp \left(\nabla_U^\top V,W\right)  +  \deux^\perp \left(V,\nabla_U^\top W \right) \\
& \qquad + \deux^\top \left( U\,,\,\deux^\perp ( V , W )\right).
\end{split}
\end{equation}
where the first line gives the normal component, and the second the tangential one.

\subsubsection{Normal coordinates on $ \mathcal{L}$}
The coordinate system given by $\Phi=\operatorname{exp}_0 : T_0 \mathcal{L} \rightarrow \mathcal{L}$ is normal at 0. It is well-known that in this coordinate system
the Christoffel symbols vanish at 0. It can be expressed as
\begin{equation}
\label{nablatopDPhi}
\nabla^\top D \Phi_{|0} = 0.
\end{equation}

\subsubsection{Commuting covariant derivatives with vector fields}
For a coordinate system $(s,u)$, and $F(s,u)$ a function valued on $\mathcal{L}$, we get since the Levi-Civita connection is torsion free that
\beq
\label{comagain}
\nabla_s^\top \partial_u F = \nabla_u^\top \partial_s F.
\eeq

\subsubsection{Commuting covariant derivatives}
The tangent curvature tensor $R^\top$ is defined by
\beq
\label{comT}
R^\top (X,Y) Z \overset{def}{=} \nabla_X^\top \nabla_Y^\top Z - \nabla_Y^\top \nabla_X^\top Z - \nabla_{[X,Y]}^\top Z
\eeq
for $X,Y,Z$  sections of the tangent bundle.
It is given by the Gauss equation
$$
R^\top(X,Y)Z = P^\top R(X,Y) Z + \deux^\top(Y,\deux^\perp(X,Z)) - \deux^\top(X,\deux^\perp(Y,Z)).
$$
Similarly, the normal curvature tensor $R^\perp$ is defined by
\beq
\label{comperp}
R^\perp (X,Y) Z \overset{def}{=} \nabla_X^\perp \nabla_Y^\perp Z - \nabla_Y^\perp \nabla_X^\perp Z - \nabla_{[X,Y]}^\perp Z,
\eeq
if $X,Y$ are sections of the tangent bundle, and $Z$ is a section of the normal bundle.
It is given by
$$
R^\perp(X,Y)N =  P^\perp R(X,Y) N + \deux^\perp(Y,\deux^\top(X,N)) - \deux^\perp(X,\deux^\top(Y,N)).
$$

\subsection{Functional spaces}

Recall first the classical Sobolev spaces. For a map $F$ with values in $\mathbb{R}^N$, 
for any $s \in \mathbb{N}$, $H^s$ is given by its norm
$$
\left\| F \right\|_{H^s}^2 \overset{def}{=} \sum_{m \leq s} \left\| \partial_x^m F \right\|_{L^2}^2.
$$
For vector fields $v \in u^{-1} T \mathcal{M}$ , this definition is still valid since we can always assume that $\mathcal{M}$ is embedded in $\mathbb{R}^N$
as a Riemannian manifold. Nevertheless, it will be more convenient for us to use covariant derivatives in the definition:
$$ \left\| v \right\|_{H^s}^2 \overset{def}{=} \sum_{m\leq s} \left\| \nabla_x^m v \right\|_{L^2}^2.$$
The two definitions coincide if $s>\frac{1}{2}$ and $\nabla u$ is at least as smooth as $v$, which will be the case for us below.
A first variant  which will be needed is $H^1_\epsilon$ whose norm reads
\begin{equation}
\label{H1epsilon}
\left\| v \right\|_{H^1_\epsilon}^2 \overset{def}{=} \|v\|_{L^2}^2 + \epsilon \| \nabla_x v \|_{L^2}^2.
\end{equation}
Next, we want to define Sobolev spaces which are anisotropic in space and time.
First, let us set up our notation for multiindices: if $m=(m_0,m_1) \in \mathbb{N}^2$, define
$$
\partial^m \overset{def}{=} \left( \epsilon^2 \partial_t \right)^{m_0} \left( \partial_x \right)^{m_1} \quad \mbox{,} \quad 
\nabla^m \overset{def}{=} \left( \epsilon^2 \nabla_t \right)^{m_0} \left( \nabla_x \right)^{m_1} 
$$
and in a similar way
$$
\left( \nabla^\top \right)^m \overset{def}{=} \left( \epsilon^2 \nabla_t^\top \right)^{m_0} \left( \nabla_x^\top \right)^{m_1} \quad \mbox{and} \quad
\left( \nabla^\perp \right)^m \overset{def}{=} \left( \epsilon^2 \nabla_t^\perp \right)^{m_0} \left( \nabla_x^\perp \right)^{m_1}.
$$
The length of $m$ is denoted $|m| = m_0 + m_1$. When we do not want to keep track of the exact nature of the derivatives involved, but simply of their number, 
we shall abuse notations by writing $\partial^{|m|}$ or $\nabla^{|m|}$  instead of $\partial^m$ or $\nabla^m$. For instance,
$$
\partial^2 f \quad \mbox{can denote} \quad \partial_x^2 f \quad , \quad (\epsilon^2 \partial_t)^2 f \quad \mbox{or} \quad \epsilon^2 \partial_t \partial_x f.
$$
For maps $u(t,x) \in \mathcal{M}$, we define an  anisotropic space-time semi norm  $\|\cdot \|_{\mathcal{H}^s}$,
which we also abbreviate $\|\cdot \|_s$:
$$
\left\| u(t) \right\|_{\mathcal{H}^s}^2 = \left\| u(t) \right\|_s^2 \overset{def}{=} \sum_{|m|\leq s} \left\| \nabla^m u(t,\cdot) \right\|_{L^2(\mathbb{R})}^2.
$$
Note that this  involves derivatives in time.

With  the above definition of $\partial^m$, we record the   following elementary product estimate in dimension $1$, which we will use
repeatedly:
\beq
\label{prod}
\| \partial^m v (t) \partial^{m'} w (t)\|_{L^2(\mathbb{R})} \leq C  \|v(t) \|_{k} \|w(t) \|_{k}, \quad |m|+|m'|\leq k, \, k \geq 1
\eeq
with $C$ independent of $\epsilon$ (again, recall that with our notation, $\partial^m$ depends on $\epsilon$ when it involves time derivatives).
The same estimate holds replacing $\partial^m$ by $\nabla^m$ for vector fields along $u^{-1} T \mathcal{M}.$

\subsection{Notations}

If $A$ and $B$ are two numbers, we denote
\begin{align*}
& A \lesssim B \quad \mbox{or} \quad A =O(B) \quad \mbox{if there exists $C>0$ independent of  $ \epsilon \in (0, 1]$ such that} \quad A \leq CB \\
& A \sim B \quad \mbox{if $A \lesssim B$ and $B \lesssim A$}
\end{align*}
(of course, the value of $C$ can change between occurences of $\lesssim$).
If $f$ is a function, $X$ a Banach space, and $B$ a number, we use the notation
$$
f = O_X(B) \quad \mbox{if} \quad \|f\|_X \lesssim B.
$$
For instance, $f = O_{L^2}(1)$ if, for some constant $C$ independent of $\epsilon$, $\|f\|_{L^2} \leq C$.

\section{The Euclidean case  $\mathcal{M}= \mathbb{R}^{2d}$}

\label{sectioneuclid}
In this section we prove theorems \ref{theounif} and \ref{theoKdV} in the 
case where $\mathcal{M}= \mathbb{R}^{2d}$ with the Euclidean metric.

In this simpler framework, the KdV scaling \eqref{KdVscaling} takes the form
\beq
\label{simplegeod}
u = p + \epsilon^2n, \quad p= \Phi(\epsilon \phi)
\eeq
where $\Phi$ is the exponential map at $0$ on $\mathcal{L}$.

Since $\nabla i= 0$, the tensor $i$ is  constant.  To simplify the exposition, {\bf we shall furthermore also assume in this section that  $B$  is a  constant tensor}. This implies that the properties stated in~\eqref{H2} hold for every $p \in \mathbb{R}^{2d}$. In geometric terms, this means that $B$ and $i$ are two anticommuting and parallel complex structures on $\mathbb{R}^{2d}$, which turns it into a hyperk\"ahler manifold\footnote{To make things a little more concrete, consider the case $d=2$, where $\mathbb{C}^2$ is viewed as $\mathbb{R}^4$ with the complex structure $\left( \begin{array}{cccc} 0 & 0 & -1 & 0 \\ 0 & 0 & 0 & -1 \\ 1 & 0 & 0 & 0 \\ 0 & 1 & 0 & 0 \end{array} \right)$. Then $B$ is skew symmetric and satisfies $Bi = -iB$ et $B^2 = -\mu I$ if and only if $B = \left( \begin{array}{cccc} 0 & \alpha & 0 & -\beta \\ -\alpha & 0 & -\beta & 0 \\ 0 & \beta & 0 & -\alpha \\ -\beta & 0 & \alpha & 0 \end{array} \right)$ with $\alpha^2 + \beta^2 = \mu$. }. The submanifold $\mathcal{L}$ is then assumed to be Lagrangian for both $i$ and $B$.

Next, recall that $V$ can be expanded as
$$
V(p+n) = \lambda |n|^2 + V_1(p)(n,n,n) + V_2(p,n) \qquad \mbox{if $p \in \mathcal L$ and $n \in N_p \mathcal L$},
$$
where $V_1$ is symmetrical and $V_2 (p,n) = O(|n|^4)$. It is easy to see that
$$
V'(p+n) = 2  \lambda  n + F_1(p)(n,n) + R^V(p,n),
$$
where $F_1$ was defined in~\eqref{F1def} and $R^V$ is at least cubic in $n$. In the scaling~\eqref{simplegeod} above, this becomes
\begin{equation}
\label{loriot}
V'(p + \epsilon^2n) = 2 \lambda \epsilon^2n + \epsilon^4 F_1(p)(n,n) + R^V(p,\epsilon^2 n).
\end{equation}

\subsection{Plan of the estimates} 
As already explained, we shall proceed in two main steps.
Set
\begin{equation}
\label{Esdef1}
\left\{ 
\begin{array}{l}
\mathcal{E}_{s, 1}(u,t) \overset{def}{=} \epsilon \left\|\phi \right\|_{s+1} + \epsilon^2 \left\|  n \right\|_{s+1} 
 +\epsilon^2 \left\| \partial_{xx} \phi \right\|_s 
+ \epsilon^3 \left\| \partial_{xx} n \right\|_s \\
\mathcal{E}_{s, 2}(u,t) \overset{def}{=} \left\| \partial_x \phi \right\|_s + \left\| n \right\|_s + \epsilon \left\| \partial_x n\right\|_s
\end{array}
\right.
\end{equation}
and
\begin{equation}
\label{Esdef2}
\mathcal{E}_s(u,t) \overset{def}{=}  \mathcal{E}_{s, 1}(u,t) + \mathcal{E}_{s, 2}(u,t). 
\end{equation}

Our plan is as follows:
\begin{itemize}
\item The a priori control of the energy $\mathcal{E}_{s,1}$ is obtained in Section~\ref{eotse} by working primarily on the Schr\"odinger equation.
\item In section~\ref{diffHS}, we will commute high order derivatives with the hydrodynamical system as a preparation to controlling $\mathcal{E}_{s,2}$.
\item The a priori control of the energy $\mathcal{E}_{s,2}$ is then obtained in Section~\ref{estHS} by working on the differentiated hydrodynamical system.
\item  The bootstrap argument is justified in section \ref{prooftheounif1} where the estimate of
 $\epsilon\| \phi \|_{L^\infty}$ is also performed.
\end{itemize}

For all the a priori estimates to come,  we shall work on an interval of time $[0, T^\epsilon]$ on which we assume that a solution $u$ exists and satisfies the uniform estimates
\beq
\label{hypapriori}
\sup_{[0,T^\epsilon]} \epsilon \|\phi \|_{L^\infty} + \epsilon^2 \|n\|_{L^\infty} \leq r \quad \mbox{and} \quad \sup_{[0,T^\epsilon]} \mathcal{E}_s(u) \leq R
\eeq
for two constants $r$ and $R$. We pick $r$ sufficiently small for the coordinate system \eqref{simplegeod} to be well defined (in view of the assumptions on the initial data in Theorem~\ref{theounif}, this assumption is verified initially as soon as $\epsilon$ is sufficiently small). As for $R$, it will be determined later.

Before going any further, we note the following lemma. Its proof uses the hydrodynamical system~\eqref{eqhydro}, which can be obtained by projecting the $u$ equation~\eqref{SMKdV} on $T\mathcal{L}$ and $N\mathcal{L}$.
\begin{lem}
On $[0, T^\epsilon]$, we have for $s \geq 2$,  the estimates 
\begin{align}
\label{dtphi}
& \epsilon^2 \| \partial_{t} \phi(t) \|_{s- 1} = O( \mathcal{E}_{s}(u,t)), \\
\label{dxxxphi}& \epsilon \| \partial_{x}^3 \phi(t) \|_{s-1} =  O(\mathcal{E}_{s}(u,t) ).
\end{align}
\end{lem}
\begin{proof}

Indeed, from the first  equation in \eqref{eqhydro}, we have
\begin{multline*} \epsilon^2 \partial_{t} \phi = c \partial_{x} \phi 
+ S_{1}^{-1} \Big(i \Big[ {1 \over 2}  \deux^\perp \big(S_{1} \epsilon \partial_{x} \phi, D\Phi \epsilon  \partial_{x} \phi \big) 
+ {\epsilon^2 \over 2 } \big( \nabla^\perp_{x})^2 n  + B S_{1}\partial_{x} \phi-  2 \lambda  n - \epsilon^2 F_{1}(p)(n,n) \\- { 1 \over \epsilon^2} P^\perp R^V (p,\epsilon^2n)\Big] \Big).
\end{multline*}
Consequently, by using the product  law  \eqref{prod} and  \eqref{difftensor}, we get that for $s \geq 2$
\beq    \|\epsilon^2 \partial_{t} \phi (t)\|_{s-1} =  O\big(  \|n\|_{s-1}+ \|\epsilon \partial \phi \|_{s- 2} + \|\partial_{x} (\phi, n) \|_{s-1} +  \epsilon^2 \|\partial_{xx}n \|_{s- 1} 
\big).\eeq
Note that the  dependence in $\|\epsilon \partial \phi \|_{s-2}$ is due to the fact that $ D\Phi$, $\deux^\perp$, and $F_1$ depend on $\epsilon \phi$.
In view of the definition of $\mathcal{E}_{s}$, this yields \eqref{dtphi}.
 
In a similar way, by writing the second equation of \eqref{eqhydro} under the form
\begin{multline*}
{1 \over 2} \partial_{xx} \phi  = S_1^{-1} \left[ - i \big( \epsilon^2 \nabla^\perp_{t}n - (c + iB) \nabla_{x}^\perp n  \big) 
- {1 \over 2} \epsilon^2  \deux^\top \big( D\Phi\partial_{x}\phi, \nabla_{x}^\perp n)
-  {1 \over 2} P^\top \big( (\nabla_{x}S^1)\partial_{x}\phi\big) \right. \\
\left.+ \frac{1}{\epsilon^3} P^\top R^V(p,\epsilon^2 n) \right],
\end{multline*}
we also get
$$ \epsilon   \| \partial_{xxx} \phi \|_{s-1} =  O( \mathcal{E}_{s}(u,t) ).$$
\end{proof}

\subsection{Estimates on the Schr\"odinger equation}

\label{eotse}

We shall first prove.
\begin{prop}
\label{propschroplat}
The following estimate holds if $t \in [0,T^\epsilon]$:
$$
\mathcal{E}_{s,1}^2(u,t) \lesssim \mathcal{E}_{s}^2(u,0) + \epsilon^2 O(\mathcal{E}_{s}^2(u,t)) + \int_0^t O(\mathcal{E}_s^2 (u,\tau))\,d\tau.
$$
\end{prop}
 To get this estimate, we shall use a cancellation in the singular terms that come from the underlying  wave maps structure in the spirit of  
~\cite{ShatahZeng}. 
\begin{proof} \underline{Step 1: initial decomposition.}
Taking into account the formula~\eqref{loriot} for $V'$, equation~\eqref{SMKdV} reads
$$
\left( \epsilon^2 \partial_t - (c + i B) \partial_x \right) u = i \left( \frac{1}{2} \epsilon \partial_{xx} u 
- 2 \lambda \epsilon n - \epsilon^3 F^1(p)(n,n) - \frac{1}{\epsilon} R^V(p,\epsilon^2 n) \right).
$$
For  a multiindex $m$ such that $|m| \leq s$, we shall apply  $(\epsilon^2 \partial_t - (c + Bi) \partial_x ) \partial^m$ to the above equation.
By using~\eqref{H2}, we note that  for the left hand side
 \begin{align*}
 ( \epsilon^2 \partial_{t} - (c +  B i) \partial_{x} \big)\big( \epsilon^2 \partial_{t} - (c + iB) \partial_{x} \big)
  & = (\epsilon^2 \partial_{t} - c \partial_{x})^2 - (Bi + iB)\partial_{x} \big( \epsilon^2 \partial_{t} - c \partial_{x}\big)  -B^2 \partial_{x}^2  \\
  & =  (\epsilon^2 \partial_{t} - c \partial_{x})^2  + \mu \partial_{x}^2
  \end{align*}
while for the right hand side
$$ \big(\epsilon^2 \partial_{t} - (c+ Bi) \partial_{x}\big)i= i \big(\epsilon^2 \partial_{t} - (c + iB) \partial_{x} \big).$$
This yields    
\begin{equation*}
\begin{split}
(\epsilon^2 \partial_t - c  \partial_x)^2 \partial^m u + \mu \partial_{x}^2  \partial^m u & = -\frac{1}{4} \epsilon^2 \partial_{xx}^2 \partial^m u + \frac{1}{2} \epsilon \partial_{xx} \partial^m \left[ 2 \lambda \epsilon n + \epsilon^3 F^1(p)(n,n) + \frac{1}{\epsilon} R^V(p,\epsilon^2 n) \right] \\
& \qquad \quad - i \big(\epsilon^2 \partial_t - (c + iB)\partial_x\big)  \partial^m \left[ 2 \lambda \epsilon n + \epsilon^3 F^1(p)(n,n) + \frac{1}{\epsilon} R^V(p,\epsilon^2 n) \right].
\end{split}
\end{equation*}
Taking the scalar product in $L^2(\mathbb{R}^{2d})$ against $\left( \partial_t - \frac{c}{\epsilon^2} \partial_x \right) \partial^m u$ gives
\begin{equation*}
\begin{split}
& \frac{d}{dt} \int \left[ \frac{1}{2} \left| \left( \epsilon^2 \partial_t - c  \partial_x \right) \partial^m u \right|^2 
+ \frac{1}{8} \epsilon^2 \left| \partial_{xx} \partial^m u \right|^2 - { \mu  \over 2} |\partial_{x} \partial^m u |^2 \right] \,dx \\
& \qquad \qquad \qquad
=\underbrace{ \frac{1}{2} \int \partial_{xx} \partial^m \left[ 2 \lambda n + \epsilon^2 F^1(p)(n,n) \right] \cdot \left(\epsilon^2 \partial_t -  c \partial_x \right) \partial^m u\,dx}_I \\ 
& \qquad \qquad \qquad \quad - \underbrace{\frac{1}{\epsilon} \int i (\epsilon^2 \partial_t - (c + i B) \partial_x) \partial^m \left[ 2 \lambda n + \epsilon^2 F^1(p)(n,n) \right] 
\cdot (\epsilon^2 \partial_t - c \partial_x) \partial^m u\,dx}_{II} +O(\mathcal{E}_s^2)
\end{split}
\end{equation*}
(the $R^V$ terms above are easily seen to contribute $O(\mathcal{E}_s^2)$; also notice that the scalar product denoted by $\cdot$ above is simply that of $\mathbb{R}^{2d}$). 
Next we decompose further $I$ and $II$ by splitting each scalar product into its tangential and normal parts:
\begin{equation*}
 \begin{split}
I = & \underbrace{\int P^\top \partial_{xx} \partial^m \left[ 2 \lambda n + \epsilon^2 F^1(p)(n,n) \right] \cdot P^\top \left(\epsilon^2 \partial_t - c \partial_x \right) \partial^m u \,dx}_{Ia} \\
& \qquad + \underbrace{\int P^\perp \partial_{xx} \partial^m \left[ 2 \lambda n + \epsilon^2 F^1(p)(n,n) \right] \cdot P^\perp \left(\epsilon^2 \partial_t - c \partial_x \right) \partial^m u \,dx}_{Ib}
\end{split}
\end{equation*}
and
\begin{equation*}
\begin{split}
& II = \underbrace{ \frac{1}{\epsilon} \int P^\top \left(i \left(\epsilon^2 \partial_t - (c + iB) \partial_x \right) \partial^m \left[ 2 \lambda n + \epsilon^2 F^1(p)(n,n) \right] \right) \cdot 
P^\top \left(\left(\epsilon^2 \partial_t - c \partial_x \right) \partial^m u \right)}_{IIa} \\
& \qquad \qquad \qquad \qquad \underbrace{+ \frac{1}{\epsilon} \int P^\perp \left(i \left(\epsilon^2 \partial_t - (c + iB)  \partial_x \right) \partial^m \left[ 2 \lambda n + \epsilon^2 F^1(p)(n,n) \right] \right) \cdot 
P^\perp \left(\left(\epsilon^2 \partial_t - c \partial_x \right) \partial^m u \right)}_{IIb}.
\end{split}
\end{equation*}

\bigskip

\noindent \underline{Step 2: estimating $Ia$.}
Observe that
$$
\partial_{xx} \partial^m n - {\nabla_x^{\perp 2}} {\nabla^{\perp m}} n = \deux^{\top} \left( D \Phi \epsilon \partial_x^2 \partial^m \phi\,,\, n\right) + O_{L^2}(\mathcal{E}_s).
$$
Applying $P^\top$ gives
$$
P^\top \partial_{xx} \partial^m n = \deux^{\top} \left( D \Phi \epsilon \partial_x^2 \partial^m \phi\,,\, n\right) 
+ O_{L^2} (\mathcal{E}_s).
$$
Also notice that, since $F^1(p)(n,n)$ is valued in $N_p \mathcal{L}$,
$$
\epsilon^2 P^\top \partial_{xx} \partial^m F^1(p)(n,n) = \epsilon O_{L^2} (\mathcal{E}_s).
$$
Therefore,
\begin{align}
\label{pingouin1}
Ia = & \lambda \int \deux^{\top} \left( D \Phi \epsilon \partial_x^2 \partial^m \phi\,,\, n\right) 
\cdot \left(\epsilon^2 \partial_t - c \partial_x \right) \partial^m u\,dx  +O(\mathcal{E}_s^2).
\end{align}
Notice that
\begin{equation}
\begin{split}
& P^\top \left(\epsilon^2 \partial_t - c \partial_x \right) \partial^m u = P^\top \left( \epsilon^2 \partial_t - c \partial_x \right) D\Phi \epsilon \partial^m \phi  + \epsilon O_{L^2}(\mathcal{E}_s).
\end{split}
\end{equation}
Therefore~(\ref{pingouin1}) can be written
\begin{align*}
Ia & = \lambda \int \deux^{\top} \left( D \Phi \epsilon \partial_x^2 \partial^m \phi\,,\, n\right) \cdot \left(\epsilon^2 \partial_t - c \partial_x \right) 
D \Phi \epsilon \partial^m \phi \,dx +O(\mathcal{E}_s^2). 
\end{align*}
Finally, integrating by parts and relying on the symmetry of $\deux^\top(\cdot,n)$ (see Proposition \ref{deuxsym}) gives
\begin{equation*}
\begin{split}
Ia & = - \frac{1}{2} \int \lambda \epsilon^2 \left( \epsilon^2 \partial_t - c \partial_x \right) \left[ \deux^{\top} \left( D \Phi \partial_x \partial^m \phi\,,\, n\right)
\cdot D \Phi \partial_x \partial^m \phi \right] \,dx + O(\mathcal{E}_s^2) \\
& = -  \frac{\lambda \epsilon^4}{2} \frac{d}{dt}  \int \deux^\top \left( D\Phi \partial_x \partial^m \phi\,,\,n\right) \cdot D\Phi \partial_x \partial^m \phi \,dx
+ O(\mathcal{E}_s^2) .
\end{split}
\end{equation*}

\bigskip

\noindent \underline{Step 3: estimating $Ib$.} Start by noticing that
$$
P^\perp \left( \epsilon^2 \partial_t - c \partial_x \right) \partial^m u = P^\perp \left( \epsilon^2 \partial_t - c \partial_x \right) \partial^m \epsilon^2 n + R,
$$
where the remainder $R$ is such that $\partial_x R = \epsilon O_{L^2}(\mathcal{E}_s)$. Therefore $Ib$ can be split into
\begin{subequations}
\begin{align}
\label{petrel1}
Ib & = \frac{1}{2} \int P^\perp \partial_{xx} \partial^m 2\lambda n \cdot P^\perp \left( \epsilon^2 \partial_t - c \partial_x \right) \partial^m \epsilon^2 n \,dx \\
\label{petrel2}
& \quad + \frac{1}{2}\int P^\perp \partial_{xx} \partial^m \epsilon^2 F_1(p)(n,n) \cdot P^\perp \left( \epsilon^2 \partial_t - c \partial_x \right) \partial^m \epsilon^2 n \,dx \\
\label{petrel3}
& \quad +\frac{1}{2} \int P^\perp \partial_{xx} \partial^m (2 \lambda n + \epsilon^2 F_1(p)(n,n)) \cdot R\,dx.
\end{align}
\end{subequations}
In order to deal with~(\ref{petrel1}), it suffices to integrate by parts in $x$ as follows:
\begin{equation*}
\begin{split}
(\ref{petrel1}) & = - \lambda \epsilon^2 \int P^\perp \partial_x \partial^m n \cdot P^\perp \left( \epsilon^2 \partial_t - c \partial_x \right) \partial_x \partial^m n \,dx 
+ O(\mathcal{E}_s^2) \\
& = - \lambda \epsilon^2 \int \nabla_x^\perp {\nabla^{\perp m}} n \cdot \left( \epsilon^2 \partial_t - c \partial_x \right) \nabla_x^\perp {\nabla^{\perp m}} n \, dx 
+  O(\mathcal{E}_s^2) \\
& = - \frac{ \lambda  \epsilon^4}{2} \frac{d}{dt} \int \left| \nabla_x^\perp {\nabla^{\perp m}} n \right|^2 \,dx +  O(\mathcal{E}_s^2).
\end{split}
\end{equation*}
We next estimate~(\ref{petrel2}), integrating by parts in $x$ and relying on the symmetry properties of $F_1$:
\begin{equation*}
\begin{split}
(\ref{petrel2}) & = -  \int \epsilon^2 F_1(p)(\nabla^\perp_{x} \nabla^{\perp m} n,n) \cdot (\epsilon^2 \partial_t - c\partial_x) \nabla_x^\perp \nabla^{\perp m} \epsilon^2 n  \,dx + O(\mathcal{E}_s^2) \\
& = -  \frac{\epsilon^6}{2} \frac{d}{dt} \int F_1(p)(\nabla^\perp_{x} \nabla^{\perp m} n,n) \cdot \nabla_x^\perp \nabla^{\perp m} n  \,dx + O(\mathcal{E}_s^2).
\end{split}
\end{equation*}
Finally,~(\ref{petrel3}) can be estimated directly (after an integration by parts in $x$)
$$
|(\ref{petrel3})| \lesssim \epsilon \|\partial_x \partial^{|m|} (2 \lambda n + \epsilon^2 F_1(p)(n,n)) \|_{L^2} \|\frac{1}{\epsilon} \partial_x R \|_{L^2} = O(\mathcal{E}_s^2).
$$

\bigskip

\noindent \underline{Step 4: estimating $IIa$.} First observe that
\begin{equation}
\begin{split}
& P^\top \left( \epsilon^2 \partial_t - c  \partial_x \right) \partial^m u =  \left( \epsilon^2 \nabla_t^\top - c \nabla_x^\top \right) \nabla^{\top m} \Phi^\epsilon  + \operatorname{II}^{\top}(D\Phi \left( \epsilon^2 \partial_t - c \partial_x \right) \epsilon \partial^m \phi \,,\,\epsilon^2 n)
 + \epsilon^3 O_{L^2}(\mathcal{E}_s).
\end{split}
\end{equation}
Therefore,
\begin{subequations}
\begin{align}
\label{colibri1}
IIa & = \frac{2\lambda}{\epsilon} \int P^\top \left(i \left( \epsilon^2 \partial_t - (c + i B) \partial_x \right) \partial^m n \right) \cdot 
\left( \epsilon^2 \nabla_t^\top - c \nabla_x^\top \right) \nabla^{\top m} \Phi^\epsilon\,dx \\
\label{colibri1etdemi}
& \quad +  \epsilon \int P^\top \left(i \left( \epsilon^2 \partial_t - (c + i B) \partial_x \right) \partial^m F_1(p)(n,n) \right) \cdot 
\left( \epsilon^2 \nabla_t^\top - c \nabla_x^\top \right) \nabla^{\top m} \Phi^\epsilon \,dx \\
\label{colibri2}
& \quad + 2\lambda \epsilon^2 \int P^\top \left(i \left( \epsilon^2 \partial_t - (c + i B) \partial_x \right) \partial^m n \right) \cdot 
\operatorname{II}^{\top}(D\Phi \left( \epsilon^2 \partial_t - c \partial_x \right) \partial^m \phi \,,\,n) \,dx \\
\label{colibri3}
& \quad + O(\mathcal{E}_s^2).
\end{align}
\end{subequations}
Next, substituting first covariant to flat derivatives, and then using the differentiated hydrodynamical system~(\ref{eqhydro}) in Proposition~\ref{propdiffHs} gives
\begin{equation}
\label{colibri01}
\begin{split}
P^\top i \left( \epsilon^2 \partial_t - (c  + i B)\partial_x \right) \partial^m n & 
= i \left( \epsilon^2 \nabla_t^\perp - (c + i B)\nabla_x^\perp \right) {\nabla^{\perp m}} n + \epsilon O_{L^2}(\mathcal{E}_s) \\
& = - \frac{1}{2} \nabla_x^{\top 2} \nabla^{\top m}\Phi^\epsilon - \frac{1}{2} \epsilon^2 \deux^\top \left(D \Phi( \partial_x^2 \partial^m \phi ) \, , \, n \right)
+ \epsilon O_{L^2} (\mathcal{E}_s).
\end{split}
\end{equation}
Replacing $P^\top i \left( \epsilon^2 \partial_t -( c + iB) \partial_x \right) \partial^m n$ by the above expression in~(\ref{colibri1}), and then integrating by parts
while keeping in mind that $\deux^\top (\cdot\,,\,n)$ is symmetric, yields
\begin{equation*}
\begin{split}
(\ref{colibri1}) & =  -\int \lambda  \nabla^{\top 2}_x \nabla^{\top m} \Phi^\epsilon \cdot \left( \epsilon^2 \partial_t - c \partial_x \right) \nabla^{\top m} \Phi^\epsilon   \,dx \\
& \quad - \int \lambda \epsilon^2 \deux^\top \left( D \Phi \partial_x^2 \partial^m \phi\,,\,n\right) 
\cdot \left( \epsilon^2 \partial_t - c \partial_x \right)  \nabla^{\top m} \Phi^\epsilon  \,dx 
+ O(\mathcal{E}_s^2) \\
& = \frac{d}{dt} \left[ \frac{\lambda}{2} \epsilon^2 \int \left|  \nabla^{\top}_x \nabla^{\top m}\Phi^\epsilon \right|^2 \,dx 
+ \frac{\lambda}{2} \epsilon^4 \int \deux^\top \left( D \Phi \partial_x \partial^m \phi\,,\,n\right) \cdot D\Phi \partial_x \partial^m \phi \,dx \right] + O(\mathcal{E}_s^2).
\end{split}
\end{equation*}
 To handle \eqref{colibri1etdemi}, we proceed in the same way, using in addition the symmetry properties of $F^1$. We first note that 
 \begin{align*}
(\ref{colibri1etdemi}) &= 2 \int \epsilon^2 i   F_1(p)(( \epsilon^2 \nabla_t^\perp - (c + iB) \nabla_x^\perp ) \nabla^{\perp m} n,n)  \cdot 
\left( \epsilon^2 \partial_t - c \partial_x \right) D\Phi \partial^m \phi \,dx  \\
& \quad +  \epsilon^2 \int   \Big(2  i F_{1}(p)( i B \nabla_{x}^\perp \nabla^{\perp m}n,n)   + 2B F_{1}(p)( \nabla_{x}^\perp \nabla^{\perp m}n,n) \Big) \cdot \left( \epsilon^2 \partial_t - c \partial_x \right) D \Phi \partial^m \phi  \,dx+ O(\mathcal{E}_s^2) \\
   & =  2 \int \epsilon^2 i   F_1(p)((\epsilon^2 \nabla_t^\perp - (c + iB) \nabla_x^\perp ) \nabla^{\perp m} n,n)  \cdot 
\left( \epsilon^2 \partial_t - c \partial_x \right) D\Phi \partial^m \phi \,dx + O(\mathcal{E}_s^2). 
   \end{align*}
    Indeed, we have  used that $\epsilon^2 \| \partial_{x} n \|_{m} \| \partial \phi \|_{m}$ is controlled by $\mathcal{E}_{s}$.
Then, by using again the hydrodynamical system, we find
\begin{equation*}
\begin{split}
(\ref{colibri1etdemi}) 
& = 2 \int \epsilon^2 i F_1(p) (\frac{i}{2} D\Phi \partial_x^2 \partial^m \phi ,n) \cdot \left( \epsilon^2 \partial_t - c \partial_x \right) D\Phi \partial^m \phi \,dx + O(\mathcal{E}_s^2)\\
& = \int \epsilon^2 F_1(p) (i D\Phi \partial_x \partial^m \phi ,n) \cdot i \left( \epsilon^2 \partial_t - c \partial_x \right) D\Phi \partial_x \partial^m \phi \,dx + O(\mathcal{E}_s^2)\\
& = \frac{d}{dt} \frac{\epsilon^4}{2} \int F_1(p)(i D\Phi \partial_x \partial^m \phi ,n)\cdot iD\Phi \partial_x \partial^m \phi \,dx + O(\mathcal{E}_s^2).
\end{split}
\end{equation*}
Once again replacing $P^\top i \left( \epsilon^2 \partial_t - (c + iB) \partial_x \right) \partial^m n$ by the above expression \eqref{colibri01} in~(\ref{colibri2}),  integrating by parts
and using the symmetry of $\deux^\top(\cdot, n)$,
 we find
\begin{equation*}
\begin{split}
(\ref{colibri2}) & = - \lambda \epsilon^2 \int D\Phi \partial_x^2 \partial^m \phi 
\cdot \deux^\top \left(D \Phi \left( \epsilon^2 \partial_t - c \partial_x \right) \partial^m \phi \,,\,n \right)\,dx + O(\mathcal{E}_s^2)\\
& = \frac{d}{dt} \frac{\lambda}{2} \epsilon^4 \int \deux^\top \left( D \Phi \partial_x \partial^m \phi\,,\,n\right) \cdot D\Phi \partial_x \partial^m \phi \,dx+ O(\mathcal{E}_s^2).
\end{split}
\end{equation*}

\bigskip

\noindent \underline{Step 5: Estimating $IIb$.} Start by noticing that
$$
P^\perp i \left( \epsilon^2 \partial_t - (c + i B) \partial_x \right) \partial^m n = i \deux^\top \left(D\Phi \epsilon (\epsilon^2 \partial_t) \partial^m \phi\,,\,n\right)
  + \epsilon O_{L^2}(\mathcal{E}_s).
$$
while
$$
P^\perp i \left( \epsilon^2 \partial_t - (c + i B) \partial_x \right) \partial^m \epsilon^2 F_1(p)(n,n) = \epsilon O_{L^2} (\mathcal{E}_s).
$$
Moreover, substituting first covariant to flat derivatives, and using the differentiated hydrodynamical equation derived in Proposition~\ref{propdiffHs},
\begin{equation*}
\begin{split}
P^\perp \left( (\epsilon^2 \partial_t - c \partial_{x}) \partial^m u\right) & = \left( \epsilon^2 \nabla^\perp_t -  (c + iB) \nabla_x^\perp \right) {\nabla^{\perp m}} \epsilon^2 n
 + iB \epsilon^2 \nabla_{x}^\perp \nabla^{\perp m } n + 
\epsilon O_{L^2}(\mathcal{E}_s) \\
& = \frac{i}{2} \epsilon^2 D \Phi \partial_x^2 \partial^m \phi   + \epsilon O_{L^2}(\mathcal{E}_s).
\end{split}
\end{equation*}
Therefore, we find 
\begin{equation*}
\begin{split}
IIb & =  \lambda \epsilon^2 \int \deux^\top \left( D\Phi (\epsilon^2 \partial_t) \partial^m \phi\,,\,n\right) 
\cdot  D \Phi \partial_x^2 \partial^m \phi \, dx +  O(\mathcal{E}_s^2)
\end{split}
\end{equation*}
and finally integrating by parts while using the
symmetry of $\deux^\top \left(\cdot\,,\,n\right)$ gives  
\begin{equation*}
\begin{split}
IIb
& = - \frac{\lambda  \epsilon^4}{2}  \frac{d}{dt} \int \deux^\top \left( D\Phi \partial_x \partial^m \phi\,,\,n\right)
\cdot  D \Phi \partial_x \partial^m \phi \, dx +  O(\mathcal{E}_s^2) .
\end{split}
\end{equation*}

\bigskip

\noindent \underline{Step 6: Conclusion.}
Gathering the results of steps 1 to 5 gives
\begin{equation}
\label{Empart1}
\begin{split}
\frac{d}{dt} E_m = O(\mathcal{E}_s^2)
\end{split}
\end{equation}
with (notice the cancellation between $Ia$ and $IIb$ which, however, is not needed for the estimates to close)
\begin{equation}
\begin{split}
 E_m \overset{def}{=} & \int \left[ \frac{1}{2} \left| (\epsilon^2 \partial_t - c \partial_x )\partial^m u \right|^2 
+ \frac{1}{8} \epsilon^2 \left| \partial_{xx} \partial^m u \right|^2  - {\mu \over 2 } | \partial_{x} \partial^m u |^2 
+ \frac{\lambda}{2} \epsilon^2 \left| \nabla^{\top}_x \nabla^{\top m}  \Phi^\epsilon \right|^2 + {\lambda \over 2 } \epsilon^4 \left| \nabla_x^\perp {\nabla^{\perp m}} n \right|^2 \right. \\
&  \qquad + \lambda \epsilon^4 \deux^\top (D\Phi \partial_x \partial^m \phi\,,\,n) \cdot D\Phi \partial_x \partial^m \phi 
+ \frac{1}{2} \epsilon^6 F_1(p)(\nabla_x^\perp \nabla^{\perp m} n,n) \cdot \nabla_x^\perp \nabla^{\perp m} n  \\
& \qquad \qquad \left. + \frac{1}{2} \epsilon^4 F_1(p)(iD\Phi \partial_x \partial^m \phi ,n) \cdot iD\Phi \partial_x \partial^m \phi \right]\,dx. 
\end{split}
\end{equation}
To conclude, it suffices to note that
$$
\sum_{|m| \leq s} E_m = \mathcal{E}^2_{s, 1} + \epsilon^2 O(\mathcal{E}_{s,2}^2).
$$
\end{proof}

\subsection{Differentiating the hydrodynamical system}

\label{diffHS}

In order to derive a priori bounds, we will need to differentiate the hydrodynamical form \eqref{eqhydro} of the  equation.
Recall that this system is obtained by using the decomposition \eqref{simplegeod} of $u$, the expression~\eqref{loriot} for $V'$,
and by noticing thanks to \eqref{deuxperpdef} that 
$$ \partial u= \partial p +\epsilon^2 \partial n=  \partial p +  \epsilon^2 \deux^\top (\partial p, n)  + \epsilon^2 \nabla^\perp n= 
S_{0}\partial p + \epsilon^2 \nabla^\perp n$$ where the first term is tangent to $\mathcal{L}$ and the second term is normal to $\mathcal{L}$.   
The system \eqref{eqhydro} results from projecting \eqref{SMKdV} on $N_p \mathcal L$ and $T_p \mathcal L$ respectively.

The following proposition gives the system which is solved by $(\nabla^{\top m} {\Phi \over \epsilon},\nabla^{\perp m} n)$ up to error terms which will not matter in the estimates. Note that we use the convention that when a covariant derivative hits a function it  coincides with the standard derivative so that
 $  {1 \over \epsilon }  \nabla^\top \Phi= D\Phi \partial \phi.$ For notational convenience we shall set
 $$ \nabla^{ \top m } \Phi^\epsilon = {1 \over \epsilon }  \nabla^{\top m} \Phi$$
  which is the natural order one object since we roughly  have  that 
  $ \nabla^{ \top m } \Phi^\epsilon  = D\Phi \cdot \partial^m \phi$ plus $\epsilon$ times lower order terms.
 
\begin{prop}
\label{propdiffHs}
For $ 1 \leq |m | \leq s$ and $s \geq 2$, we get the following system for $(\nabla^{\top m }\Phi^\epsilon, \nabla^{\perp m } n)$ if $t \in [0,T^\epsilon]$:
\begin{equation}
\label{eqhydrom}
\left\{ 
\begin{array}{l} 
\displaystyle 
\left(S_{0} \nabla_t^\top   -  {1 \over \epsilon^2}(c+iB)S_{0}\nabla_x^\top \right) \nabla^{\top m} \Phi^\epsilon
= i \left[ \frac{1}{2}  \deux^{\perp}(S_{0} \nabla_{x}^\top \nabla^{\top m} \Phi^\epsilon, D\Phi \partial_{x} \phi) 
+   \frac{1}{2}  \deux^{\perp}(S_1 \partial_x   \phi, \nabla_{x}^\top \nabla^{\top m} \Phi^\epsilon)  \right. \\
\displaystyle
\hspace{2cm}\left.+ \frac{1}{2}
(\nabla^{\perp}_x)^2  \nabla^{\perp m}n -  { 2 \lambda \over \epsilon^2} \nabla^{\perp m}n  - 2F_1(n, \nabla^{\perp  m}  n) - 2 \lambda i  \deux^\top\big(  i \, n, \nabla^{\perp m}n \big) \right]  + O_{H^1}(\mathcal{E}_{s}) \\
\displaystyle
\left(\nabla_t^\perp - {1 \over \epsilon^2} (c+iB) \nabla_{x}^\perp \right)  \nabla^{\perp m }n  =   i \left[  \deux^{\top} \left( 
D\Phi \partial_x \phi \, , \,\nabla_x^\perp  \nabla^{\perp m } n  \right) + { 1 \over 2} \deux^\top \left( \nabla_{x}^\top \nabla^{\top m} \Phi^\epsilon , \nabla_{x}^\perp n \right)  \right. \\
\displaystyle \hspace{6cm} \left. +\frac{1}{2 \epsilon^2} \nabla_x^\top \big(S_{0} \nabla_x^\top \nabla^{\top m} \Phi^\epsilon \big) \right] + O_{H^1_{\epsilon}}\big( \mathcal{E}_s\big)
\end{array}
\right.
\end{equation}
(recall that $H^1_\epsilon$ was defined in~\eqref{H1epsilon}).
\end{prop}
 
\begin{proof}
As a preliminary remark, we note that since $R^V(p,N)$ is $O(|N|^3)$, we get thanks to \eqref{prod}
that
$${1 \over \epsilon^4} \partial^m  \big(P^\top R^V(p, \epsilon^2 n) \big) = O_{H^1} (\mathcal{E}_{s}), \quad 
{1 \over \epsilon^5} \partial^m  \big(P^\perp R^V(p, \epsilon^2 n) \big) = O_{H^1_{\epsilon}} (\mathcal{E}_{s}) $$
so that the remainder from the potential does not contribute.

\bigskip

\noindent
\underline{Step 1: the left-hand side of~$(\ref{eqhydro})_1$.}
We start by applying the operator $\nabla^{\top m}$  for $1\leq |m|\leq s$ to the left-hand side
of the first line of \eqref{eqhydro}.

Observe that we can expand  expand  $\nabla^{\top m } \Phi^\epsilon$
\begin{equation}
\label{nablaPhim0} \nabla^{\top m } \Phi^\epsilon =  D\Phi \partial^m \phi + \sum \epsilon^{k+1} \star_{k} (\nabla^k D\Phi)_{|\epsilon \phi} \big( \partial^{\alpha_{1}} \phi, \cdots, \partial^{\alpha_{k+1}} \phi \big)
\end{equation}
where $\star_{k}$ are harmless coefficients and the sum is for $k \geq 1$  and for $1 \leq \alpha_{1}, \cdots, \alpha_{k+1} <m$.
 Consequently, by using \eqref{dtphi} and \eqref{prod}, we  get that for $1 \leq |m| \leq s$, we have
 \begin{equation}
 \label{nablaPhim1}
 \| \nabla^{\top m} \Phi^\epsilon \|_{L^2}= O(\mathcal{E}_s).
 \end{equation}
  In a similar way, by differentiating \eqref{nablaPhim0}  in $x$, we also obtain that
 \begin{equation}
 \label{nablaPhim2}
  \|  \nabla_{x}^\top \nabla^{\top m} \Phi^\epsilon - D\Phi \partial_{x} \partial^m \phi \|_{L^2}= \epsilon O(\mathcal{E}_s)
  \end{equation}
  and hence that
  \begin{equation}
 \label{nablaPhim3}
  \|  \nabla_{x}^\top \nabla^{\top m} \Phi^\epsilon \|_{L^2}=  O(\mathcal{E}_s)
  \end{equation}

 We note that
\begin{align*} \nabla^{\top m} \left(  S_{1} (  \partial_{t} \phi  - {c \over \epsilon^2} \partial_{x} \phi ) - { 1 \over \epsilon^2} iB S_{1} \partial_{x }\phi \right)
\nonumber & = {1 \over \epsilon} \nabla^{\top m} \left( S_{0} \left( \partial_{t} \Phi - {c \over \epsilon^2} \partial_{x}  \Phi - { 1 \over \epsilon^2} S_{0}^{-1} iB S_{0} \partial_{x} \Phi \right) \right)
\end{align*}
 and that thanks to Corollary \ref{BdeuxT}, we have
 \begin{equation}
 \label{commutateurB}
  iBS_{0} X  = S_{0} iB X, \quad \forall X \in T\mathcal{L}.
  \end{equation}
  Therefore, we obtain
\begin{align}
 \nonumber&  \nabla^{\top m} \left(  S_{1} (  \partial_{t} \phi  - {c \over \epsilon^2} \partial_{x} \phi ) - { 1 \over \epsilon^2} iB S_{1} \partial_{x }\phi \right)
  = {1 \over \epsilon} \nabla^{\top m} \left( S_{0} \left( \partial_{t} \Phi -  { 1 \over \epsilon^2}(c  + iB) \partial_{x}  \Phi \right) \right)\\
\nonumber & = {1 \over \epsilon }   S_{0} \nabla^{\top m} \left( \partial_{t} \Phi - { 1 \over \epsilon^2}(c  + iB) \partial_{x} \Phi \right)
+  \left[ \nabla^{\top m},  {S_{0} \over \epsilon } \right] \left( \partial_{t} \Phi - {1 \over \epsilon^2}(c + iB) \partial_{x} \Phi \right)\\
\label{Lphidef}  & \overset{def}{=}  L^\phi_{1}+ L^\phi_{2}.
\end{align}
To express $L^\phi_{1}$, let us  write following our conventions $\nabla^{\top m }= \nabla^{\top \tilde m } \nabla^\top $ with $|\tilde m | =|m|-1$
and $ \nabla^\top=\epsilon^2 \nabla_{t}^\top$ or $\nabla_{x}^\top$. First  note that
 since $\nabla i= \nabla B= 0$, and $\nabla$ is torsion-free, we have
$$ 
L^\phi_{1}=  S_{0}\nabla^{\top \tilde m} \left( \nabla_{t}^\top- {1 \over \epsilon^2}(c+iB) \nabla_{x}^\top \right)  { \partial \Phi \over \epsilon}.$$
Next, writing $\nabla^{\top \tilde m} = (\epsilon^2 \nabla_{t}^\top)^{\tilde m_{0}} (\nabla_{x}^\top)^{\tilde{m}_{1}}$,
we find  by using the formula \eqref{comT} that if  $\tilde m_{1}>0$,
\beq
\label{L1phi11} \nabla^{\top \tilde m}  \nabla_{t}^\top  
\left({ \partial  \Phi \over \epsilon }\right) = (\epsilon^2 \nabla_{t}^\top)^{\tilde m_{0}}
(\nabla_{x}^\top)^{\tilde{m}_{1}- 1} \nabla^\top_{t}  \nabla_{x}^\top \left(  {\partial \Phi  \over \epsilon }\right) +
\Big(  (\epsilon^2 \nabla_{t}^\top)^{\tilde m_{0}}
(\nabla_{x}^\top)^{\tilde{m}_{1}- 1}  \Big)\left( R^\top(\nabla^\top_{t} \Phi,  \nabla^\top_{x} \Phi) 
{\partial \Phi \over \epsilon} \right).\eeq
Since the last term can be written under the form 
$$   \Big(  (\epsilon^2 \nabla_{t}^\top)^{\tilde m_{0}} (\nabla_{x}^\top)^{\tilde m_{1}- 1}  \Big)\Big( R^\top( D\Phi \epsilon^2 \partial_{t} \phi,   D\Phi \partial_{x}\phi) 
D\Phi \partial \phi \Big),$$
we get by using  \eqref{difftensor} and  \eqref{prod} that for $s \geq 2$
\begin{equation*} \Big\| \Big(  (\epsilon^2 \nabla_{t}^\top)^{\tilde m_{0}} (\nabla_{x}^\top)^{\tilde m_{1}- 1}  \Big)\Big( R^\top( D\Phi \epsilon^2 \partial_{t} \phi,   D\Phi \partial_{x}\phi) 
D\Phi \partial \phi \Big) \Big\|_{H^1} \\
= O\big(  \|\partial_{x} \phi \|_{s} +  \| \epsilon^2 \partial_{t}\phi \|_{s- 1}  \big) 
\end{equation*}
and hence thanks to \eqref{dtphi} that:
$$ \Big\| \Big(  (\epsilon^2 \nabla_{t}^\top)^{\tilde m_{0}} (\nabla_{x}^\top)^{\tilde m_{1}- 1}  \Big)\Big( R^\top( D\Phi \epsilon^2 \partial_{t} \phi,   D\Phi \partial_{x}\phi) 
D\Phi \partial \phi \Big) \Big\|_{H^1}
= O(\mathcal{E}_s^2).$$
By using  again  repeatedly  \eqref{comT}, we can thus obtain from \eqref{L1phi11} that
$$  \nabla^{\top \tilde m}  \nabla_{t}^\top  
\big({ \partial  \Phi \over \epsilon }\big) =   \nabla^\top_{t}  \nabla^{\top \tilde m }  \big({ \partial \Phi \over \epsilon} \big) + O_{H^1}(\mathcal{E}_{s}).
$$
which yields 
\beq
\label{L1phidt2}
S_{0} \nabla^{\top \tilde m} \nabla^\top_{t} {\partial  \Phi \over \epsilon}= S_{0}   \nabla_{t}^\top \nabla^{\top m} \Phi^\epsilon   + O_{H^1}(\mathcal{E}_{s}).
\eeq
To study the second part of  $L_{1}^\phi$, we  can use the same arguments  as above (note that ${1 \over \epsilon^2} \partial_{x}$ and $\partial_{t}$ behaves in the same way, in our weighted spaces) and the fact that $\nabla (iB)= 0$ to also obtain
$$ S_{0} \nabla^{\top \tilde m}  {1 \over \epsilon^2}( c+ iB) \nabla^\top_{x}  {\partial  \Phi \over \epsilon}=
 { 1 \over \epsilon^2}  S_{0}(c+ i B) \nabla_{x}^\top  \nabla^{\top m } \Phi^\epsilon
  + O_{H^1}(\mathcal{E}_{s}).$$
Consequently, we have proven that
\beq
\label{L1phi}
L_{1}^\phi=  \big(S_{0} \nabla_{t}^\top - {1 \over \epsilon^2} (c+iB)S_{0} \nabla_{x}^\top  \big)  \nabla^{\top m }\Phi^\epsilon   + O_{H^1}(\mathcal{E}_{s}).
\eeq
It remains to study $ L_{2}^\phi$. By using the definition of $S_{0}$, we can reduce the estimate of this commutator to the estimate
of terms of the form
$$ 
(\nabla^{\top \alpha} \deux^\top) \Big(  \nabla^{\top \gamma} \big( D\Phi \epsilon^2 \partial_{t} \phi - (c+ iB) D\Phi \partial_{x} \phi \big), \nabla^{\perp \beta} n \Big)
$$
with $| \alpha|  + |\beta|+ |\gamma| \leq |m|$ and $|\gamma| < |m|$. By using again \eqref{prod}, we find
$$
(\nabla^{\top \alpha} \deux^\top) \Big(  \nabla^{\top \gamma} \big( D\Phi \epsilon^2 \partial_{t} \phi - (c + iB) D\Phi \partial_{x} \phi \big), \nabla^{\perp \beta} n\Big)= 
O_{H^1}(\mathcal{E}_{s})
$$
except in the case $ \beta = m$. Indeed, in this case the estimate would involve the norm $|\partial_{x} n |_{m}$
without the $\epsilon$ weight  which is in the definition of $\mathcal{E}_{s, 2}.$ Consequently, it remains to estimate
$$ \deux^\top\big( D\Phi \epsilon^2 \partial_{t} \phi - (c + iB) D\Phi \partial_{x} \phi,  \nabla^{\perp m} n \big).$$

In order to get a simpler expression for the first term above, we can use the first equation of \eqref{eqhydro}.
This yields 
\beq
\label{reducfin}  \deux^\top\big( D\Phi \epsilon^2 \partial_{t} \phi - (c + iB) D\Phi \partial_{x} \phi,  \nabla^{\perp m} n \big)
=  \deux^\top\big(   - i 2 \lambda n, \nabla^{\perp m} n\big)  + 
\deux^\top\big(    R, \nabla^{\perp m} n \big)\eeq
where $R$ is such that
$$
\|R\|_{W^{1, \infty}} =   \epsilon \,O \big( \|\partial_{x} \phi \|_{W^{1, \infty}} + \|n \|_{W^{1,\infty}} + \epsilon \|\partial_{x}^2 n \|_{W^{1, \infty}}
\big) = \epsilon\, O\big( \mathcal{E}_{s})
$$
for $s \geq 2$. Consequently, we find that
$$
\deux^\top\big( D\Phi \epsilon^2 \partial_{t} \phi - (c + iB) D\Phi \partial_{x} \phi, \nabla^{\perp m} n,  \big)
=   \deux^\top\big(  - 2 \lambda   i  n  , \nabla^{\perp m} n\big)   +  O_{H^1}\big( \mathcal{E}_{s}\big).
$$
In summary, we have proven that
\beq
\label{L2phi}
L_{2}^\phi=   \deux^\top\big( - 2 \lambda  i  n ,  \nabla^{\perp m} n  \big)   +  O_{H^1}\big( \mathcal{E}_{s}\big) 
\eeq
Consequently, thanks to \eqref{L1phi}, \eqref{L2phi} and \eqref{Lphidef}, we find 
\begin{multline}
\label{leftphi}   
\nabla^{\top m} \left(  S_{1} (  \partial_{t} \phi  - {c \over \epsilon^2} \partial_{x} \phi ) - { 1 \over \epsilon^2} iB S_{1} \partial_{x }\phi \right)
\\=  \left( S_{0}  \nabla_{t}^\top -  {1 \over \epsilon^2} (c+iB)S_{0} \nabla_{x}^\top \right) \nabla^{\top m} \Phi^\epsilon  +  \deux^\top\big(- 2 \lambda in, \nabla^{\perp m} n \big)  
 + O_{H^1}(\mathcal{E}_{s}).
\end{multline}

\bigskip

\noindent
\underline{Step 2:  the right-hand side of~$(\eqref{eqhydro})_1$.} 
We now differentiate the right-hand side of the first equation in \eqref{eqhydro}. At first, we easily get that
\begin{align}
\nonumber \nabla^{\top m} \left( i \deux^\perp\big( S_{1} \partial_{x} \phi, D\Phi \partial_{x} \phi  \big) \right)
 & =  \nabla^{\top m} \left( i \deux^\perp \big( S_{0}{ \partial_{x} \Phi \over \epsilon}, {\partial_{x} \Phi \over \epsilon} \big) \right) \\
& \label{rightphi1} =     i\Big(  \deux^\perp\big( S_{0} \nabla_{x}^\top \nabla^{\top m} \Phi^\epsilon, D\Phi \partial_{x}  \phi  \big) +    
\deux^\perp\big( S_{1} \partial_{x}  \phi, \nabla_{x}^\top \nabla^{\top m } \Phi^\epsilon  \big)  \Big)
+ O_{H^1}(\mathcal{E}^s)\end{align}
for $s \geq 2$. Next, we need to expand 
$$ \nabla^{\top m}  \left( {i \over 2} (\nabla_{x}^\perp)^2 n\right)= {i \over 2}  \nabla^{\perp m}  \Big( (\nabla_{x}^\perp)^2 n\Big) .$$
We first note that if $\nabla^{\perp m}= \nabla_{x}^{\perp m}$ there is no commutator. Consequently, denoting $m=(m_0,m_1)$, we only have to study the case $m_{0}\geq 1$. We thus write
\begin{multline*}  {i \over 2}  \nabla^{\perp m}  \Big( (\nabla_{x}^\perp)^2 n\Big)= 
{i \over 2} \left( \epsilon^2 \nabla_{t}^\perp \right)^{m_{0}} (\nabla_{x}^\perp)^2  \nabla_{x}^{\perp m_{1}} n 
\\=  {i \over 2} \left( \epsilon^2 \nabla_{t}^\perp \right)^{m_{0} - 1} \nabla_{x}^\perp \big( \epsilon ^2\nabla_{t}^\perp) \nabla_{x}^{\perp}  \nabla_{x}^{\perp m_{1}} n 
+ {i \over 2} \left( \epsilon^2 \nabla_{t}^\perp \right)^{ m_{0} - 1}  \Big(  \epsilon^2 R^\perp\big( \epsilon^2 D\Phi \partial_{t} \phi, D\Phi \partial_{x} \phi
\big) \nabla_{x}^{\perp}  \nabla_{x}^{\perp m_{1}}n \Big)
\end{multline*}
thanks to \eqref{comperp}. By using again \eqref{prod} and \eqref{dtphi}, we see that  the last term is
$O_{H^{1}}(\mathcal{E}_{s})$. By using repeatedly  the commutator identity \eqref{comperp}, we 
obtain
$$
{i \over 2}  \nabla^{\perp m}  \Big( (\nabla_{x}^\perp)^2 n\Big)= {i \over 2}\big( \nabla_{x}^{\perp})^2 \nabla^{\perp m} n + R_2,
$$
where $R_2$ is a sum of terms of the type
$$
\frac{i}{2} \nabla^{\perp \alpha} \left( \epsilon^2 R(\epsilon^2 D \Phi \partial_t \phi,D\Phi \partial_x \phi)
\nabla^{\perp \beta} n \right),
$$
with $|\alpha| + |\beta| \leq s$ and $\alpha_1 + \beta_1 \geq 1$. This implies
\beq
\label{rightphi2}   {i \over 2}  \nabla^{\perp m}  \Big( (\nabla_{x}^\perp)^2 n\Big)= {i \over 2}\big( \nabla_{x}^{\perp})^2 \nabla^{\perp m} n
+ O_{H^1}(\mathcal{E}_{s}).
\eeq
It is clear that
\begin{equation}
\label{rightphi2'}
\nabla^{\perp m} F_1(p)(n,n) = 2 F_1(p) (\nabla^{\perp m} n,n) + O_{H^1}(\mathcal{E}_s).
\end{equation}

Consequently, by combining \eqref{leftphi}, \eqref{rightphi1}, \eqref{rightphi2},  \eqref{rightphi2'} and the last line above, we obtain the first line of \eqref{eqhydrom}.

\bigskip

\noindent
\underline{Step 3: the left-hand side of~$(\ref{eqhydro})_2$.} By using \eqref{comperp}, we  get
$$ \nabla^{\perp m} \nabla_{t}^\perp  n= \nabla_{t}^\perp  \nabla^{\perp m} n + R_{3}$$
where the commutator $R_{3}$ is a sum of terms of the form
$$ \nabla^{\perp \alpha}\big(  R^\perp( \epsilon^2 D\Phi \partial_{t} \phi , D\Phi \partial_{x}\phi) \nabla^{\perp \beta} n \big)$$
with $|\alpha| + | \beta |  \leq |m|-1$. Consequently, by using  \eqref{prod} and \eqref{dtphi}
we find that
$$ R_{3}= O_{H^1_{\epsilon}} \big( \mathcal{E}_{s} \big)$$
(actually for this term, one could get a much better estimate).
The  estimate of the commutator 
$$ { 1 \over \epsilon^2} \Big[  \nabla^{\perp m} , \nabla_{x}^\perp \Big] n$$
is very similar to what we already described.  Since $\nabla (iB)= 0$, we thus get that
\beq
\label{Ln1}
\nabla^{\perp m} \left( \nabla_{t}^\perp -  { 1 \over \epsilon^2}(c+ iB) \nabla_x^\perp \right) n =\left( \nabla_{t}^\perp - {1 \over \epsilon^2}(c+ iB) \nabla_x^\perp \right) \nabla^{\perp m} n
+ O_{H^1_{\epsilon}} \big( \mathcal{E}_{s}\big).
\eeq

\bigskip

\noindent \underline{Step 4: the right-hand side of~$(\ref{eqhydro})_2$.}
For the right hand side of the second line of \eqref{eqhydro}, we first  study the differentiation of
$$ L_{1}^n \overset{def}{=} {i \over 2 \epsilon^2} \nabla^\top_{x}\big( S_{0} D\Phi \partial_{x} \phi \big).$$
As before, we note that $\nabla^{\top  m}$ commutes with $\nabla^\top_{x}$ if $m_{0}= 0$, hence we write
$$ \nabla^{\perp m } L_{1}^n =  {i \over 2 \epsilon^2}  \big( \epsilon^2 \nabla_{t}^\top \big)^{ m_{0} } \nabla_{x}^\top \nabla_{x}^{\top m_{1}}  \big (S_{0}
D\Phi \partial_{x} \phi \big) $$
and by using again \eqref{comT}, we have
$$
\nabla^{\perp m } L_{1}^n =  {i \over 2 \epsilon^2} \nabla_{x}^\top \nabla^{\top m}  \big (S_{0}
D\Phi \partial_{x} \phi \big)  +  O_{H^1_{\epsilon}}(\mathcal{E}_{s}).
$$
Indeed, the commutator involves  a sum of  terms of the form
  $$  { i\over 2}  \nabla^{\top \alpha} \Big( R^\top \big( \epsilon^2 D\Phi \partial_{t} \phi, D\Phi \partial_{x} \phi \big)  \nabla^{\top \beta}
   \big( S_{0} D\Phi \partial_{x} \phi\big)\Big)$$
    with $| \alpha |+| \beta | \leq |m|-1$ that can be again estimated by using \eqref{prod} and \eqref{difftensor}. Commuting now $\nabla^{\top m}$ with $S_0$ gives
    \beq
    \label{nablamL1n} 
\nabla^{\perp m } L_{1}^n =  {i \over 2 \epsilon^2} \nabla_{x}^\top   \Big (S_{0} \nabla^{\top m}
  \big( D\Phi \partial_{x} \phi \big) \Big)  +  {i \over 2 \epsilon^2} \nabla_{x}^\top \Big( \big[\nabla^{\top m}  , S_{0} \big]  \big( D\Phi \partial_{x} \phi \big) \Big)  +  O_{H^1_{\epsilon}}(\mathcal{E}_{s}).\eeq
To estimate the second term in the above right hand side, we can use the definition of $S_{0}$ and the derivation formula \eqref{difftensor2}, 
to see that we have to estimate a sum of terms of the form
$$
\nabla_{x}^\top \Big( \big(\nabla^{\top \alpha} \deux^\top \big) \big(  \nabla^{\top \beta}  \big( D \Phi \partial_{x} \phi),  \nabla^{\perp \gamma} n \big)  \Big)
$$
with $| \alpha | + | \beta | + | \gamma | =  |m| $ and $\beta \neq m$.  When $\gamma \neq m$, these terms are clearly estimated in $H_{1}^\epsilon$ by 
$$
\big(1+ O(\mathcal{E}_s^2)\big)\big( \|n \|_{H^1_{\epsilon}} + \|\partial_{x} \phi \|_{s} +  
\epsilon^2 \| \nabla_x^\perp n \|_s + \epsilon^3 \| \nabla_x^{\perp 2} n \|_s
+ \epsilon^2 \| \partial_{xx} \phi \|_s + \epsilon \| \partial_x^3 \phi \|_{s-1} \big)
$$
By using \eqref{dxxxphi} and the definition of $\mathcal{E}_{s}$(note that here  we also use $\mathcal{E}_{s, 1})$,  we thus  get that these terms are 
$O_{H^1_{\epsilon}}(\mathcal{E}_{s})$.  Consequently, since $n$ is a normal vector field, we  get
\beq
\label{L1n1} {i \over 2 \epsilon^2} \nabla_{x}^\top \Big( \big[\nabla^{\top m}  , S_{0} \big]  \big( D\Phi \partial_{x} \phi \big) \Big) 
= {i \over 2}  \deux^\top \big( D\Phi \partial_{x} \phi,  \nabla_{x}^\perp \nabla^{\perp m}  n\big)
+ O_{H^1_{\epsilon}} \big( \mathcal{E}_{s}\big)
\eeq
     
It remains to study the first term in the right-hand side of~(\ref{nablamL1n}). Let us observe that
$$
\nabla^{\top m} \big( D\Phi \partial_{x} \phi) = \nabla^{\top m } \left( { \partial_{x} \Phi \over \epsilon } \right)=
\nabla^{\top \tilde m} \left(  \nabla_{x}^\top  {\partial\Phi \over \epsilon} \right), \quad |\tilde m |= s- 1.$$
This yields by using again  \eqref{comT} that
\begin{equation}
\label{grive} {i \over 2 \epsilon^2} \nabla_{x}^\top   \Big (S_{0} \nabla^{\top m}
\big( D\Phi \partial_{x} \phi \big) \Big) =   {i \over 2 \epsilon^2} \nabla_{x}^\top   \Big (S_{0} \nabla_{x}^\top  \nabla^{\top m} { \Phi^\epsilon} \Big) + O_{H^1_\epsilon}(\mathcal{E}_s).
\end{equation}
By combining \eqref{nablamL1n}, \eqref{L1n1}, and \eqref{grive}, 
we thus get
      \beq
      \label{nablamL1nfin}
      \nabla^{\perp m } L_{1}^n =   {i \over 2 \epsilon^2} \nabla_{x}^\top \  \Big (S_{0} \nabla_{x}^\top  \nabla^{\top m } \Phi^\epsilon  \Big)  
       +  {i \over 2}  \deux^\top \big( D\Phi \partial_{x} \phi,  \nabla_{x}^\perp \nabla^{\perp m}  n\big)
 + O_{H^1_{\epsilon}} \big( \mathcal{E}_{s}\big).
 \eeq
 
By using similar arguments, we also obtain
 $${1 \over 2 } \nabla^{\perp m} \Big( i  \deux^\top  \big( D\Phi \partial_{x} \phi, \nabla_{x}^\perp n \big) \Big)
 = {1 \over 2} i  \Big( \deux^\top \big( \nabla_{x}^\top \nabla^{\top m } \Phi^\epsilon, \nabla_{x}^\perp n \big) + \deux^\top \big(  D\Phi \partial_{x} \phi, 
  \nabla_{x}^\perp \nabla^{\perp m} n \big) \Big) + O_{H^1_{\epsilon}} \big( \mathcal{E}_{s} \big).$$

By combining the last identity, \eqref{nablamL1n} and \eqref{Ln1},  we finally  get the second line  of \eqref{eqhydrom}. This ends
 the proof of Proposition \ref{propdiffHs}.
  
 \end{proof}

\subsection{Estimates on the hydrodynamical system}
\label{estHS}
Our main result in this section will be the following:

\begin{prop}
\label{esthydroplat}
The following estimate holds if $t \in [0,T^\epsilon]$:
$$
\mathcal{E}_{s,2}^2(u,t) \lesssim \mathcal{E}_{s}^2(u,0) + \epsilon^2 O( \mathcal{E}_{s}^2 (u,t)) + \int_0^t O(\mathcal{E}_{s}^2(u,\tau))\,d\tau.
$$
\end{prop}

\begin{proof} \underline{Step 1: the multiplier method, and splitting of the resulting terms.}

Observe first that
by using again \eqref{commutateurB},  the system ~(\ref{eqhydrom}) can be written equivalently
\begin{equation}
\label{eqhydrombis}
\left\{ 
\begin{array}{l} 
\displaystyle 
\left( \nabla_t^\top - \frac{c}{\epsilon^2}\nabla_x^\top \right) \nabla^{\top m } \Phi^\epsilon  = S_0^{-1} i \left[ \frac{1}{2}  \deux^{\perp}(S_{0} \nabla_{x}^\top \nabla^{\top m } \Phi^\epsilon, D\Phi \partial_{x} \phi) +   \frac{1}{2}  \deux^{\perp}(S_1 \partial_x   \phi,  \nabla_{x}^\top  \nabla^{\top m } \Phi^\epsilon)  \right. \\
\displaystyle
 \hspace{0.5cm}\left.+ \frac{1}{2}
(\nabla^{\perp}_x)^2  \nabla^{\perp m}n  + { 1 \over \epsilon^2} B S_{0} \nabla_{x}^{\top} \nabla^{\top m } \Phi^\epsilon- { 2 \lambda \over \epsilon^2} \nabla^{\perp m}n  - 2 F_1(n,\nabla^{\perp m}n)  - 2 \lambda i  \deux^\top\big(  i \, n, \nabla^{\perp m}n \big) \right]  + O_{H^1}(\mathcal{E}_{s}) \\
\displaystyle
\nabla_t^\perp  \nabla^{\perp m }n - \frac{c}{\epsilon^2} \nabla^{\perp}_x \nabla^{\perp m} n =   i \left[  \deux^{\top} \left( 
 D\Phi \partial_x \phi \, , \,\nabla_x^\perp  \nabla^{\perp m } n  \right) + { 1 \over 2} \deux^\top \left( \nabla_{x}^\top  \nabla^{\top m } \Phi^\epsilon, \nabla_{x}^\perp n \right)  \right. \\
 \displaystyle \hspace{6cm} \left. +\frac{1}{2 \epsilon^2} \nabla_x^\top \big(S_{0} \nabla_x^\top \nabla^{\top m } \Phi^\epsilon \big)  + {1 \over \epsilon^2} B \nabla_{x}^\perp \nabla^{\perp m} n\big) \right] + O_{H^1_{\epsilon}}\big( \mathcal{E}_s\big)
\end{array}
\right.
\end{equation}
Now take the scalar product of the first and second lines of the above system
with appropriate integrating factors:
\begin{multline*}
(\ref{eqhydrombis})_1 \cdot \left[ -\nabla_x^\top \left( S_0 S_0 \nabla_x^\top  \nabla^{\top m } \Phi^\epsilon \right)  - 2  S_{0} B \nabla_{x}^\perp \nabla^{\perp m} n
- 2 \epsilon^2 \deux^\top\big(D\Phi \partial_{x} \phi, \nabla_{x}^\perp  \nabla^{\perp m} n \big) \right]  \\+  (\ref{eqhydrombis})_2 \cdot 
\left[ -\epsilon^2 (\nabla_x^\perp)^2 \nabla^{\perp m} n  - 2 B S_{0} \nabla_{x}^\top  \nabla^{\top m } \Phi^\epsilon
  - 2 \epsilon^2  \deux^\perp \big( D\Phi \partial_{x} \phi,  \nabla_{x}^\top  \nabla^{\top m } \Phi^\epsilon \big) \right. \\
\left.+ 4\lambda \nabla^{\perp m} n  + 4 \epsilon^2 F_1 (n,\nabla^{\perp m}n)+ 4 \lambda  i \epsilon^2 \deux^\top(in\,,\,\nabla^{\perp m} n ) \right],
\end{multline*}
then integrate over $\mathbb{R}$. We rearrange the resulting terms by writing the above as
\begin{equation}
\label{dauphin}
I + II + III + IV = V + VI + O(\mathcal{E}_s^2)
\end{equation}
where: the term $I$ corresponds to  the contribution of the  left-hand side of $(\ref{eqhydrombis})_1$ with the symmetric part of  its  multiplier 
\begin{equation*}
\begin{split}
I &\overset{def}{=} \int \operatorname{LHS} (\ref{eqhydrombis})_1 \cdot \left[ - \nabla_x^\top \left( S_0 S_0 \nabla_x^\top  \nabla^{\top m } \Phi^\epsilon \right) \right] \,dx \\
& = \int \left[ \left( \nabla_t^\top - \frac{c}{\epsilon^2}\nabla_x^\top \right)  \nabla^{\top m } \Phi^\epsilon \right] \cdot \left[ -\nabla_x^\top \left( S_0 S_0 \nabla_x^\top \nabla^{\top m} \Phi^\epsilon \right) \right] \,dx;
\end{split}
\end{equation*}
the term $II$ corresponds to the  contribution of the left-hand side of $(\ref{eqhydrombis})_2$ with the symmetric part of its  multiplier 
\begin{equation*}
\begin{split}
 II & \overset{def}{=} \int \operatorname{LHS} (\ref{eqhydrombis})_2 \cdot \left[ -\epsilon^2 (\nabla_x^\perp)^2 \nabla^{\perp m} n + 4 \lambda \nabla^{\perp m} n  +  4 \epsilon^2 F_1 (n,\nabla^{\perp m} n) +  4 \lambda i \epsilon^2 \deux^\top(in\,,\,\nabla^{\perp m} n ) \right]\,dx \\
& = \int \left[ \nabla_t^\perp  \nabla^{\perp m }n - \frac{c}{\epsilon^2} \nabla^{\perp}_x \nabla^{\perp m} n \right] \\
& \hspace{2cm} \cdot \left[ -\epsilon^2 (\nabla_x^\perp)^2 \nabla^{\perp m} n + 4 \lambda  \nabla^{\perp m} n + 4 \epsilon^2 F_1 (n,\nabla^{\perp m}n) + 4 \lambda i \epsilon^2 \deux^\top(in\,,\,\nabla^{\perp m} n ) \right]\,dx;
\end{split}
\end{equation*}
the terms $III$ and $IV$  correspond to the contributions of the left-hand sides of  $(\ref{eqhydrombis})_1$ and  $(\ref{eqhydrombis})_2$ that involve  the first order
 terms of the multipliers 
\begin{multline*}
III\overset{def}{=} \int\left[ \left( \nabla_{t}^\top - { c \over \epsilon^2} \nabla_{x}^\top\right) \nabla^{\top m} \Phi^\epsilon \cdot \left(- 2 S_{0} B \nabla_{x}^\perp \nabla^{\perp m} n \right) \right.
     \\ \left.+  \left( \nabla_{t}^\perp - { c \over \epsilon^2} \nabla_{x}^\perp\right) \nabla_{x}^{\perp m } n  \cdot \left(- 2  B S_{0} \nabla_{x}^\top  \nabla^{\top m } \Phi^\epsilon  \ \right)      \right];
\end{multline*}
\begin{multline*}
IV \overset{def}{=} \int\left[ \left( \nabla_{t}^\top - { c \over \epsilon^2} \nabla_{x}^\top\right)  \nabla^{\top m } \Phi^\epsilon \cdot \left(- 2 \epsilon^2 \deux^\top\big(D\Phi \partial_{x} \phi, \nabla_{x}^\perp  \nabla^{\perp m} n \big)\right) \right.
     \\ \left.+  \left( \nabla_{t}^\perp - { c \over \epsilon^2} \nabla_{x}^\perp\right) \nabla^{\perp m } n  \cdot \left( - 2 \epsilon^2  \deux^\perp \big( D\Phi \partial_{x} \phi,  \nabla_{x}^\top \nabla^{\top m} \Phi^\epsilon\big)  \right)      \right];
\end{multline*}
the term $V$ contains the higher order terms coming from the right-hand sides of $(\ref{eqhydrombis})_1$ and $(\ref{eqhydrombis})_2$:
\begin{equation*}
\begin{split}
& V \overset{def}{=} \frac{1}{2} \int S_0^{-1} i \left[ (\nabla^{\perp}_x)^2  \nabla^{\perp m}n + { 2 \over \epsilon^2} B S_{0} \nabla_{x}^\top \nabla^{\top m} \Phi^\epsilon
 + 2 \deux^\perp\big(D\Phi \partial_{x} \phi,  \nabla_{x}^\top  \nabla^{\top m} \Phi^\epsilon \big) \right. \\
 & \left.  \hspace{2cm} - {4  \lambda\over \epsilon^2} \nabla^{\perp m}n  - 4 F_1(n, \nabla^{\perp m }n )- 4 \lambda i  \deux^\top\big(  i \, n, \nabla^{\perp m}n \big) \right] \\
& \hspace{3cm} \cdot \left[ -\nabla_x^\top \left( S_0 S_0 \nabla_x^\top  \nabla^{\top m} \Phi^\epsilon \right)  - 2 S_{0} B \nabla_{x}^\perp \nabla^{\perp m} n - 2 \epsilon^2 \deux^\top \big(
D\Phi \partial_{x} \phi, \nabla_{x}^\perp \nabla^{\perp m} n \big)  \right] \,dx \\
& + \frac{1}{2\epsilon^2} \int i \left[ \nabla_x^\top \big(S_{0} \nabla_x^\top  \nabla^{\top m} \Phi^\epsilon \big) +  { 2 } B \nabla_{x}^\perp \nabla^{\perp m } n  + 2 \epsilon^2 \deux^\top \big( D\Phi \partial_{x} \phi, 
\nabla_{x}^\perp \nabla^{\perp m }n \big)   \right] \\
& \hspace{1cm} \cdot \left[ -\epsilon^2 (\nabla_x^\perp)^2 \nabla^{\perp m} n  - { 2 } B S_{0} \nabla_{x}^\top \nabla^{\top m} \Phi^\epsilon  - 2 \epsilon^2 \deux^\perp \big( D\Phi \partial_{x} \phi,  \nabla_{x}^\top \nabla^{\top m} \Phi^\epsilon \big)  \right.\\
& \left. \hspace{2cm} + 4 \lambda \nabla^{\perp m} n + 4\epsilon^2 F_1(n, \nabla^{\perp m} n) +  4 \lambda i \epsilon^2 \deux^\top(in\,,\,\nabla^{\perp m} n ) \right]\,dx;
\end{split}
\end{equation*}
and finally the term $VI$ gathers the lower order terms from the right-hand sides of $(\ref{eqhydrombis})_1$ and $(\ref{eqhydrombis})_2$:
\begin{equation*}
\begin{split}
VI \overset{def}{=} & \frac{1}{2} \int S_0^{-1} i \left[ \deux^{\perp}( \epsilon^2 \deux^\top( \nabla^{\top m} \Phi^\epsilon, n), D\Phi \partial_{x} \phi) +  \deux^{\perp}( \epsilon^2 \deux^\top(D\Phi \partial_x   \phi, n),  \nabla_{x}^\top  \nabla^{\top m} \Phi^\epsilon )  \right] \\
 & \qquad \qquad  \cdot \left[ -\nabla_x^\top \left( S_0 S_0 \nabla_x^\top  \nabla^{\top m} \Phi^\epsilon \right)  - { 2 } S_{0} B \nabla_{x}^\perp \nabla^{\top m} n
 - 2 \epsilon^2 \deux^\top \big( D\Phi  \partial_{x} \phi, \nabla_{x}^\perp \nabla^{\perp m} n \big)
 \right]\,dx \\
& + \int  i \left[   { 1 \over 2} \deux^\top \left( \nabla_{x}^\top \nabla^{\top m} \Phi^\epsilon), \nabla_{x}^\perp n \right) \right] 
 \cdot \left[ -\epsilon^2 (\nabla_x^\perp)^2 \nabla^{\perp m} n  - 2  B S_{0}\nabla_{x}^\top \nabla^{\top m} \Phi^\epsilon \right.\\
 & \left.  - 2 \epsilon^2 \deux^\perp \big( D\Phi \partial_{x} \phi, \nabla_{x}^\top \nabla^{\top m} \Phi^\epsilon \big)
 + 4 \lambda \nabla^{\perp m} n + 4 \epsilon^2 F_1(n, \nabla^{\perp m} n) +  4 \lambda i \epsilon^2 \deux^\top(in\,,\,\nabla^{\perp m} n ) \right]\,dx \\
 \overset{def}{=} &VI_{1} + VI_{2}.
\end{split}
\end{equation*}
Note that  to gather the terms in $V$ and $VI$, we have used that $\deux^\perp$ is symmetric in its arguments ((1) in Proposition \ref{deuxsym}).
\bigskip

\noindent
\underline{Step 2: treating $I$.} Since $S_0$ is self-adjoint, the first term gives
\begin{equation}
\label{loriot1}
- \int \nabla_t^\top  \nabla^{\top m} \Phi^\epsilon \cdot \nabla_x^\top \left( S_0 S_0 \nabla_x^\top  \nabla^{\top m} \Phi^\epsilon\right)\,dx = \int S_0 \nabla_x^\top \nabla_t^\top  \nabla^{\top m} \Phi^\epsilon\cdot 
S_0 \nabla_x^\top  \nabla^{\top m} \Phi^\epsilon\,dx
\end{equation}
after an integration by parts. Applying first~(\ref{comT}), and keeping then in mind the definition of $S_0$,
we can commute covariant derivatives to obtain
$$
S_0 \nabla_x^\top \nabla_t^\top  \nabla^{\top m} \Phi^\epsilon
= S_0  \nabla_t^\top \nabla_x^\top  \nabla^{\top m} \Phi^\epsilon  + O_{H^1} (\mathcal{E}_s)
= \nabla_t^\top  \left(S_0 \nabla_x^\top   \nabla^{\top m} \Phi^\epsilon\right) + O_{L^2} (\mathcal{E}_s).
$$
Coming back to~(\ref{loriot1}), this gives
\begin{equation}
\label{macareux1}
- \int \nabla_t^\top  \nabla^{\top m} \Phi^\epsilon\cdot \nabla_x^\top \left( S_0 S_0 \nabla_x^\top \nabla^{\top m} \Phi^\epsilon\right)\,dx 
= \frac{d}{dt} \frac{1}{2} \int \left| S_0 \nabla_x^\top  \nabla^{\top m} \Phi^\epsilon\right|^2 \,dx + O(\mathcal{E}_s^2).
\end{equation}
Proceeding similarly, we obtain
\begin{equation}
\label{goeland}
\begin{split}
\int \frac{c}{\epsilon^2} \nabla_x^\top  \nabla^{\top m} \Phi^\epsilon \cdot \nabla_x^\top \left( S_0 S_0 \nabla_x^\top  \nabla^{\top m} \Phi^\epsilon \right)\,dx 
& = -\frac{c}{\epsilon^2}  \int S_0 \nabla_x^\top \nabla_x^\top \nabla^{\top m} \Phi^\epsilon \cdot  S_0 \nabla_x^\top  \nabla^{\top m} \Phi^\epsilon\,dx \\
& = -\frac{c}{\epsilon^2} \int \nabla_x^\top \left(S_0 \nabla_x^\top   \nabla^{\top m} \Phi^\epsilon\right) \cdot  S_0 \nabla_x^\top   \nabla^{\top m} \Phi^\epsilon\,dx + O(\mathcal{E}_s^2) \\
& =  O(\mathcal{E}_s^2).
\end{split}
\end{equation}
Putting together~(\ref{macareux1}) and~(\ref{goeland}) gives
\begin{equation}
\label{toucan1}
I = \frac{d}{dt} \frac{1}{2} \int \left| S_0 \nabla_x^\top   \nabla^{\top m} \Phi^\epsilon \right|^2 \,dx + O(\mathcal{E}_s^2).
\end{equation}

\bigskip

\noindent
\underline{Step 3: treating $II$.} We start with
$$
\int \nabla^\perp_t \nabla^{\perp m} n \cdot \left[ -\epsilon^2 (\nabla_x^\perp)^2 \nabla^{\perp m} n + 4 \lambda \nabla^{\perp m} n \right]\,dx 
= \frac{d}{dt} \frac{1}{2} \int \left[ \epsilon^2 | \nabla_x^\perp \nabla^{\perp m} n |^2 + 4  \lambda |\nabla^{\perp m} n|^2 \right] \,dx + O(\mathcal{E}_s^2)
$$
where we used~(\ref{comperp}) to estimate the term resulting from the commutation of $\nabla_x^\perp$ and $\nabla_t^\perp$.
Next, by making use of the symmetry of $i \deux^\top (in,\cdot)$ (Proposition~\ref{deuxsym}) and $F_1(n,\cdot)$, we obtain
$$
\int \nabla^\perp_t \nabla^{\perp m} n \cdot \left[ 4 \lambda i \epsilon^2 \deux^\top(in\,,\,\nabla^{\perp m} n ) \right] \,dx 
= \epsilon^2 \frac{d}{dt} \int  2 \lambda\nabla^{\perp m} n \cdot i \deux^\top(in\,,\,\nabla^{\perp m} n )\,dx + O(\mathcal{E}_s^2), 
$$
and
$$
\int \nabla^\perp_t \nabla^{\perp m} n \cdot  4 \epsilon^2 F_1 (n,\nabla^{\perp m}n)  \,dx 
= \epsilon^2 \frac{d}{dt} \int \nabla^{\perp m} n \cdot 2 F_1(n, \nabla^{\perp m} n)\,dx  + O(\mathcal{E}_s^2),
$$
Gathering the three previous equalities gives
\begin{equation*}
\begin{split}
& \int  \nabla^\perp_t \nabla^{\perp m} n \cdot \left[ -\epsilon^2 (\nabla_x^\perp)^2 \nabla^{\perp m} n + 4 \lambda \nabla^{\perp m} n + 4 \epsilon^2 F_1(n, \nabla^{\perp m}n )+  4\lambda i \epsilon^2 \deux^\top(in\,,\,\nabla^{\perp m} n ) \right] \,dx \\
& = \frac{d}{dt} \frac{1}{2} \int \left[ \epsilon^2 | \nabla_x^\perp \nabla^{\perp m} n |^2 + 4 \lambda |\nabla^{\perp m} n|^2 + 4 \epsilon^2 F_1 (n,\nabla^{\perp m}n) \cdot \nabla^{\perp m} n + 4 \lambda\epsilon^2 \nabla^{\perp m} n \cdot i \deux^\top(in\,,\,\nabla^{\perp m} n ) \right]  \,dx \\
& \hspace{3cm}  + O(\mathcal{E}_s^2).
\end{split}
\end{equation*}
Similarly, one can show that
\begin{align*}
& \int  \frac{1}{\epsilon^2} \nabla^\perp_x \nabla^{\perp m} n \cdot \left[  
-\epsilon^2 (\nabla_x^\perp)^2 \nabla^{\perp m} n + 4 \lambda \nabla^{\perp m} n + 4 \epsilon^2 F_1(n, \nabla^{\perp m}n) \cdot \nabla^{\perp m }n \right.\\
& \hspace{8cm} + \left. 4 \lambda  i \epsilon^2 \deux^\top(in\,,\,\nabla^{\perp m} n ) \right] \,dx = O(\mathcal{E}_s^2).
\end{align*}
Gathering the two previous equalities gives
\begin{multline}
\label{toucan2}
 II = \frac{d}{dt} \frac{1}{2} \int \left[ \epsilon^2 | \nabla_x^\perp \nabla^{\perp m} n |^2 + 4 \lambda |\nabla^{\perp m} n|^2 \right.  \\
 \left.+4 \epsilon^2 F_1 (n,\nabla^{\perp m}n) \cdot \nabla^{\perp m }n + 4 \lambda \epsilon^2 \nabla^{\perp m} n \cdot i \deux^\top(in\,,\,\nabla^{\perp m} n ) \right]  \,dx + O(\mathcal{E}_s^2).
\end{multline}

\bigskip

\noindent
\underline{Step 4: treating $III$.}
 We shall first compute
 $$ III_1=  - 2 \int  \left[ \nabla^\top_{t}  \nabla^{\top m} \Phi^\epsilon \cdot  S_{0}B \nabla_{x}^\perp \nabla^{\perp m } n  +  \nabla_{t}^\perp  \nabla^{\perp m} n  \cdot B S_{0} \nabla_{x}^\top \nabla^{\top m} \Phi^\epsilon \right] \, dx.$$
   We note that
   \begin{multline*}  \int    \nabla_{t}^\perp  \nabla^{\perp m} n  \cdot B S_{0} \nabla_{x}^\top  \nabla^{\top m} \Phi^\epsilon= {d \over dt} \left( \int \nabla^{\perp m} n \cdot   B S_{0}\nabla_{x}^\top  \nabla^{\top m} \Phi^\epsilon \, dx \right)
   \\- \int \nabla^{\perp m} n \cdot \nabla_{t}^\perp( B S_{0}\nabla_{x}^\top  \nabla^{\top m} \Phi^\epsilon)\, dx
   \end{multline*}
   To compute the last term, we note that thanks to~\eqref{H2} and \eqref{comT} we have
   \begin{align*}
     \int \nabla^{\perp m} n \cdot \nabla_{t}^\perp( BS_{0} \nabla_{x}^\top \nabla^{\top m} \Phi^\epsilon)\, dx & =  \int \nabla^{\perp m} n \cdot \nabla_{x}^\perp(B S_{0} \nabla_{t}^\top \nabla^{\top m} \Phi^\epsilon ) 
   \, dx + O(\mathcal{E}_s^2) \\
    & = -  \int \nabla_{x}^\perp \nabla^{\perp m} n \cdot B S_{0} \nabla_{t}^\top  \nabla^{\top m} \Phi^\epsilon\, dx
    + O(\mathcal{E}_s^2)  \\
    & =  \int  S_{0}B  \nabla_{x}^\perp \nabla^{\perp m} n \cdot \nabla_{t}^\top  \nabla^{\top m} \Phi^\epsilon\, dx + O(\mathcal{E}_s^2)
\end{align*}
since $S_{0}$ is symmetric and $B$ is skew symmetric.
 This yields 
$$ III_1=  - 2 {d \over dt} \int \nabla^{\perp m} n \cdot   BS_{0} \nabla_{x}^\top  \nabla^{\top m} \Phi^\epsilon\, dx  + O(\mathcal{E}_s^2).$$
  Since by using again the  skew symmetry of $B$ and the symmetry of $S_{0}$, we have that
  \begin{align*}
   III_2  & = { 2 c \over \epsilon^2} \int  \left[ \nabla^\top_{x}  \nabla^{\top m} \Phi^\epsilon \cdot S_{0} B \nabla_{x}^\perp \nabla^{\perp m } n  +  \nabla_{x}^\perp  \nabla^{\perp m} n  \cdot B
    S_{0} \nabla_{x}^\top \nabla^{\top m} \Phi^\epsilon \right] \, dx= 0
  \end{align*}
  we have thus proven that
    \begin{equation}
  \label{IIB}
  III=  - 2 {d \over dt} \int \nabla^{\perp m} n \cdot   B S_{0} \nabla_{x}^\top   \nabla^{\top m} \Phi^\epsilon \, dx 
       + O(\mathcal{E}_s^2).
  \end{equation}
  
\noindent
\underline{Step 5: treating $IV$.} 
Let us first remark  that for any  tangent vector field $X \in T \mathcal{L}$ and normal vector field $N \in N \mathcal{L}$, we have
\begin{equation}
\label{deuxB1}
  \deux^\top ( D\Phi \partial_{x} \phi, N ) \cdot X  +  N \cdot \deux^\perp( D\Phi \partial_{x} \phi, X)= 0.
  \end{equation}
  Indeed, by using Proposition \ref{deuxsym}, we can write
  \begin{align*}
   \deux^\top ( D\Phi \partial_{x} \phi, N ) \cdot X&= - i i \deux^\top(D\Phi \partial_{x} \phi, N) \cdot X =  - i \deux^\perp(D\Phi \partial_{x} \phi, iN) \cdot X \\
   & 
   =  -  iN \cdot i \deux^\perp  ( D\Phi \partial_{x} \phi, X) = - N \cdot \deux^\perp (D\Phi \partial_{x} \phi, X).
   \end{align*}
    This immediately yields that 
   $$ 
    \int \left[  \nabla_{x}^\top   \nabla^{\top m} \Phi^\epsilon \cdot  \deux^\top\big(D\Phi \partial_{x} \phi, \nabla_{x}^\perp  \nabla^{\perp m} n \big)
    +  \nabla_{x}^\perp \nabla^{\perp m } n  \cdot  \deux^\perp \big( D\Phi \partial_{x} \phi,  \nabla_{x}^\top  \nabla^{\top m} \Phi^\epsilon\big) \right] dx= 0.$$
    To handle the terms with time derivatives we can proceed as previously due to this skew symmetry property. We observe that
    \begin{align*}
    &  \int \nabla_{t}^\perp \nabla^{\perp m } n  \cdot  \left(  - 2 \epsilon^2 \deux^\perp \big( D\Phi \partial_{x} \phi, \nabla_{x}^\top   \nabla^{\top m} \Phi^\epsilon \big) \right) \, dx \\
   &   =  - {d \over dt} \int  2  \epsilon^2  \nabla^{\perp m} n \cdot   \deux^\perp \big( D\Phi \partial_{x} \phi, \nabla_{x}^\top   \nabla^{\top m} \Phi^\epsilon \big) \, dx
     \\
     & \qquad  \qquad +  2  \epsilon^2 \int\nabla^{\perp m } n \cdot  \deux^\perp \big(  D\Phi \partial_{x} \phi, \nabla_{t}^\top  \nabla_{x}^\top  \nabla^{\top m} \Phi^\epsilon \big) dx
     + O(\mathcal{E}_s^2)
     \end{align*}
Focusing on the last integral above, it can be transformed using successively~\eqref{deuxB1}, integrating by parts, and using the differentiated hydrodynamical equation~\eqref{eqhydrom} to replace $\nabla_{t}^\top    \nabla^{\top m} \Phi^\epsilon$, to give
\begin{align*}
& 2  \epsilon^2 \int\nabla^{\perp m } n \cdot  \deux^\perp \big(  D\Phi \partial_{x} \phi, \nabla_{t}^\top  \nabla_{x}^\top  \nabla^{\top m} \Phi^\epsilon \big) dx =-2  \epsilon^2 \int\nabla_{t}^\top  \nabla_{x}^\top  \nabla^{\top m} \Phi^\epsilon\cdot  \deux^\top \big(  D\Phi \partial_{x} \phi,  \nabla^{\perp m } n \big) dx \\
& \quad = 2  \epsilon^2 \int\nabla_{t}^\top    \nabla^{\top m} \Phi^\epsilon\cdot  \deux^\top \big(  D\Phi \partial_{x} \phi,  \nabla^{\perp m } \nabla_{x}^\perp n \big) dx + 2  \epsilon^2 \int\nabla_{t}^\top    \nabla^{\top m} \Phi^\epsilon\cdot  \deux^\top \big(  D\Phi \partial_{x}^2 \phi,  \nabla^{\perp m } n \big) dx+ O(\mathcal{E}_s^2)\\
& \quad =  2  \epsilon^2 \int\nabla_{t}^\top    \nabla^{\top m} \Phi^\epsilon\cdot  \deux^\top \big(  D\Phi \partial_{x} \phi,  \nabla^{\perp m } \nabla_{x}^\perp n \big) dx + O(\mathcal{E}_s^2).
\end{align*}
Going back to the definition of $IV$, and gathering the above equalities, it follows that
\begin{equation}
\label{deuxdeux}
IV =  -2 {d \over dt}  \int \epsilon^2  \nabla^{\perp m} n \cdot  \deux^\perp\big(D\Phi \partial_{x} \phi,  \nabla_{x}^\top  \nabla^{\top m} \Phi^\epsilon \big) \, dx+ O(\mathcal{E}_s^2).
\end{equation}

\noindent
\underline{Step 6: treating $V$.}   By using that $S_{0}^{-1}$ is symmetric, that $i$ is skew symmetric  and that in the first integral
$$ \nabla_{x}^\top\big( S_{0} S_{0} \nabla_{x}^\top  \nabla^{\top m} \Phi^\epsilon \big)= S_{0} \nabla_{x}^\top\big(
S_{0} \nabla_{x}^\top  \nabla^{\top m} \Phi^\epsilon \big) + [ \nabla_{x}^\top, S_{0} ] S_{0} \nabla_{x}^\top   \nabla^{\top m} \Phi^\epsilon,$$
we observe that  the two integrals cancel almost  exactly, simply leaving the above commutator term, as well as another lower order commutator. In other words, $V$ reduces to
\begin{align*}
V &  = \frac{1}{2} \int S_0^{-1} i \left[ (\nabla^{\perp}_x)^2  \nabla^{\perp m}n  + { 2 \over \epsilon^2}  B S_{0} \nabla_{x}^\top  \nabla^{\top m} \Phi^\epsilon- {4 \lambda \over \epsilon^2} \nabla^{\perp m}n
 + 2 \deux^\perp\big(D\Phi \partial_{x} \phi,  \nabla_{x}^\top   \nabla^{\top m} \Phi^\epsilon \big)  \right. \\
&  \qquad \qquad  \left. - 4 F_1(n, \nabla^{\perp m}n) - 4 \lambda i  \deux^\top\big(  i \, n, \nabla^{\perp m}n \big) \right] 
 \cdot \left[ \left[ S_0,\nabla_x^\top \right] \left( S_0 \nabla_x^\top  \nabla^{\top m} \Phi^\epsilon \right) \right] \,dx \\
 & \quad  + \frac{1}{2} \int S_0^{-1} i \left[ (\nabla^{\perp}_x)^2  \nabla^{\perp m}n  + { 2 \over \epsilon^2}  B S_{0} \nabla_{x}^\top \nabla^{\top m} \Phi^\epsilon- {4 \lambda \over \epsilon^2} \nabla^{\perp m}n + 2 \deux^\perp\big(D\Phi \partial_{x} \phi,  \nabla_{x}^\top  \nabla^{\top m} \Phi^\epsilon \big) 
 \right. \\
&  \qquad  \left. - 4 F_1(n, \nabla^{\perp m}n) - 4 \lambda i  \deux^\top\big(  i \, n, \nabla^{\perp m}n \big) \right] 
 \cdot  2 \epsilon^2 (S_{0} - Id) \deux^\top (D\Phi \partial_{x} \phi, \nabla_{x}^\perp \nabla^{\perp m } n )  \, dx \\
&    = V_{1} + V_{2}.
\end{align*}
Since
$$
\left[ S_0,\nabla_x^\top \right] = \left[ Id + \epsilon^2 (\deux^\top)(\cdot,n) , \nabla_x^\top \right]
= -\epsilon^2 \deux^\top (\cdot,\nabla_x^\top n) - \epsilon^3 ( \nabla_{D\Phi \partial_x \phi} \deux^\top) (\cdot,n),
$$
it is easy to see that $V_{1}$ can be further reduced to
\begin{equation*}
\begin{split}
V_{1} & =  - \frac{\epsilon^2}{2} \int S_0^{-1} i (\nabla^{\perp}_x)^2  \nabla^{\perp m}n \cdot \deux^\top  
\left(  S_0 \nabla_x^\top  \nabla^{\top m} \Phi^\epsilon,\nabla^\perp_x n \right)  \,dx + O(\mathcal{E}_s^2). \\
\end{split}
\end{equation*}
Finally,  by using again that  $S_0 = Id + \epsilon^2 \deux^\top(\cdot,n)$, this simplifies further to
\begin{equation*}
V_{1}= - \frac{\epsilon^2}{2} \int i (\nabla^{\perp}_x)^2  \nabla^{\perp m}n \cdot \deux^\top  
\left(  \nabla_x^\top \left(  \nabla^{\top m} \Phi^\epsilon \right),\nabla^\perp_x n \right)  \,dx + O(\mathcal{E}_s^2).
\end{equation*}
  The term $V_{2}$ is easy to handle. Again since $S_{0}- Id = O(\epsilon^2)$, we obtain that
  $$V_{2}= O(\mathcal{E}_s^2).$$
  Note that we use in a crucial way that the $\mathcal{E}_{s,1}$ part of the energy allows to control
   $ \epsilon^3\| \partial_{xx} n \|_{s}$.
   We have thus proven that
 \begin{equation}
\label{toucan3}
V = - \frac{\epsilon^2}{2} \int i (\nabla^{\perp}_x)^2  \nabla^{\perp m}n \cdot \deux^\top  
\left(  \nabla_x^\top  \nabla^{\top m} \Phi^\epsilon,\nabla^\perp_x n \right)  \,dx + O(\mathcal{E}_s^2).
\end{equation}
  
\bigskip

\noindent
\underline{Step 7: treating $VI$.} 
We start with the first integral in the definition of $VI$: using the definition of $S_0 = Id + \epsilon^2 \deux^\top(\cdot,n)$, and the symmetry of $i \deux^\perp(\cdot,D\Phi \partial_{x} \phi)$ (Proposition~\ref{deuxsym}),
\begin{equation}
\label{mesange1}
\begin{split}
VI_{1} = & \frac{1}{2} \int \left[ S_0^{-1} i \deux^{\perp}(S_1 \partial_x \partial^m  \phi, D\Phi \partial_{x} \phi) \right] 
\cdot \left[ -\nabla_x^\top \left( S_0 S_0 \nabla_x^\top \left( D\Phi \partial^m \phi \right) \right) \right]\,dx + O(\mathcal{E}_s^2)\\
&\qquad = \frac{1}{2} \int \left[ i \deux^{\perp}(D\Phi \partial_x \partial^m  \phi, D\Phi \partial_{x} \phi) \right] \cdot 
\left[ - (\nabla_x^\top)^2 \left( D\Phi \partial^m \phi \right) \right]\,dx + O(\mathcal{E}_s^2) \\
& \qquad  = \frac{1}{2} \int \left[ i \deux^{\perp}(D\Phi \partial_x \partial^m  \phi, D\Phi \partial_{x} \phi) \right] 
\cdot \left[ - \nabla_x^\top \left( D\Phi \partial_x \partial^m \phi \right) \right]\,dx + O(\mathcal{E}_s^2) \\
&  \qquad = O(\mathcal{E}_s^2).
\end{split}
\end{equation}
Finally,  it is easy to see that 
\begin{equation*}
VI_2 = \frac{1}{2} \int   \left[i \deux^\top \left(  \nabla_{x}^\top \nabla^{\top m } \Phi^\epsilon, \nabla_{x}^\perp n \right) \right]  \cdot \left( -\epsilon^2 (\nabla_x^\perp)^2 \nabla^{\perp m} n \right) \,dx
 + O(\mathcal{E}_s^2)
\end{equation*}
(notice that this last line and the expression we found for $III$ in \eqref{toucan3} will cancel to leave only $O(\mathcal{E}_s^2)$).
We have thus proven that 
\begin{equation}
\label{toucan4}
VI = \frac{1}{2} \int   \left[i \deux^\top \left(\nabla_x^\top  \nabla^{\top m } \Phi^\epsilon, \nabla_{x}^\perp n \right) \right] \cdot \left[ -\epsilon^2 (\nabla_x^\perp)^2 \nabla^{\perp m} n \right]\,dx + O(\mathcal{E}_s^2),
\end{equation}
so that
$$
V + VI =  O(\mathcal{E}_s^2).
$$
\bigskip

\noindent
\underline{Step 8: conclusion.} From the identity  \eqref{dauphin} and~(\ref{toucan1}), (\ref{toucan2}),  \eqref{IIB}, \eqref{deuxdeux}, (\ref{toucan3}) and (\ref{toucan4}) we deduce that
\begin{align*}
&\frac{d}{dt} \frac{1}{2} \int \left[ | S_0 \nabla_x^\top   \nabla^{\top m }  \Phi^\epsilon |^2 + \epsilon^2 | \nabla_x^\perp \nabla^{\perp m} n |^2 + 4\lambda |\nabla^{\perp m} n|^2
 - 4  \nabla^{\perp m } n  \cdot BS_{0} \nabla_{x}^\top  \nabla^{\top m } \Phi^\epsilon
 \right. \\ 
&\left. - 4  \epsilon^2  \nabla^{\perp m} n \cdot  \deux^\perp\big(D\Phi \partial_{x} \phi,  \nabla_{x}^\top  \nabla^{\top m} \Phi^\epsilon \big)+ 4 \epsilon^2 \nabla^{\perp m} n \cdot F_1 (n, \nabla^{\perp m }n) + 4 \lambda \epsilon^2 \nabla^{\perp m} n \cdot i \deux^\top(in\,,\,\nabla^{\perp m} n ) \right]  \,dx \\
& \qquad \qquad \qquad = O(\mathcal{E}_s^2)
\end{align*}
 for $1 \leq |m| \leq s$.
 Let us call $E^m$, $|m| \geq 1$ the integral in the left hand side.
  In the case $m=0$, we can get a similar estimate by using directly \eqref{eqhydro}. Note that since the structure of \eqref{eqhydro} is slightly different from the structure of the system in Proposition \ref{propdiffHs} for $|m| \geq 1$, we can proceed as previously by using a slightly different
   multiplier.  Let us set
   \begin{multline*}
E^0=  {1 \over 2 }\int \Bigl[   4 \lambda |n^\epsilon|^2 + | \epsilon^2 \nabla_{x}^\perp n^\epsilon |^2  + {4 \over 3} \epsilon^2 F_{1}(n^\epsilon,n^\epsilon) \cdot n^\epsilon
 + 
  |S_{0} D\Phi \partial_{x} \phi^\epsilon|^2 \Bigr. \\
 \Bigl.-   4 \ iB\left( D\Phi \partial_{x} \phi^\epsilon + { 1 \over 2} \epsilon^2 \deux^\top (D\phi \partial_{x} \phi^\epsilon, n^\epsilon) \right) \cdot in^\epsilon  \Bigr].
\end{multline*}
Using the same arguments as above, we can also prove that
$$ {d \over dt} E^0= O(\mathcal{E}_s^2).$$
(we shall perform a related more precise computation in the proof of Lemma \ref{lemP} below).
The conclusion follows  if we prove that  $ \tilde{\mathcal{E}}_{s, 2}^2 \sim \mathcal{E}_{s,2}^2$ with 
$$
 \tilde{\mathcal{E}}_{s, 2}^2 = \sum_{|m|\leq s} E^m.
$$
By using \eqref{nablaPhim1}, \eqref{nablaPhim2}, we first easily get that
\begin{align*}\tilde{\mathcal{E}}_{s, 2}(u,t)^2  & = \sum_{|m|\leq s} \frac{1}{2} \int \left[ | S_0 \nabla_x^\top  \nabla^{\top m} \Phi^\epsilon |^2 + \epsilon^2 | \nabla_x^\perp \nabla^{\perp m} n |^2 + 4 \lambda  |\nabla^{\perp m} n|^2    \right. \\
& \qquad  \qquad  \qquad  \left. - 4  \nabla^{\perp m } n  \cdot BS_{0} \nabla_{x}^\top  \nabla^{\top m } \Phi^\epsilon
   \right] \, dx 
 + \epsilon^2 O(\mathcal{E}_{s}(u,t)^2) \\
&  = O( \mathcal{E}_{s}^2).
 \end{align*}
 This also yields 
 \begin{multline*}
 \tilde{\mathcal{E}}_{s, 2}(u,t)^2  \geq
   \sum_{|m|\leq s} \frac{1}{2} \int \left[ | S_0 \nabla_x^\top   \nabla^{\top m } \Phi^\epsilon |^2 + \epsilon^2 | \nabla_x^\perp \nabla^{\perp m} n |^2 + 4 \lambda  |\nabla^{\perp m} n|^2   \right.
\\  \left.- 4    | \nabla^{\perp m } n  | \, |B| \, | S_{0} \nabla_{x}^\top  \nabla^{\top m } \Phi^\epsilon| \right] \, dx 
 - \epsilon^2 O(\mathcal{E}_{s}(u,t)^2). 
 \end{multline*}
Note that, by using~\eqref{H2} and \eqref{cdef}, we have that
$$ | B|^2= \mu < \lambda, $$
therefore  the quadratic form 
$$ Q(X_{1}, X_{2})=  X_{1}^2 + 4 \lambda X_{2}^2 - 4 |B|  X_{1} X_{2}$$
is positive definite. This yields
$$   \sum_{|m|\leq s} \frac{1}{2} \int \left[| S_0 \nabla_x^\top   \nabla^{\top m } \Phi^\epsilon |^2 + \epsilon^2 | \nabla_x^\perp \nabla^{\perp m} n |^2 + 4 \lambda  |\nabla^{\perp m} n|^2 
 \right] \, dx \lesssim  \tilde{\mathcal{E}}_{s, 2}(u,t)^2  +  \epsilon^2 O(\mathcal{E}_{s}(u,t)^2)
 $$
 and by using again\eqref{nablaPhim1}, \eqref{nablaPhim2}, we finally obtain
 $$ \mathcal{E}_{s,2}(u,t)^2 \lesssim  \tilde{\mathcal{E}}_{s, 2}(u,t)^2  +  \epsilon^2 O(\mathcal{E}_{s}(u,t)^2).$$
 \end{proof}
 
\subsection{Proof of Theorem \ref{theounif} in the case $\mathcal{M}= \mathbb{R}^{2d}$}
\label{prooftheounif1}

The local well-posedness of the Schr\"odinger system \eqref{SMKdV}, which is semi linear, can be deduced from the a priori estimates which have been established in the previous subsections.

It gives a non-empty time interval where there exists a unique solution of  \eqref{SMKdV} 
such that $\mathcal{E}_{s}(u,t) <+\infty$ and where the representation   $ u= \Phi(\epsilon \phi) + \epsilon^2n$ is well-defined. 
For constants $R$ and $r$, we can thus define a maximal existence time
$$ T^\epsilon \overset{def}{=} \sup \big\{ T>0, \, \sup_{[0,T]} \mathcal{E}_{s}(u) \leq R \; \mbox{and}\; \sup_{[0,T]} \epsilon \|\phi \|_{L^\infty} + \epsilon^2 \|n \|_{L^\infty} \leq r \big\}.$$
Choose $R$ to be a constant times the energy of the data, and $r$ such that the representation $(\epsilon \phi,\epsilon^2n) \mapsto \Phi(\epsilon \phi) + \epsilon^2n$
is a diffeomorphism if $\epsilon \|\phi \|_{L^\infty} + \epsilon^2 \|n \|_{L^\infty} \leq r$ and $r>2 c_{0}.$

We shall prove that $T^\epsilon$ is bounded from below by a positive time, uniformly in $\epsilon$.
Consider $T \leq T^\epsilon$. From the  a priori estimates of Proposition \ref{propschroplat} and Proposition \ref{esthydroplat}, 
we get the existence of $C_{0}$ independent of $\epsilon$  and $T$ that depends only on $\mathcal{E}_{s}(u,0)$ such that
\begin{equation}
\label{finalplat1} \sup_{[0, T]} \mathcal{E}_{s}(u,t) \leq C_{0} + (\epsilon + T) O(R) 
\end{equation}
Next, simply by Sobolev embedding,  
\begin{equation}
\sup_{[0,T]} \label{finalplat2} \epsilon^2 \|n(t) \|_{L^\infty} \lesssim \epsilon^2 \sup_{[0,T]} \mathcal{E}_{s}(u,t) \leq \epsilon^2 R.
\end{equation}
   
It remains to estimate $\epsilon \| \phi \|_{L^\infty}$. From the first line of \eqref{eqhydro}, we first obtain  by integrating in 
time  and by using the uniform control of $\mathcal{E}_{s}$ for $s \geq 2$ that
\begin{equation}
\label{phifinal1} |  \epsilon \phi(t,x) | \lesssim   \sup_{x} \int_{0}^T  {1 \over \epsilon }\left|  W (s, x )\right| \, ds 
+   T  O(r, R)+ c_{0}, \quad \forall t \in [0, T].
\end{equation}
where 
 \begin{equation}
      \label{Vdef} W\overset{def}{=}  (c+ iB)D\Phi \partial_{x} \phi  - 2 i \lambda n.
      \end{equation}
      We shall thus estimate the quantity involving $W$ in the right hand side of \eqref{phifinal1}.
     From the first line of  \eqref{eqhydrom} with $m=(0, 1)$ i.e. $\partial^m= \partial_{x}$, we have
        \begin{equation}
        \label{phifinal2}  \nabla_{t}^\top ( D\Phi \partial_{x} \phi)=   {1 \over \epsilon^2}\nabla_{x}^\top W  +  O_{L^2} (r,R)
        \end{equation}
In a similar way,  by using that from~\eqref{H2}, we have $ i(c+ iB)= (c-iB) i$, we obtain 
 from the second line of \eqref{eqhydro}
 that
 \begin{equation}
 \label{infinal} \nabla_{t}^\top (in)={1 \over \epsilon^2}  \nabla_{x}^\top \big( (c-iB) in - { 1 \over 2} D\Phi \partial_{x} \phi \big)
       + O_{L^2}(r, R).
       \end{equation}
       
Now by using~\eqref{H2} again, we observe that $(c-iB) (c+ iB)=  c^2 + \mu= \lambda$ and hence that
 $$ (c\pm iB)^{-1}=  {1 \over \lambda} (c\mp iB).$$
  This yields
  \begin{equation}
  \label{nfinalconv}  \nabla_{t}^\top (in)=  - {1 \over \epsilon^2} \nabla_{x}^\top \big( {(c-iB) \over 2 \lambda}  W \big) 
       + O_{L^2}(r,R),
       \end{equation}
 and therefore 
   $$ \nabla_{t}^\top W =  {2 c \over \epsilon^2} \nabla_{x}^\top W  + O_{L^2}(r, R).$$
Hence,  $|W|^2$ solves the scalar  transport equation
  $$ \partial_{t} |W|^2 =  {2 c \over \epsilon^2} \partial_{x} |W|^2 + F(t,x), \quad F= O_{L^2_x}(r,R) \cdot W = O_{L^1_x}(R,r).$$ 
  Solving explicitly this equation, we find  that 
$$  |W(t,x)|^2=  \left|W_{0} (x+{2 \, c \over \epsilon^2} t ) \right|^2 + \int_{0}^t F(s,  x+{2 \, c \over \epsilon^2} (t-s) \big)\, ds.$$
This yields by using the Fubini Theorem and a change of variable
\begin{align*}    \int_{0}^T |W(t,x)|^2 \, dt   & \lesssim  \epsilon^2 \|W_{0}\|_{L^2}^2 +   \Big|  \int_{0}^T  \int_{0}^t F(s,  x+{2 \, c \over \epsilon^2} (t-s)) \,ds\, dt \Big| \\
& \lesssim  \epsilon^2 \|W_{0}\|_{L^2}^2 + \epsilon^2  \int_{0}^T \Big| \int_{ x}^{ x+ {2 c \over \epsilon^2}(T-s)} F(s, y)  \, dy \Big| \, ds.
\end{align*}
Therefore, since $F = O_{L^1_x}(R,r)$, 
 \begin{align*}
 \sup_{x} \int_{0}^T |W(t,x)|^2 \, dt&\lesssim  \epsilon^2 ( \|W_{0}\|_{L^2}^2 +  T O(r,R)).
\end{align*}
 This finally yields 
\begin{equation}
\label{finpourW}
 \sup_{x} {1 \over \epsilon}  \int_{0}^T | W(t, x) | \, dt \lesssim \sqrt T \big( \|W_{0}\|_{L^2}+ \sqrt T O(r, R) \big).
 \end{equation}
By plugging this estimate in \eqref{phifinal1}, we thus obtain
\begin{equation}
\label{phifinalfinal}
 \sup_{[0, T]}\|\epsilon \phi \|_{L^\infty} \leq c_{0} + (\sqrt T + T) O(r,R).
 \end{equation}

By combining the last estimate and \eqref{finalplat1}, \eqref{finalplat2}, we get by a classical bootstrap argument that
$T^\epsilon$ is bounded from below by $T_0>0$ that is uniform for $\epsilon \in [0,1].$
This ends the proof of Theorem \ref{theounif} in the case $\mathcal{M}= \mathbb{R}^{2d}.$

\subsection{Proof of Theorem~\ref{theoKdV} in the case $\mathcal{M} = \mathbb{R}^{2d}$}
\label{sectionKdVlimit}
In this section only, we add $\epsilon$ superscripts to $u$, $p$, $n$, etc... in order to make the dependence on $\epsilon$ more clear.

We first remark that  $u^\epsilon= p^\epsilon+ \epsilon^2n^\epsilon= \Phi(\epsilon \phi^\epsilon) + \epsilon^2n^\epsilon$. 
Thanks to the uniform estimates of Theorem \ref{theounif}, we have by Sobolev embedding that 
$n^\epsilon$ is uniformly bounded in  $L^\infty$  and hence $\epsilon^2 n^\epsilon$ converges strongly
to zero in $L^\infty([0, T] \times \mathbb{R})$.  The study of the convergence of $\epsilon \phi^\epsilon$ is more delicate. This will be the first
step in the proof of Theorem \ref{theoKdV}.
   
\bigskip
\noindent
\underline{Step 1: Convergence of $\epsilon \phi^\epsilon$ to $0$ in $L^\infty$.} Notice that thanks to Theorem \ref{theounif}, we already have the estimate
 $$ \| \epsilon \phi^\epsilon\|_{L^\infty([0,T] \times \mathbb{R})} \leq 2 c_{0}.$$
Moreover, by using again the first line of \eqref{eqhydro}, we have
$$ \sup_x \int_{0}^T | \epsilon \partial_{t} \phi^\epsilon (t, x) | \, dt \lesssim \sup_{x} \int_{0}^T  {1 \over \epsilon }\left| \left( (c+iB)  D\Phi\partial_{x} \phi^\epsilon - 2 i \lambda  n^\epsilon \right) (t, x )\right| \, dt 
+     O(1), \quad \forall t \in [0, T].$$
Consequently, by using \eqref{Vdef} and \eqref{finpourW}, we get that
\begin{equation}
\label{bornedtphifinal} \sup_{x} \int_{0}^T | \epsilon \partial_{t} \phi^\epsilon (t, x) | \, dt  =O(1)
\end{equation}

We also have that   $\|\epsilon \partial_{x} \phi^\epsilon \|_{L^\infty([0,T] \times \mathbb{R})}  = O(\epsilon)$ thanks to 
 the uniform estimates of  Theorem \ref{theounif}. We thus  get from the Arzela-Ascoli Theorem that a subsequence $\epsilon_{n} \phi^{\epsilon_{n}}$ converges in $\mathcal{C}_{loc}( [0,T] \times \mathbb{R})$ towards some $\gamma_{\infty}(t,x)$
 and that since  $\epsilon_{n} \partial_x \phi^{\epsilon_{n}}$ converges to zero (in the distribution sense for example), we must have
$\partial_{x} \gamma_{\infty}(t,x)= 0$. Hence $\gamma_{\infty}$ is a function of time only that satisfies
\begin{equation}
     \label{gammaprop}
      \sup_{[0,T]} |\gamma_{\infty}(t)| + \int_{0}^T | \partial_{t} \gamma_{\infty} (t) | \, dt \leq C
     \end{equation}
for some $C>0$. However, it is not possible to prove directly that $\gamma_\infty = 0$, and hence $\epsilon \phi^\epsilon \to 0$.

In order to prove that $\epsilon \phi^\epsilon \to 0$, we shall proceed differently and make a crucial use of the assumption \eqref{hypW} and the conserved (or almost conserved) quantities of the Schr\"odinger maps system.

From the one-dimensional Sobolev inequality, we first observe that
\begin{equation}
\label{convphiLinfty} ( \epsilon \| \phi^\epsilon \|_\infty)^2 \leq { 1 \over 2} \| \epsilon \phi^\epsilon \|_{L^2} \| \epsilon \partial_{x} \phi^\epsilon  \|_{L^2} \leq  C \epsilon^2 \| \phi^\epsilon \|_{L^2}
       \end{equation}
        since $\|\partial_{x} \phi^\epsilon \|_{L^2}$ is uniformly bounded thanks to the estimates of Theorem \ref{theounif}.
         Next, by using  again the first line of \eqref{eqhydro}, we get that
         \begin{equation}
         \label{convphiL2} \| \epsilon^2 \phi^\epsilon (t) \|_{L^2}\lesssim \| \epsilon^2 \phi^\epsilon(0) \|_{L^2} + \int_{0}^t \| W^\epsilon (s)  \|_{L^2}\, ds + O( \epsilon^2)
         \end{equation}
          with $W^\epsilon=  (c+iB) D\Phi \partial_{x} \phi^\epsilon - 2 \lambda i n^\epsilon.$
          
\bigskip
          
Thus it suffices to prove that $ \lim_{\epsilon \rightarrow 0 }\sup_{[0, T]} \| W^\epsilon(t) \|_{L^2}= 0$ in order to deduce that $\epsilon \phi^\epsilon \overset{L^\infty}{\to} 0$. To this aim we now turn.

Observe that $|W|^2 = \lambda |D \Phi (\partial_{x} \phi^\epsilon)|^2 + 4 \lambda^2 |n^\epsilon|^2 - 4 \lambda i n^\epsilon \cdot (c+iB) D\Phi \partial_{x} \phi^\epsilon$. After integrating, the two first terms correspond, up to lower order terms, to the Hamiltonian, while the last term gives a quantity which can be thought of as momentum\footnote {The same problem occurs for the usual nonlinear Schr\"odinger equation, where the conserved momentum is given by $\mathcal{P}= {1 \over \epsilon^3} \int i \partial_{x}u^\epsilon \cdot u^\epsilon \, dx.$}, which should be almost conserved.

Guided by this idea, we will derive directly an approximately conserved quantity by working directly on the hydrodynamical system \eqref{eqhydro}. Keeping only the terms up to order $1$ in $\epsilon$ in \eqref{eqhydro}, it reads 
\begin{equation}
\label{eqhydrohamilton}
\left\{ 
\begin{array}{l} 
\displaystyle
S_{0} D\Phi \partial_t \phi= \frac{1}{  \epsilon^2}   i \Bigl[- 2 \lambda n  - i (c +  i B) S_{1} \partial_{x}   \phi   \\
 \displaystyle
   \mbox{\hspace{5cm}}  + \epsilon^2  \frac{1}{2} \deux^{\perp}(D\Phi \partial_x \phi, D\Phi \partial_{x} \phi) + \frac{1}{2} \epsilon^2 
(\nabla^{\perp}_x)^2 n  - \epsilon^2 F_{1}(n,n) \Bigr] + O_{H^1}(\epsilon) \\
\displaystyle
\nabla_t^\perp n    = {1 \over \epsilon^2} i \left[ \frac{1}{2} \nabla_x^\top (S_1 \partial_x \phi) 
 - i  (c+ i B) \nabla_{x}^\perp n 
 + \epsilon^2 \frac{1}{2} \deux^{\top} \left(
 D\Phi \partial_x \phi \, , \,\nabla_x^\perp n  \right) \right] + O_{H^1}(\epsilon).
\end{array}
\right.
\end{equation}
In view of the above right hand side, we define
\begin{multline*}
H(u^\epsilon)= \int \Bigl[   \lambda |n^\epsilon|^2 + {1 \over 4}| \epsilon^2 \nabla_{x}^\perp n^\epsilon |^2  + {1 \over 3} \epsilon^2 F_{1}(n^\epsilon,n^\epsilon) \cdot n^\epsilon
 + 
 {1 \over 4} |S_{0} D\Phi \partial_{x} \phi^\epsilon|^2 \Bigr. \\
 \Bigl.-   \ (c+iB)\left( D\Phi \partial_{x} \phi^\epsilon + { 1 \over 2} \epsilon^2 \deux^\top (D\phi \partial_{x} \phi^\epsilon, n^\epsilon) \right) \cdot in^\epsilon  \Bigr].
\end{multline*}
The above quantity is almost conserved in the sense that 
\begin{lem}
\label{lemP}
Under the assumptions of Theorem \ref{theounif}, we have
$$ 
{d \over dt} H (u^\epsilon)(t) = O(\epsilon)
$$
uniformly on $[0,T]$.
\end{lem}
We postpone the proof of this lemma until the end of this section. \\

\bigskip

We observe that for $t \in [0, T]$, we have
$$ H(u^\epsilon(t))=   {1\over 2} \int \left[ 2 \lambda |n^\epsilon|^2 + {1 \over 2} |D\Phi \partial_{x} \phi^\epsilon|^2  - 2 (c+ i B)D\Phi\partial_{x}\phi^\epsilon \cdot i  n^\epsilon \right]\, dx + O(\epsilon^2)$$
and hence that by using the definition \eqref{Vdef} of $W^\epsilon$ and~\eqref{H2}, we have
$$   H(u^\epsilon(t))= \frac{1}{4 \lambda} \int |W^\epsilon|^2 \,dx  + O(\epsilon^2).$$
By combining this observation with Lemma \ref{lemP}, we obtain that  
$$
\sup_{t \in [0,T]} \| W^\epsilon(t) \|_{L^2}^2 \lesssim \|W^\epsilon (0) \|_{L^2}^2 + \epsilon$$
 and hence 
since $\|W^\epsilon (0) \|_{L^2} \to 0$ as $\epsilon \to 0$, we finally obtain that 
\begin{equation}
\label{Wfort}
\sup_{t \in [0,T]} \| W^\epsilon(t) \|_{L^2} \to 0 \quad \mbox{as $\epsilon \to 0$}.
\end{equation}
 We have thus proven thanks to \eqref{convphiLinfty} and \eqref{convphiL2} that
 \begin{equation}
 \label{epsphi0}
 \lim_{\epsilon \rightarrow 0} \sup_{[0, T]} \| \epsilon \phi \|_{L^\infty} = 0.
 \end{equation}
 This ends the first step.

\bigskip

\noindent
\underline{Step 2: Derivation of the KdV limit.} From the estimates of Theorem \ref{theounif} that yield $\epsilon^2 n^\epsilon = O_{L^\infty}(\epsilon^2)$, we get that 
\beq
\label{uto0}
\| u^\epsilon  \|_{L^\infty([0,T] \times \mathbb{R})} \rightarrow 0.
\eeq
This yields in particular that tensors such as $\deux^\top_{p},$  $ \deux^\perp_{p}$ that  implicitly depend on $p$
converge uniformly towards $ \deux^\top_{0},$  $ \deux^\perp_{0}$.

The uniform $H^s$ estimates of Theorem \ref{theounif} provide local compactness in space, therefore we only need to 
get compactness in time in order to obtain strong convergence.

Define
$$
A^\epsilon = 2 i \lambda n^\epsilon \quad \mbox{and} \quad U^\epsilon= (c-iB) D\Phi \partial_{x} \phi^\epsilon + A^\epsilon.
$$
Note that $W^\epsilon = (c+iB) D\Phi \partial_{x}\phi^\epsilon- A^\epsilon$.
By combining  \eqref{phifinal2} and \eqref{nfinalconv},
$$ \nabla_{t}^\top U^\epsilon=  O_{L^2}(1).$$
In particular, we obtain that  $U^\epsilon$ satisfies for every $t$, $s$,  $0 \leq s \leq t \leq T$
$$ \| U^\epsilon(t) -U^\epsilon(s)\|_{L^2}^2 = \left| \int_{\mathbb{R}} \int_{s}^t \nabla_{t}^\top U(\tau,x) \cdot U(\tau,x) \, d\tau  dx \right| \lesssim |t-s|$$
by the Cauchy-Schwarz inequality.

From Theorem \ref{theounif}, we have that  $U^\epsilon$  is  bounded in $L^\infty \big([0, T], H^s(u^{-1}T\mathcal{L}) \big)$,
where $H^s(u^{-1}T\mathcal{L})$ is given by the application of covariant derivatives. Since $\phi$ is bounded in $H^s(\mathbb{R}^d)$, this implies that $U^\epsilon$ is bounded in $L^\infty \big([0, T], H^s(\mathbb{R}^{2d}) \big)$, where this time we simply view $U^\epsilon$ as a vector in $\mathbb{R}^{2d}$. We can now apply the Arzela-Ascoli Theorem, and get that there exists a sequence $\epsilon_{n}$ such that  $U^{\epsilon_{n}}$ converges in $ \mathcal{C}([0,T], H^\sigma_{loc})$ to $U$  for every $\sigma$, $\sigma <s$ for some $U\in L^\infty([0, T], H^s)$ . From \eqref{Wfort}, we already  had   that  $ W^\epsilon$ converges strongly to   $0$ in  $ \mathcal{C}([0,T], L^2)$.
Since $W^\epsilon$ is also bounded in $L^\infty \big([0, T], H^s \big)$, this yields by interpolation that 
it converges strongly to $0$ in  $ \mathcal{C}([0,T], H^\sigma)$ for every $ \sigma$, $\sigma <s$.
      Next, since
      \begin{align}
      \label{relationfinale}
       D\Phi \partial_{x} \phi^\epsilon=  {1 \over 2 c} ( U^\epsilon  +  W^\epsilon) , \quad A^\epsilon =  2 \lambda i n^\epsilon =  {1 \over 2 c} \left( (c+iB) U^\epsilon - (c-iB) W^\epsilon\right)
       \end{align}
       we get that  $D\Phi \partial_{x} \phi^{\epsilon_{n}}$  and $in^{\epsilon_{n}}$  also converge in $ \mathcal{C}( [0, T], H^{\sigma}_{loc})$.
        Moreover, by denoting $A \in L^\infty([0, T], H^s)$ the limit of $A^{\epsilon_{n}}$, we get that
    \begin{equation}
    \label{convergence}
    D\Phi \partial_{x}\phi^{\epsilon_{n}} \rightarrow  (c+iB)^{-1} A = { 1 \over \lambda}(c-iB) A, \quad in^{\epsilon_{n}} \rightarrow  {1 \over 2 \lambda } A \qquad \mbox{ in }  \mathcal{C}([0,T], H^\sigma_{loc})
\end{equation} 
To identify, the limit system,  we need to make all the order one terms explicit.  By applying $\nabla_{x}^\top$ to the first line  of \eqref{eqhydro} written as 
\begin{align*}
& D \Phi \partial_t \phi^\epsilon - \frac{1}{\epsilon^2}(c+iB) D\Phi \partial_x \phi^\epsilon \\
&\qquad  = S_0^{-1} i \left[ \frac{1}{2} \deux^{\perp}(S_1 \partial_x \phi^\epsilon, D\Phi \partial_{x} \phi^\epsilon) + \frac{1}{2} (\nabla^{\perp}_x)^2 n^\epsilon -   2 \lambda \frac{n^\epsilon}{\epsilon^2} - F_{1}(p)(n^\epsilon,n^\epsilon) -  {1 \over \epsilon^4}P^\perp R^V(p^\epsilon,\epsilon^2n^\epsilon) \right]
\end{align*}
(with the help of Corollary~\ref{BdeuxT}), we find
 \begin{multline*}
\nabla_{t}^\top (D\Phi \partial_{x}\phi^\epsilon) =  {1 \over \epsilon^2} \nabla_{x}^\top W^\epsilon  +  {1 \over 2}  (\nabla_{x}^\top)^3 (in^\epsilon)  - 2 i F_{1}(\nabla_{x}^\perp n^\epsilon, n^\epsilon)
   \\+i   \deux^\perp( \nabla_{x}^\top (D\Phi \partial_{x}\phi^\epsilon), D\Phi \partial_{x}\phi^\epsilon) - 4 \lambda i \deux^\perp ( \nabla_{x}^\top (in^\epsilon), i n^\epsilon) +O_{L^2}(\epsilon)
\end{multline*}
Note that, for  the last term, we have used that  by the definition of $S_0$ and Proposition \ref{deuxsym},
\begin{align*}
 \nabla_x^\top S_0^{-1} i \left(-2 \lambda  \frac{n}{\epsilon^2} \right) & =  \nabla_x^\top \left[ \operatorname{Id} + \epsilon^2 \deux^\top (\cdot,n) \right]^{-1}  i \left(-2\lambda \frac{n}{\epsilon^2} \right) \\
& = -\frac{2 \lambda }{\epsilon^2} \nabla_x^\top (in^\epsilon) -  2 \lambda \nabla_x^\top \deux^\top ( i n^\epsilon, i n^\epsilon) + O_{L^2}(\epsilon) \\
& =  -\frac{2 \lambda }{\epsilon^2} \nabla_x^\top (in^\epsilon) - 4i \lambda  \deux^\perp (\nabla_x^\top( in ^\epsilon), in^\epsilon) + O_{L^2}(\epsilon).
\end{align*}
Similarly, the second line of~\eqref{eqhydro} can be written as a more precise version of \eqref{infinal} as follows
\begin{align*}
\nabla_{t}^\top  ( 2 \lambda in^\epsilon) = - {1 \over \epsilon^2}( c-iB) \nabla_{x}^\top W^\epsilon
   +   2  i  \lambda  \deux^\perp \big(D\Phi\partial_{x}\phi^\epsilon , \nabla_{x}^\top i  n^\epsilon \big)  + i \lambda  \deux^\perp (\nabla_{x}^\top(D\Phi \partial_{x} \phi^\epsilon), i n^\epsilon\big)  + O_{L^2}(\epsilon).
\end{align*}
Note that we have again used Proposition \ref{deuxsym}.
Consequently, we can combine the two equations   to get
\begin{equation}
\begin{split}
\label{presquefini1} \nabla_{t}^\top \big(  (c-iB) D\Phi \partial_{x} \phi^\epsilon &+ 2 \lambda i n^\epsilon)  = (c-iB) \left[
{1 \over 2}  (\nabla_{x}^\top)^3 (in^\epsilon)  - 2 i F_{1}(\nabla_{x}^\perp n^\epsilon, n^\epsilon) \right. \\ 
& \left. +i   \deux^\perp( \nabla_{x}^\top (D\Phi \partial_{x}\phi^\epsilon), D\Phi \partial_{x}\phi^\epsilon) - 4 \lambda i \deux^\perp ( \nabla_{x}^\top (in^\epsilon), i n^\epsilon) \right] \\
& +  2  i  \lambda  \deux^\perp \big(D\Phi\partial_{x}\phi^\epsilon , \nabla_{x}^\top i  n^\epsilon \big)  + i \lambda  \deux^\perp (\nabla_{x}^\top(D\Phi \partial_{x} \phi^\epsilon), i n^\epsilon\big)  + O_{L^2}(\epsilon).
\end{split}
\end{equation}
Using \eqref{convergence}, it is easy to pass to the limit weakly (in $\mathcal{S}'(\mathbb{R}^{2d})$) in all the terms above, except for the one involving a time derivative on the right hand side. Indeed, only poor estimates are available on $\epsilon \partial_{t} \phi^\epsilon$. Due to this term, we proceed differently: take $\psi \in \mathcal{C}^\infty_{c}(]0, T[ \times \mathbb{R}, \mathbb{R}^d)$  (identifying $T_{0} \mathcal{L}$ and $\mathbb{R}^d$) and multiply the above by $D\Phi (\epsilon \phi^\epsilon) \psi$ before letting $\epsilon \to 0$. 

Thanks to \eqref{convergence}, it is easy to see that $\epsilon_n \to 0$,
\begin{equation}
\label{membredroite}
\begin{aligned}
\iint RHS\eqref{presquefini1} \cdot D\Phi (\epsilon \phi^\epsilon) \psi\,dx\,dt \rightarrow & \iint \psi \cdot \left[ (c-iB) \left[ { 1 \over  4 \lambda}
 (\nabla_{x}^\top)^3  A  - { 1 \over 2 \lambda^2} iF_{1, 0}( i \nabla_{x}^\top A,  iA) \right. \right.\\
& \left.
  + {1 \over \lambda^2} i \deux^\perp_{0}( (c-iB) \nabla_{x}^\top A,  (c-iB)A)
   -  {1 \over  \lambda}  i \deux^\perp_{0} ( \nabla_{x}^\top A, A) \right] \\
& \left.    + {1 \over \lambda}i \deux^\perp_{0} ((c-iB) A, \nabla_{x}^\top A)
      +  {1 \over 2 \lambda} i \deux^\perp_{0}( (c-iB) \nabla_{x}^\top A, A) \right] \,dx\,dt.
\end{aligned}
\end{equation}
(where $F_{1,0}= F_1(0)$ and $\deux^\perp_0 = \deux^\perp(0)$).

The left-hand side of~\eqref{presquefini1} is more delicate, because of the time derivative which appears there, and of the poor estimates available on $\epsilon \partial_{t} \phi^\epsilon$. In order to pass to the limit on this term, we basically need to justify that if $X^\epsilon$ is a vector field in $T_{\Phi(\epsilon \phi^\epsilon)} \mathcal L$ that  converges strongly in $L^2_{loc}([0, T] \times 
  \mathbb{R})$ towards $X$ and which is bounded in $L^\infty([0, T] \times \mathbb{R})$, then
   $\nabla_{t}^\top  X^\epsilon= P^\top_{\Phi(\epsilon \phi^\epsilon)} \partial_{t} X^\epsilon$ converges weakly towards $\nabla^\top_{t} X= P^\top(0) \partial_{t} X$.
    Taking as above $\psi \in \mathcal{C}^\infty_{c}(]0, T[ \times \mathbb{R}, \mathbb{R}^d)$, we have after integrating by parts
    \begin{multline*}
     \int_{\mathbb{R} \times \mathbb{R}}
     \nabla_{t}^\top  X^\epsilon \cdot D\Phi(\epsilon \phi^\epsilon) \psi \, dx dt=
      - \int_{\mathbb{R} \times \mathbb{R}} X^\epsilon \cdot  D\Phi(\epsilon \phi^\epsilon) \partial_{t} \psi \,dt\, dx  \\
       -  \int_{\mathbb{R} \times \mathbb{R}} X^\epsilon \cdot  \nabla^\top D\Phi_{\epsilon \phi^\epsilon}(  \epsilon \partial_{t} \phi^\epsilon, \psi) \, dt\, dx .
\end{multline*}
  The first integral in the right-hand side above obviously converges towards
  $$  -\int_{\mathbb{R} \times \mathbb{R}} X \cdot \partial_{t} \psi \, dt dx =   \int_{\mathbb{R} \times \mathbb{R}}  \nabla_{t}^\top X \cdot  \psi \, dt dx,$$
   thus we just have to prove that the second integral tends to zero.
    By using  \eqref{bornedtphifinal}, we obtain that
    \begin{multline*} \left|  \int_{\mathbb{R} \times \mathbb{R}} X^\epsilon \cdot  \nabla^\top D\Phi_{\epsilon \phi^\epsilon}(  \epsilon \partial_{t} \phi^\epsilon, \psi) \, dt\, dx  \right| \\
     \lesssim  \| \nabla^\top D\Phi_{\epsilon \phi^\epsilon}\|_{L^\infty( [0, T] \times \mathbb{R})} \left( \sup_{x} \int_{0}^T | \epsilon \partial_{t} \phi^\epsilon|  \right) \, \| \psi\|_{L^1_x L^\infty_t ([0, T] \times \mathbb{R})} 
      \| X^\epsilon \|_{L^\infty([0, T] \times \mathbb{R}) } \end{multline*}
      In the above estimate, all the terms are uniformly bounded,  and since $\epsilon \phi^\epsilon$ converges uniformly to zero, we have thanks to \eqref{nablatopDPhi}
       (which relies on the choice of normal coordinates on $\mathcal{L}$) that
       $ \| \nabla^\top (D\Phi)_{\epsilon \phi^\epsilon}\|_{L^\infty( [0, T] \times \mathbb{R})}$ tends to zero. We have thus proven that\
\begin{equation}
\label{membregauche}
\iint LHS\eqref{presquefini1} \cdot D\Phi (\epsilon \phi^\epsilon) \psi\,dx\,dt \rightarrow \iint \psi \cdot \nabla_{t}^\top \left( \frac{(c-iB)^2}{\lambda} A + A \right)\,dx\,dt
\end{equation}
Combining~\eqref{membredroite} and~\eqref{membregauche} gives the following equality in $\mathcal{S}'(\mathbb{R}^d)$ (identifying $T_0 \mathcal L$ with $\mathcal{R}^d$):
\begin{equation}
\label{KdVversion1}
\begin{split}
\nabla_{t}^\top \left( \frac{(c-iB)^2}{\lambda} A + A \right) = &  (c-iB) \left[ { 1 \over  4 \lambda}
 (\nabla_{x}^\top)^3  A  - { 1 \over 2 \lambda^2} iF_{1, 0}( i \nabla_{x}^\top A,  iA) \right. \\
& \left.
  + {1 \over \lambda^2} i \deux^\perp_{0}( (c-iB) \nabla_{x}^\top A,  (c-iB)A)
   -  {1 \over  \lambda}  i \deux^\perp_{0} ( \nabla_{x}^\top A, A) \right] \\
&  + {1 \over \lambda}i \deux^\perp_{0} ((c-iB) A, \nabla_{x}^\top A)
      +  {1 \over 2 \lambda} i \deux^\perp_{0}( (c-iB) \nabla_{x}^\top A, A).
\end{split}
\end{equation}
The above system \eqref{KdVversion1} is 
 is the desired KdV type equation. 
  We can simplify it a little bit, by noticing that thanks to~\eqref{H2}, we have
  $$ {(c-iB)^2 \over \lambda} A + A =   2 c {(c-iB) \over \lambda} A$$
  and by using Corollary \ref{BdeuxT} to obtain
\begin{equation}
\label{KdVversionfinale}
2  c\nabla^\top_{t} A  = \frac{1}{4} \nabla^{\top 3}_{xxx}  A  + \left( \frac{3}{2}- \frac{2 \mu}{\lambda} - \frac{2c}{\lambda}i_0 B_0 \right) i_0 \deux^\perp_{0} \left(  \nabla^\top_{x} A, A\right)  - \frac{i_0}{2\lambda} F_{1,0}(i_0 \nabla^\top_{x} A, i_0 A).
\end{equation}
Note that  here, $\nabla^\top$  stands for $P^\top_{0} \nabla$ and therefore, since $T_{0}\mathcal{L}$ is a  fixed vector
space, we can also  write it  as
\begin{equation}
\label{presquefini2}
2  c\partial_{t} A  = \frac{1}{4} \partial_{xxx}  A  + \left( \frac{3}{2} - \frac{2 \mu}{\lambda} - \frac{2c}{\lambda}i_0 B_0 \right) i_0 \deux^\perp_{0} \left(  \partial_{x} A, A\right)  - \frac{i_0}{2\lambda} F_{1,0}(i_0 \partial_{x} A, i_0 A).
\end{equation}
From the uniqueness for this  KdV-type system (see section \ref{ROTLKS}), 
we thus get that the whole family  $A^\epsilon $, $(c+iB) \Phi \partial_{x} \phi^\epsilon$ tends to $A$ in $\mathcal{C}([0, T], H^\sigma_{loc}).$

\bigskip

\noindent
\underline{Step 3: Global in space convergence.} To obtain the convergence in  $\mathcal{C}([0, T], L^2)$, we can proceed as follows. At first, we note that the convergence of $U^\epsilon$  also holds in  $\mathcal{C}[0, T], L^2_{w})$ ($L^2$ being equipped with the weak topology) and that $U^\epsilon$ tends to $$ U= {(c-iB)^2 \over \lambda} A + A =   2 c {(c-iB) \over \lambda} A.$$
  We shall prove the  global strong convergence in $L^2(\mathbb{R})$ of $U^\epsilon$.
  We note that by using that  $W^\epsilon $ converges strongly to zero in   $ \mathcal{C}([0,T], H^\sigma)$, for $\sigma <s$, 
 and the relations \eqref{relationfinale}, 
 we can rewrite \eqref{presquefini1} as
 \begin{multline*} 
 \nabla_{t}^\top  U^\epsilon= {1 \over 8 c } (\nabla_{x}^\top)^3 \left( U^\epsilon - {(c-iB)^2 \over \lambda} W^\epsilon \right)  - {1 \over 8 \lambda^2c^2}(c-iB) iF_{1}\left( (c-iB) i \nabla_{x}^\top U^\epsilon, (c-iB) iU^\epsilon\right) \\
+ {5 \over 8 c^2}  (c-iB) i \deux^\perp \big(\nabla_{x}^\top U^\epsilon, U^\epsilon  \big)  - {1 \over 4 \lambda c^2} (c-iB) i \deux^\perp \left((c+iB) \nabla_{x}^\top U^\epsilon, (c+iB) U^\epsilon\right)
  +  o_{L^2}(1).
\end{multline*}
(where $o_L^{2}(1)$ stands for a function $f^\epsilon$ such that $f_\epsilon \to 0$ in $L^\infty_t L^2_x$). This implies that
 \begin{align*}
{1 \over 2} {d \over dt} \int_{\mathbb{R}} |U^\epsilon|^2\, dx = &   - {1 \over 8c\lambda} \int (c-iB)^2 ( \nabla_{x}^\top)^3 W^\epsilon \cdot U^\epsilon \,dx + {5 \over 8 c^2} \int  (c-iB) i \deux^\perp \big(\nabla_{x}^\top U^\epsilon, U^\epsilon  \big) \cdot U^\epsilon \,dx \\
& - {1 \over 8 \lambda^2 c^2} \int (c-iB) iF_{1}\left( (c-iB) i \nabla_{x}^\top U^\epsilon, (c-iB) iU^\epsilon\right) \cdot U^\epsilon \,dx \\
& - {1 \over 4 \lambda c} \int  (c-iB) i \deux^\perp \left((c+iB) \nabla_{x}^\top U^\epsilon, (c+iB) U^\epsilon\right) \cdot U^\epsilon  \,dx +   o(1).
 \end{align*}
By using~\eqref{H2},  assertion (4) in Proposition \ref{deuxsym} and the symmetry of $F_{1}$, we  have by integrating by parts  that
\begin{align*}
&   \int (c-iB) iF_{1}\left( (c-iB) i \nabla_{x}^\top U^\epsilon, (c-iB) iU^\epsilon\right) \cdot U^\epsilon \,dx  \\
 & \qquad  = -  \int F_{1}\left( (c-iB) i \nabla_{x}^\top U^\epsilon, (c-iB) i U^\epsilon\right) \cdot (c-iB)i U^\epsilon \,dx =O(\epsilon) \\
 &  \int  (c-iB) i \deux^\perp \left((c+iB) \nabla_{x}^\top U^\epsilon, (c+iB) U^\epsilon\right) \cdot U^\epsilon \,dx \\
 & \qquad = \int i \deux^\perp \left( (c+iB) \nabla_{x}^\top U^\epsilon, (c+iB) U^\epsilon\right) \cdot (c+iB) U^\epsilon \,dx =
  O(\epsilon),
\end{align*}
 and also from 
 the last property of Corollary \ref{BdeuxT} that
 $$  \int  (c-iB) i \deux^\perp \big(\nabla_{x}^\top U^\epsilon, U^\epsilon  \big) \cdot U^\epsilon\,dx  = O(\epsilon).$$
 Moreover, integrating by parts again,
 $$  \left|  \int  (c-iB)^2( \nabla_{x}^\top)^3 W^\epsilon \cdot U^\epsilon \,dx \right| \lesssim \| W^\epsilon \|_{H^2} \|U^\epsilon \|_{H^1} =o(1).$$
  We have thus proven that
 $$ {d \over dt} {1 \over 2} \int_{\mathbb{R}} | U^\epsilon  |^2 \,dx=   o(1).$$
 This yields for $t \in [0,T]$
 $$ \| U^\epsilon  (t) \|_{L^2}^2 =  \|U^\epsilon (0) \|_{L^2}^2+ o(1).$$
   For  the limit equation (see Section \ref{ROTLKS}) \eqref{presquefini2}, we  have that
  $$   \| A(t) \|_{L^2}^2 =  \|A(0) \|_{L^2}^2 $$
  and we observe  (again by~\eqref{H2}) that
  $$ \|U(t)\|_{L^2}^2 =  { 4c^2 \over \lambda^2} \| (c-iB) A \|_{L^2}^2= {4 c^2 \over \lambda} \|A(t)\|_{L^2}^2, \quad \forall t \in [0, T].$$
Consequently, the $L^2$ convergence  at the initial time yields that $  \| U^\epsilon (t) \|_{L^2}^2 \rightarrow {4c^2 \over \lambda} \|A(t) \|_{L^2}^2$ uniformly in time. 
   This yields  that $U^\epsilon$ converges  in  $\mathcal{C}([0, T], L^2)$. Since we already had that  $W^\epsilon$ converges strongly to zero in 
    $\mathcal{C}([0, T], L^2)$, we finally obtain 
     from \eqref{relationfinale} 
    that $A^\epsilon$ and $(c+iB) D\Phi \partial_{x} \phi^\epsilon$ converge in   $\mathcal{C}([0, T], L^2(\mathbb{R}))$ to $A$.
     From the uniform $H^s$ estimates, this also yields convergence in $\mathcal{C}([0, T], H^\sigma (\mathbb{R}))$, $\sigma <s$ and by Sobolev embedding in $L^\infty([0, T]  \times \mathbb{R}).$
    
\bigskip

It remains to prove Lemma \ref{lemP}
\subsubsection*{Proof of Lemma \ref{lemP}}
Let us split $H$ into
$$ H (u^\epsilon)= H_{1}(u^\epsilon)+ H_{2}(u^\epsilon)+ H_{3}(u^\epsilon)$$
 with 
 \begin{align*}
 &  H_{1}(u^\epsilon)=  \int  \lambda |n^\epsilon|^2 + {1 \over 4}| \epsilon^2 \nabla_{x}^\perp n^\epsilon |^2  + {1 \over 3} \epsilon^2 F_{1}(\Phi)(n^\epsilon,n^\epsilon) \cdot n^\epsilon\,dx, \\
 & H_{2}(u^\epsilon)= \int {1 \over 4} |S_{0} D\Phi \partial_{x} \phi^\epsilon|^2 \,dx,  \\
 & H_{3}(u^\epsilon)= -\int   (c+iB)\left( D\Phi \partial_{x} \phi^\epsilon + { 1 \over 2} \epsilon^2 \deux^\top (D\Phi \partial_{x} \phi^\epsilon, n^\epsilon) \right) \cdot in^\epsilon \,dx.
 \end{align*}
 In the following computations,  we shall make an  extensive use of Proposition \ref{deuxsym}, Corollary \ref{BdeuxT}, \eqref{comperp}, \eqref{comagain} and the symmetry in its arguments 
  of the trilinear
 application defined by \eqref{F1def}.
 We   first obtain easily  that 
 \begin{equation}
 \label{H1u}
 {d \over dt }   H_{1}(u^\epsilon)= \int  \nabla^\perp_{t} n^\epsilon \cdot \left(2 \lambda n^\epsilon - { 1 \over 2} \epsilon^2 ( \nabla_{x}^\perp)^2 n^\epsilon + \epsilon^2 F_{1}(n^\epsilon, n^\epsilon) \right)\,dx
  + O(\epsilon).
 \end{equation}
 Next,
 \begin{multline*}
 {d \over dt }   H_{2}(u^\epsilon) = \int  {1 \over 2} S_{0} \nabla_{x}^\top (D\Phi \partial_{t} \phi^\epsilon) \cdot S_{0} D\Phi \partial_{x}\phi^\epsilon + {1 \over 2 } \epsilon^2 \deux^\top(D\Phi \partial_{x}
  \phi^\epsilon, \nabla_{t}^\perp n^\epsilon ) \cdot  D\Phi \partial_{x} \phi^\epsilon \,dx+ O(\epsilon)
 \end{multline*}
Integrating by parts, we have
\begin{align*}
& \int  S_{0} \nabla_{x}^\top (D\Phi \partial_{t} \phi^\epsilon) \cdot S_{0} D\Phi \partial_{x}\phi^\epsilon \,dx  \\
& \qquad \qquad \qquad =-\int  S_{0} D\Phi \partial_{t} \phi^\epsilon \cdot  \left( \nabla_{x}^\top( S_{0} D\Phi \partial_{x} \phi^\epsilon)   +  \epsilon^2 \deux^\top(  D\Phi \partial_{x}\phi^\epsilon, \nabla_{x}^\perp n^\epsilon)\right)\,dx
   + O(\epsilon)
\end{align*}
   and by using Proposition \ref{deuxsym}, we note that
 \begin{equation}
 \label{utilefin}
 Y \cdot \deux^\top (X,N) = - \deux^\perp(X,Y) \cdot N, \quad \forall X,\, Y \in T\mathcal{L}, \, \forall N \in N\mathcal{L}
 \end{equation}
    and hence
$$  \int\deux^\top(D\Phi \partial_{x}
\phi^\epsilon, \nabla_{t}^\perp n^\epsilon ) \cdot  D\Phi \partial_{x} \phi^\epsilon\,dx
= -\int \nabla_{t}^\perp n^\epsilon \cdot \deux^\perp( D\Phi \partial_{x} \phi^\epsilon, D\Phi \partial_{x} \phi^\epsilon)\,dx.$$
  This yields
\begin{multline}
\label{H2u}
{d \over dt }   H_{2}(u^\epsilon) =  -  \int  \left[ S_{0} D\Phi \partial_{t} \phi^\epsilon  \cdot \left[ {1 \over 2 }\nabla_{x}^\top(S_{0} D\Phi \partial_{x} \phi^\epsilon)  +   {1 \over 2 }
  \epsilon^2 \deux^\top(  D\Phi \partial_{x}\phi^\epsilon, \nabla_{x}^\perp n^\epsilon) \right] \right. \\
   \left.  +  {1 \over 2} \nabla_{t}^\perp n^\epsilon \cdot  \epsilon^2 \deux^\perp( D\Phi \partial_{x} \phi^\epsilon, D\Phi \partial_{x} \phi^\epsilon) \right]\,dx + O(\epsilon).
  \end{multline}
  Finally, let us study the evolution of $H_{3}(u^\epsilon)$. Write first
 \begin{multline*}
 {d \over dt } H_{3}(u^\epsilon)=  \int \left[  \nabla_{t}^\perp n^\epsilon \cdot  i  (c+ iB) \left( D\Phi \partial_{x} \phi^\epsilon  + {1 \over 2} \epsilon^2 \deux^\top( D\Phi \partial_{x} \phi^\epsilon, n^\epsilon ) \right)
   + D\Phi \partial_{t} \phi^\epsilon \cdot i (c+iB) \nabla_{x}^\perp n^\epsilon \right. \\
    \left. - {1 \over 2 } \epsilon^2  (c+iB) \deux^\top( \nabla_{x}^\top( D\Phi \partial_{t} \phi^\epsilon), n^\epsilon) \cdot i n^\epsilon - {1 \over 2 } \epsilon^2
     (c+iB) \deux^\top (D\Phi \partial_{x} \phi^\epsilon, \nabla_{t}^\perp n^\epsilon) \cdot  i n^\epsilon \right] \,dx + O(\epsilon) 
 \end{multline*}
 Next, we observe  that after integrating by parts,
 \begin{align*}
 & - {1 \over 2 } \epsilon^2  \int (c+iB) \deux^\top( \nabla_{x}^\top( D\Phi \partial_{t} \phi^\epsilon), n^\epsilon) \cdot i n^\epsilon\,dx \\
& \quad = {1 \over 2} \epsilon^2 \int  \left[ (c+ iB) \deux^\top(D\Phi \partial_{t} \phi^\epsilon, n^\epsilon) \cdot i \nabla_{x}^\perp n^\epsilon +  (c+iB) \deux^\top( D\Phi \partial_{t} \phi^\epsilon, \nabla_{x}^\perp n^\epsilon)
  \cdot i n^\epsilon \right]\,dx + O(\epsilon) \\
 & \quad =   \epsilon^2 \int \deux^\top(D\Phi \partial_{t} \phi^\epsilon, n^\epsilon) \cdot i (c+iB) \nabla_{x}^\perp n^\epsilon\,dx + O(\epsilon)
 \end{align*}
  and  that 
 $$
  -  {1 \over 2 } \epsilon^2 \int
     (c+iB) \deux^\top (D\Phi \partial_{x} \phi^\epsilon, \nabla_{t}^\perp n^\epsilon) \cdot  i n^\epsilon\,dx
     = {1 \over 2} \epsilon^2 \int  i(c+iB) \deux^\top(D\Phi \partial_{x} \phi^\epsilon,  n^\epsilon) \cdot \nabla_{t}^\perp n^\epsilon\,dx. 
    $$
    Consequently, we find
 \begin{align}
  \nonumber {d \over dt } H_{3}(u^\epsilon) & = \int \left[ \nabla_{t}^\perp n^\epsilon \cdot  i    (c+ iB) \left (  D\Phi \partial_{x} \phi^\epsilon  +\epsilon^2 \deux^\top( D\Phi \partial_{x} \phi^\epsilon, n^\epsilon ) \right) 
  \right. \\
\nonumber   & \qquad \qquad  \left. + \left( D\Phi \partial_{t} \phi^\epsilon + \epsilon^2 \deux^\top( D\Phi \partial_{t} \phi^\epsilon, n^\epsilon)\right) \cdot  i(c+iB) \nabla_{x}^\perp n^\epsilon  \right]\,dx\\
  \label{H3u}
  & =  \int \left[ \nabla_{t}^\perp n^\epsilon \cdot  i    (c+ iB)  S_{0} D\Phi \partial_{x} \phi^\epsilon   +  S_{0}D\Phi \partial_{t} \phi^\epsilon \cdot  i(c+iB) \nabla_{x}^\perp n^\epsilon  \right]\,dx.
 \end{align}
By collecting \eqref{H1u}, \eqref{H2u}, \eqref{H3u}, we thus find
\begin{align*}
{d \over dt}
H(u^\epsilon) & =  - \int \left[  S_{0}D\Phi \partial_{t}\phi^\epsilon \cdot \left( {1 \over 2 }\nabla_{x}^\top(S_{0} D\Phi \partial_{x} \phi^\epsilon)  +   {1 \over 2 }  \epsilon^2 \deux^\top(  D\Phi \partial_{x}\phi^\epsilon, \nabla_{x}^\perp n^\epsilon) - i (c+iB) \nabla_{x}^\perp n^\epsilon\right) \right] \,dx \\
& \qquad \qquad -  \int \left[  \nabla_{t}^\perp n^\epsilon \cdot \left(   - 2 \lambda n^\epsilon + { 1 \over 2} \epsilon^2 ( \nabla_{x}^\perp)^2 n^\epsilon -  \epsilon^2 F_{1}(n^\epsilon, n^\epsilon) -  i    (c+ iB)  S_{0} D\Phi \partial_{x} \phi^\epsilon \right. \right.  \\
& \qquad \qquad \qquad \qquad \left. \left. +  {1 \over 2}  \epsilon^2 \deux^\perp( D\Phi \partial_{x} \phi^\epsilon, D\Phi \partial_{x} \phi^\epsilon)  \right) \right]\,dx  + O(\epsilon).
\end{align*} 
By using the hydrodynamical system \eqref{eqhydrohamilton} to express $\nabla_{t}^\perp n $ and $S_{0}D\Phi \partial_{t} \phi$ in each term, we obtain
 that the two integrals cancel up to the remainders $O(\epsilon)$ and hence that 
$$   {d \over dt}
H(u^\epsilon) =O(\epsilon).$$
This ends the proof of Lemma \ref{lemP}.

\section{The case of a general K\"ahler manifold}
\label{sectionKahler}

With the KdV scaling \eqref{KdVscaling}, our Schr\"odinger map system reads
\beq
\label{SMK}
\Big( \partial_{t} - { c  \over \epsilon^2} \partial_{x}  \Big)u = i \Big(  {1 \over 2 \epsilon } \nabla_{x} \partial_{x} u  + \frac{1}{\epsilon^2} B \partial_{x} u  - {1 \over \epsilon^3} V'(u) \Big).
\eeq
Note that here, we deal with the general case where  $i= i(u)$ and $B=B(u)$ depend on $u$. To generalize the decomposition $u= P+N$, with $P= \Phi(\epsilon \phi)$, $N=\epsilon^2n$ that we have previously used, 
we define a  parametrization of $\mathcal{M}$ in a vicinity of $0 \in \mathcal{L}$ by
\beq
\label{localcoordK}
u=  \Psi(p, N) \overset{def}{=} \exp^\mathcal{M}_{p}( N), \quad p = \Phi(\epsilon \phi)\overset{def}{=} \exp^\mathcal{L}_{0}(\epsilon \phi), \quad N=\epsilon^2 n,
\eeq
where $p \in \mathcal{L}$, $\phi \in T_0 \mathcal{L}$, $N,n \in N_p \mathcal{L}$, and
$\exp^\mathcal{M}$ and $\exp^\mathcal{L}$ are the Riemannian exponential maps on $\mathcal{M}$ and $\mathcal{L}$ respectively.
This yields a parametrization of $\mathcal{M}$ in the vicinity of zero by the normal bundle of $\mathcal{L}$:
$$ \Psi: \, N\mathcal{L} \rightarrow \mathcal{M}.$$
\underline{We will assume in this section that}
$$
V(u) = \lambda  \operatorname{dist}(u,\mathcal{L})^2 \quad \mbox{or equivalently} \quad V(\Psi(p,N)) = \lambda  |N|^2.
$$
In other words, we assume that there are no cubic or higher order terms in the expansion \eqref{Vdefbis} of $V$, which will alleviate notations. Cubic terms and higher  do not present any
particular difficulty and can be treated by following the proof of the flat case $\mathcal{M} = \mathbb{R}^{2d}$.

\subsection{Geometric preliminaries II}

\label{sectiongeom}
\label{GPII}

\subsubsection{Basic setup}

We start with a K\"ahler manifold $\mathcal{M}$ with metric $(X,Y) \mapsto X \cdot Y = \langle X,Y \rangle$,
Levi-Civita connection $\nabla$,  complex structure $i$, and Riemannian curvature tensor $R$; 
and a Lagrangian submanifold $\mathcal{L} \subset \mathcal{M}$. We use the geometric notation defined in Section~\ref{sectionprelim}.
If $p \in \mathcal{L}$, the tangent and normal bundles of $\mathcal{L}$ are denoted $T_p \mathcal{L}$ and $N_p \mathcal{L}$ respectively. Just like in the previous section, we use  connections on these bundles denoted  by 
$\nabla^\top = P^\top \nabla$ and $\nabla^\perp = P^\perp \nabla$, and second fundamental forms by $\deux^\top = -\nabla P^\top$ and $\deux^\perp = -\nabla P^\perp$

As was already explained in the previous section, we do not distinguish in the notations between differentiation in a vector bundle and its pull-back by a map.

In order to write a hydrodynamical system,  the first step is to  understand the structure  of $T  (N \mathcal{L})$. This was already done
in \cite{ShatahZeng}.  We include it here for the sake of completeness. Then, we shall explain  how we can
   extend in a natural way geometrical objects away from $\mathcal{L}$ and thus get a hydrodynamical system.

\subsubsection{A connection on $T(N\mathcal{L})$} We start by identifying $T (N\mathcal{L})$ with $T \mathcal{L} \times N \mathcal{L}$ in the following way: given a path
$\gamma(s)= (p(s), N(s))$ in $N\mathcal{L}$, we identify
$$
\dot{\gamma}(0)\simeq (\partial_s p(0),\nabla_s^\perp N(0)) \in T_{p(0)} \mathcal{L} \times N_{p(0)} \mathcal{L}.
$$
(where $\nabla_s^\perp$ was defined in the previous section). 

Having identified $T (N\mathcal{L})$ with $T \mathcal{L} \times N \mathcal{L}$, we define 
a scalar product by simply adopting the natural one on $T \mathcal{L} \times N \mathcal{L}$:
$$ \langle \dot \gamma_{1}(0), \dot \gamma_{2}(0) \rangle \overset{def}{=} \langle \partial_{s}p_{1}(0), \partial_{s} p_{2}(0) \rangle + \langle \nabla_{s}^\perp N_{1}(0), 
\nabla_{s}^\perp N_{2}(0) \rangle.$$
We can then define a connection on it in the following way: given a path $(X(s),N(s))$ in $T \mathcal{L} \times N \mathcal{L}$, set
$$
\mathcal{D}_s (X,N) = (\nabla_s^\top X, \nabla_s^\perp N).
$$
This connection has the following properties:
\begin{itemize}
\item It is metric.
\item It is torsion free on $\mathcal{L}$\footnote{We naturally identify $\mathcal{L}$ with $\mathcal{L} \times \{0\} \subset N\mathcal{L}$}. In other words: given a two-parameter function $(p(t,s),N(t,s))$ in $N\mathcal{L}$,
$$
\mathcal{D}_t (\partial_s p,\nabla_s^\perp N) = \mathcal{D}_s (\partial_t p,\nabla_t^\perp N) \qquad \mbox{on $\mathcal{L}$}.
$$
This simply follows from the formula $\mathcal{D}_t (\partial_s p,\nabla_s^\perp N) - \mathcal{D}_s (\partial_t p,\nabla_t^\perp N) =(0,R^\perp (\partial_t p,\partial_s p)N)$.
\end{itemize}

\subsubsection{Exponential maps} We select from now on a point $0$ in $\mathcal{L}$ and set
$$
\mbox{for $X \in T_0\mathcal{L}$}, \quad \Phi(X) = \operatorname{exp}_0^\mathcal{L} (X) \; \in \mathcal{L}.
$$
We already saw that $D \Phi_{|0} = \operatorname{Id}$ and $\nabla D \Phi_{|0} = 0$. The next step is to define
$$
\mbox{for $(p,N) \in N\mathcal{L}$}, \quad \Psi(p,N) = \operatorname{exp}_p^\mathcal{M} N \; \in \mathcal{M}.
$$
It is easy to see that
\begin{equation}
\label{psi'}
\mbox{if $p \in \mathcal{L}$ and $(X,N) \in T_p \mathcal{L} \times N_p \mathcal{L}$}, \quad D \Psi_{(p,0)} (X,N) = X + N.
\end{equation}
We will also need the second derivative of $\Psi$: recall that it satisfies by definition
\begin{equation*}
\begin{split}
 \nabla_s D \Psi_{(p(s),n(s))} (X(s),N(s)) = &\nabla D \Psi_{(p(s),n(s))} \big((\partial_s p (s),\nabla_s^\perp n(s)),(X(s),N(s))\big) \\
& \qquad \qquad + D \Psi_{(p(s),n(s))} (\nabla_s^\top X(s),\nabla_s^\perp N(s)).
\end{split}
\end{equation*}
It is possible to compute $\nabla D \Psi$ on $\mathcal{L}$: if $p \in \mathcal{L}$ and $(X,N),(Y,M) \in T_p \mathcal{L} \times N_p \mathcal{L}$,
\begin{equation}
\label{psi''}
\nabla D \Psi_{(p,0)} \big((X,N),(Y,M)\big) = \deux^\top (X,M) + \deux^\top (Y,N) + \deux^\perp (X,Y). 
\end{equation}
In order to prove this formula, first observe that $\nabla D \Psi \big((X,N),(Y,M)\big)$ is symmetric in $(X,N)$ and $(Y,M)$; this follows from the connection $\mathcal{D}$ being torsion-free on $\mathcal{L}$.
Thus, it suffices to compute $\nabla D \Psi$ on $\big((X,0),(X,0)\big)$, $\big((0,N),(0,N)\big)$, and $\big((X,0),(0,N)\big)$. For the first of these, it suffices to differentiate $\Psi$ twice along a path of 
the type $(p(s),0)$, where $p(s)$ is a geodesic. For the second,  the same argument along a path of the type $(p,sN)$, with $N \in N_p \mathcal{L}$ holds. The third one is a bit more delicate: consider 
$D\Psi_{(p(s),0)}(0,N(s)) = N(s)$, where $N(s)$ is parallel in $N\mathcal{L}$, and differentiate in $s$.

\subsubsection{Extending the tangent and normal spaces and projectors} For any $p\in \mathcal{L}$, let $\tau_1(p), \dots \tau_d(p)$ be a basis of $T_p \mathcal{L}$; this implies that $i\tau_1, \dots i\tau_d$ is a basis of $N_p \mathcal{L}$ since $\mathcal{L}$ Lagrangian. At a point $u = \Psi(p,N)$, we define $\tau_1(u),\dots\tau_d(u)$ as follows: these vectors are given by the parallel transport of $\tau_1(p), \dots \tau_d(p)$ along the geodesic from $p$ to $u$, which reads $s \mapsto \Psi(p,sN)$. Define then
$$
\widetilde{T}_u \mathcal{L} = \operatorname{span} \{ \tau_1(u),\dots\tau_d(u) \} \quad \mbox{and} \quad \widetilde{N}_u \mathcal{L} = \operatorname{span} \{ i\tau_1(u),\dots i\tau_d(u) \}.
$$
With this definition, $\widetilde{T}_u \mathcal{L}$ and $\widetilde{N}_u \mathcal{L}$ are orthogonal, and such that $i \widetilde{T}_u \mathcal{L} = \widetilde{N}_u \mathcal{L}$. The latter property is clear, while the former follows from the fact that   for all $k$ and $l$,
$$ i_{u}\tau_k(u) \cdot \tau_l(u) = i_{p} \tau_k(p) \cdot \tau_l(p)= 0
$$
since the  $\tau_{j}$ are  parallel transported,  $\nabla i = 0$ and $\mathcal{L}$ is a Lagrangian submanifold.

Define then $P^\top$ and $P^\perp$ to be the orthogonal projectors from $T_u \mathcal{M}$ to $ \widetilde{T}_u \mathcal{L}$ and $ \widetilde{N}_u \mathcal{L}$ respectively. They satisfy
$$
P^\top + P^\perp = \operatorname{Id} \quad \mbox{and} \quad P^\perp i = i P^\top.
$$
Furthermore, we claim that 
\begin{equation}
\label{macareux}
\mbox{if $(p,N) \in N \mathcal{L}$}, \quad \nabla_N P^\top_p = 0 \quad \mbox{and} \quad \nabla_N P^\perp_p = 0.
\end{equation}
First, it suffices to prove the first of these two identities since $P^\top + P^\perp =  \operatorname{Id}$.
Next, consider $X \in T_p \mathcal{L}$, which we extend to a parallel vector field $X(s)$ along the geodesic $s \mapsto \Psi(p,sN)$. It is then easy to see that $P^{\top}_{\Psi(p,sN)} X(s) = X(s)$. Since $X(s)$ is parallel, the differentiation of this identity gives
\begin{equation}
\nonumber
\nabla_N P^{\top}_p X = 0.
\end{equation}
Similarly, if $N \in N_p \mathcal{L}$ is extended to a parallel vector field $N(s)$, the differentiation of the identity $P^{\top}_{\Psi(p,sN)} N(s) = N(s)$ yields
\begin{equation}
\nonumber
\nabla_N P^{\top}_p N = 0.
\end{equation}
We finally extend $\nabla^\top$ and $\nabla^\perp$  away from $\mathcal{L}$ by setting
\beq
\label{nablaetendu}
\nabla^\top = P^\top \nabla, \quad \nabla^\perp= P^\perp \nabla
\eeq
where $\nabla$ is the Levi-Civita connection on $\mathcal{M}.$

\subsubsection{Extending the second fundamental forms}
It is natural to extend the second fundamental forms by setting for $u \in \mathcal{M}$
\begin{equation*}
\begin{split}
& \deux^\top (X,N) = - (\nabla_X P^\top) N \qquad \mbox{if $(X,N) \in T_u \mathcal{M} \times \widetilde{N}_u \mathcal{L}$}\\
& \deux^\perp (X,Y) = - (\nabla_X P^\perp) Y \qquad \mbox{if $(X,Y) \in T_u \mathcal{M} \times \widetilde{T}_u \mathcal{L}$}
\end{split}
\end{equation*}
(notice that the second argument, $N$ or $Y$ above, is still required to be normal, respectively tangent, but not the first one). With these definitions,
\begin{equation*}
\begin{split}
& \nabla_X N = \nabla^\perp_X N + \deux^\top (X,N) \qquad \mbox{if $(X,N) \in T_u \mathcal{M} \times \widetilde{N}_u \mathcal{L}$}\\
& \nabla_X Y = \nabla^\top_X Y +  \deux^\perp (X,Y) \qquad \mbox{if $(X,Y) \in T_u \mathcal{M} \times \widetilde{T}_u \mathcal{L}$}.
\end{split}
\end{equation*}
Notice that~(\ref{macareux}) implies that
\begin{equation}
\label{rougegorge}
\mbox{if $(p,N) \in \mathcal{L} \times N_p \mathcal{L}$}, \qquad \deux^\top_p(N,\cdot) = 0 \quad \mbox{and} \quad \deux^\perp_p(N,\cdot) = 0.
\end{equation}
If $N \in \widetilde{N}_u \mathcal{L}$, with $u \notin\mathcal L$, $\deux^\top(\cdot,N)$ is not exactly symmetrical anymore on $\widetilde{T}_u \mathcal{L}$.
Namely, the following holds if $(N,X,Y) \in \widetilde{N}_u \mathcal{L} \times \widetilde{T}_u \mathcal{L} \times \widetilde{T}_u \mathcal{L}$:
$$
\langle \deux^\top(X,N),Y \rangle = \langle \deux^\top(Y,N),X \rangle - \langle N,[X,Y]\rangle.
$$
It follows from the elementary computation:
\begin{equation*}
\begin{split}
\langle \deux^\top(X,N),Y\rangle &  = \langle P^\top \nabla_X N,Y \rangle  = \langle \nabla_X N,Y \rangle = - \langle N , \nabla_X Y \rangle =  - \langle N , \nabla_Y X \rangle - \langle N , [X,Y] \rangle\\
&  = \langle \nabla_Y N,X \rangle - \langle N , [X,Y] \rangle = \langle \deux^\top(Y,N),X\rangle - \langle N , [X,Y] \rangle.
\end{split}
\end{equation*}
 Note that since $\tilde N_{u}$ is still orthogonal to $\tilde T_{u}$, 
$(N,X,Y) \mapsto \langle N, [X, Y] \rangle$ is indeed a tensor for $N \in  \tilde N_{u}$ and $X, \, Y \in\tilde T_{u}$ .
It implies in particular that, at the point $u = \Psi(p,\epsilon^2 n)$, 
if $(N,X,Y) \in \widetilde{N}_u \mathcal{L} \times \widetilde{T}_u \mathcal{L} \times \widetilde{T}_u \mathcal{L}$:
\begin{equation}
\label{almostsym} |\langle \deux^\top(X,N),Y\rangle - \langle \deux^\top(Y,N),X \rangle | \lesssim \epsilon^2 |X| |Y|.
\end{equation}

\subsubsection{Extending the tangent curvature tensor} Commuting tangential derivatives $\nabla^\top$ away from $\mathcal{L}$ also gives rise to a curvature tensor: if
$X,Y \in T \mathcal{M}$ and $Z \in \widetilde{T} \mathcal{L}$,
\beq
\label{riemtopkahl}
\nabla_X^\top \nabla_Y^\top Z - \nabla_Y^\top \nabla^\top_X Z - \nabla_{[X,Y]}^\top Z = R^\top(X,Y)Z 
\eeq
with
$$
R^\top(X,Y)Z = P^\top R(X,Y) Z + \deux^\top(Y,\deux^\perp(X,Z)) - \deux^\top(X,\deux^\perp(Y,Z)) 
$$
(recall that $R$ is the Riemannian curvature tensor of $\mathcal{M}$). This identity follows from the computation:
\begin{equation*}
\begin{split}
&\nabla_X^\top \nabla_Y^\top Z - \nabla_Y^\top \nabla_X^\top Z - \nabla_{[X,Y]}^\top Z  = P^\top \nabla_X P^\top \nabla_Y Z - P^\top \nabla_Y P^\top \nabla_X Z - \nabla_{[X,Y]}^\top Z \\
& \quad = P^\top (\nabla_X P^\top) \nabla_Y Z - P^\top (\nabla_Y P^\top) \nabla_X Z + P^\top R(X,Y) Z \\
& \quad = P^\top \nabla_Y \left[ (\nabla_X P^\top) Z \right] - P^\top \nabla_X \left[ (\nabla_Y P^\top) Z \right] + P^\top \left( \left[ \nabla_X \nabla_Y - \nabla_Y \nabla_X \right] P^\top \right) Z + P^\top R(X,Y) Z.
\end{split}
\end{equation*}
Starting from the identity $P^\top(\nabla_X \nabla_Y - \nabla_Y \nabla_X) Z = P^\top(\nabla_X \nabla_Y - \nabla_X \nabla_Y) P^\top Z$,
it is easy to see that $P^\top \left( \left[ \nabla_X \nabla_Y - \nabla_Y \nabla_X \right] P^\top \right) Z = 0$. Therefore, by definition of the second fundamental forms,
\begin{equation*}
\begin{split}
\nabla_X^\top \nabla_Y^\top Z -& \nabla_Y^\top \nabla_X^\top Z - \nabla_{[X,Y]}^\top Z = P^\top \nabla_Y \deux^\perp (X,Z) -  P^\top \nabla_X \deux^\perp (Y,Z) + P^\top R(X,Y) Z \\
& = \deux^\top (Y, \deux^\perp  (X,Z)) - \deux^\top (X, \deux^\perp  (Y,Z)) + P^\top R(X,Y) Z.
\end{split}
\end{equation*}

\subsubsection{Extending the normal curvature tensor} A computation similar to the preceding one gives: if $X,Y \in T \mathcal{M}$ and $N \in \widetilde{N} \mathcal{L}$,
\beq
\label{riemperpkahl}
\nabla_X^\perp \nabla_Y^\perp N - \nabla_Y^\perp \nabla^\perp_X N - \nabla_{[X,Y]}^\perp N = R^\perp(X,Y)N
\eeq
with
$$
R^\perp(X,Y)N =  P^\perp R(X,Y) Z + \deux^\perp(Y,\deux^\top(X,N)) - \deux^\perp(X,\deux^\top(Y,N)).
$$

\subsubsection{New coordinates for $u$}
We shall describe $u$ by coordinates $(\epsilon \phi,\epsilon^2 n) \in T_0 \mathcal{L} \times N_{\Phi(\epsilon \phi)} \mathcal{L}$:
$$
u = \Psi(p,\epsilon^2 n) \qquad \mbox{with} \qquad p = \Phi(\epsilon \phi) \qquad \mbox{or equivalently} 
\qquad u = \Psi(\Phi(\epsilon \phi),\epsilon^2 n).
$$
It will be convenient to denote
$$
D \Psi_{(p,N)} = \Sigma_{(p,N)}.
$$
Viewing $\Sigma$ as a map $N \mathcal{L} \simeq T_p \mathcal L \times N_p \mathcal L \rightarrow \widetilde{T}_u \mathcal{L} \times \widetilde{N}_u \mathcal{L}$, it is natural to adopt a block matrix notation, where the first coordinate is the tangential, and the second the normal one. Then we claim that $\Sigma$ can be written as
\begin{equation}
\label{rossignol}
\Sigma_{(p,\epsilon^2 n)} = \left( \begin{array}{cc} S_{\top \top}(\Phi, \epsilon^2 n) & \epsilon^4 S_{\top \perp}(\Phi, n) \\ \epsilon^4 S_{\perp \top }(\Phi, n) &  S_{\perp \perp} (\Phi,\epsilon^2 n) \end{array} \right)
\end{equation}
where $S_{\top \top}$, $S_{\top \perp}$, $S_{\perp \top }$, $S_{\perp \perp}$ depend smoothly on their arguments, with bounds uniform in $\epsilon$. To check that $P^\top \Sigma P^\perp$ is indeed $O(\epsilon^4)$, set
$$
G(s) = P^\top_{\Psi(p,sN)} \Sigma_{(p,sN)} \left( \begin{array}{l} 0 \\ N \end{array} \right).
$$
It follows from~\eqref{psi'} that $G(0) = 0$. Furthermore, by~\eqref{psi''},
$$
\nabla_s G(s)_{|s=0} = \nabla_N P^\top_p \Sigma_p \left( \begin{array}{l} 0 \\ N \end{array} \right) + P^\top_p  \nabla_N \Sigma  \left( \begin{array}{l} 0 \\ N \end{array} \right) = 0.
$$
Therefore, $G(s) = O(s^2)$, giving the desired result. A similar argument gives that $P^\perp \Sigma P^\top$ is $O(\epsilon^4)$.

\bigskip
Further properties of $S_{\top \top}$ and $S_{\perp \perp}$ that will be useful  are that for every $p \in \mathcal{L}$
 and for every $X \in T_{p}\mathcal{L}$, $N \in N_{p} \mathcal{L}$, we have
\begin{equation}
\label{devDPsi}
 (S_{\top \top})_{p}= Id, \,  (S_{\perp \perp})_{p} = Id, \, (\nabla^\perp S_{\perp \perp})_{p} = 0, \quad (\nabla_{X}^\top S_{\top \top})_{p} = 0, \, 
 ( \nabla_{N}^\top S_{\top \top})_{p}= \deux^\top(\cdot, N).
\end{equation}

Let us prove the above properties. We start with the properties of $S_{\perp \perp}$. Let us recall that by definition
 $S_{\perp \perp}(p,sN) Z=  P^\perp(\Phi(p, sN)) D \Psi_{(p, sN)} (0, Z)$ for $s$ sufficiently small and $N,$  $Z \in N_{p} \mathcal{L}.$
  For $s=0$, we get from \eqref{psi'} that $S_{\perp \perp} Z = Z$.
   This also implies that $\nabla_{X}^\perp S_{\perp \perp} = 0$.  By applying $\nabla_{s}^\perp$ and taking the value
  at $s=0$, we  also obtain that
  $$ (\nabla_{N}^\perp S_{\perp \perp})_{p} Z = ( \nabla_{N} P^\perp) Z + P^\perp \nabla D\Psi_{(p, 0)}(0, Z), (0, N))= 0$$
   by using \eqref{psi''} and \eqref{macareux}.
   In a similar way,  since
   $ S_{\top \top}(p, sN) Y=  P^\top(\Phi(p, sN)) D \Psi_{(p, sN)} (Y, 0)$ for $s$ sufficiently small and $N \in N_{p} \mathcal{L}
   , \, Y \in T_{p} \mathcal{L},$ we also find that
   $ (S_{\top \top})_{p} Y = Y$ and hence that $(\nabla_{X}^\top S_{\top \top})_{p} = 0$.
    By taking $\nabla_{s}^\top$ at $s=0$, we obtain that
    $$( \nabla_{N}^\top S_{\top \top} )_{p}Y= P^\top \nabla D\Psi_{(p,0)} ((Y,0), (0, N)) = \deux^\top(Y, N).$$ 
   
\subsubsection{Computing $V'(u)$} Recall that in this section, we assume that  $V(u) = \lambda \operatorname{dist}(u,\mathcal{L})^2$, or in other words
$$
V(\Psi(p,N)) = \lambda |N|^2.
$$
To compute $V'(u)$, consider a path $(p(s),N(s))$ such that $p(s) \in \mathcal{L}$, $N(s) \in N_{p(s)} \mathcal{L}$, and let $u(s) = \Psi(p(s),N(s))$.
Differentiate then
$$
\partial_s V(u(s)) =  \lambda \partial_s  |N(s)|^2
= 2\lambda \left\langle \nabla^\perp_s N(s), N(s) \right\rangle.
$$
Recall that $\partial_s u = D\Psi \left( \begin{array}{l} \partial_s p \\ \nabla_s^\perp N(s) \end{array} \right)$, therefore 
$\left( \begin{array}{l} \partial_s p \\ \nabla_s^\perp N(s) \end{array} \right) = (D\Psi)^{-1} \partial_s u$ and, coming back
to the above,
$$
\partial_s V(u(s)) = 2\lambda \left\langle  (D\Psi)^{-1} \partial_s u,  \left( \begin{array}{l} 0 \\ N(s) \end{array} \right) \right\rangle
= 2\lambda \left\langle  \partial_s u,  (D\Psi)^{-*} \left( \begin{array}{l} 0 \\ N(s) \end{array} \right) \right\rangle
$$
(using the notation $ (D\Psi)^{-*}$ as a shorthand for $\left((D\Psi)^{-1}\right)^*$)
which means that
\begin{equation}
\label{bassan}
V'(\Psi(p,N)) = 2\lambda (D\Psi)^{-*}_{(p,N)} \left( \begin{array}{l} 0 \\ N \end{array} \right) \quad \mbox{and} 
\quad V'(\Psi(p,\epsilon^2 n)) = 2\lambda \epsilon^2 (D\Psi)^{-*}_{(p,\epsilon^2 n)} \left( \begin{array}{l} 0 \\ n \end{array} \right).
\end{equation}
We claim that
\begin{equation}
\begin{split}
\label{chardonneret}
& D\Psi_{(p,\epsilon^2 n)}\left( \begin{array}{l} 0 \\ n \end{array} \right) = (D\Psi)^{-*}_{(p,\epsilon^2 n)}\left( \begin{array}{l} 0 \\ n \end{array} \right) + O(\epsilon^4).
\end{split}
\end{equation}
To prove this, let for $p\in \mathcal{L}$, and $n,N \in N_p \mathcal{L}$ 
$$
G(s) = D\Psi_{(p,s n)} \left( \begin{array}{l} 0 \\ n \end{array} \right) - 
(D\Psi)^{-*}_{(p,s n)} \left( \begin{array}{l} 0 \\ n \end{array} \right).
$$
It satisfies $G(0)=0$ and 
$$
G'(0)= \nabla_N D\Psi \left( \begin{array}{l} 0 \\ n \end{array} \right) - (D\Psi)^{-*} (\nabla_N D\Psi)^* (D\Psi)^{-*}  \left( \begin{array}{l} 0 \\ n \end{array} \right) = 0-0=0.
$$
Furthermore, it is easy to see that $\langle D\Psi_{(p,s n)}  \left( \begin{array}{l} 0 \\ n \end{array} \right)\,,\,\tau_k (s) \rangle = 0$; this quantity is indeed zero for $s=0$, and has a zero derivative (in $s$).
Coming back to $V'(u)$, we obtain
\begin{equation}
\label{formulaVprime}
\begin{split}
& V'(\Psi(p,\epsilon^2 n)) = 2\lambda \epsilon^2\Sigma \left( \begin{array}{l} 0 \\ n \end{array} \right) + \epsilon^6 R_V(p,n) \\
& P^\top V'(\Psi(p,\epsilon^2 n)) = \epsilon^6 R_V^\top(p,n) \\
&  P^\perp V'(\Psi(p,\epsilon^2 n)) = 2\lambda \epsilon^2 S_{\perp \perp}(p, \epsilon^2n) n + \epsilon^6  R_V^\perp (p,n)
\end{split}
\end{equation}
 where $R_V$ is a smooth function of $p$ and $n$ and we have set 
 $$R_{V}^\top(p,n)= P^\top R_{V}(p,n) + 2 \lambda S_{\top \perp}(p, n), \quad R_{V}^\perp (p, n) = P^\perp R_{V}(p,n)$$
 by using \eqref{rossignol}.

\subsubsection{Action of $iB$}
 We shall also need to describe the action of $ i B \Sigma_{(p, \epsilon^2n)}$. We can write it in block matrix form
\begin{equation}
\label{Baction} iB \Sigma_{(p, \epsilon^2n)}= \left( \begin{array}{cc} (iB)_{\top\top} &  \epsilon^4  (iB)_{\top \perp} \\  \epsilon^4(iB)_{\perp \top} &  (iB)_{\perp\perp} \end{array} \right),
\end{equation}
 where we have set 
 \begin{align*}
 & ( iB)_{\top\top }(p, \epsilon^2n)X = P^\top (iB) \Sigma_{(p, \epsilon^2n)} (X, 0), \quad (iB)_{\perp\perp}(p, \epsilon^2 n) N = P^\perp (iB) \Sigma_{(p, \epsilon^2n)} (0,N),  \\
  & (iB)_{\top \perp}(p, n) N =  {1 \over \epsilon^4}P^\top ( iB) \Sigma_{(p, \epsilon^2n)}(0,N),  \quad   (iB)_{\perp \top}(p, n) N= {1\over \epsilon^4} P^\perp (iB) \Sigma_{(p, \epsilon^2n)} (X, 0),  \end{align*}
 for every $ (X, N) \in T_{\Phi} \mathcal{L} \times N_{\Phi}\mathcal{L}.$

 Again, the tensors $(iB)_{\top \top}$, $(iB)_{\perp \perp}$,... are smooth with derivatives uniformly bounded in $\epsilon$. This is due to the fact that 
 $$ P^\top (iB) \Sigma \left( \begin{array}{c} 0 \\ N \end{array} \right)= O(\epsilon^4) N, \quad P^\perp (iB) \Sigma  \left( \begin{array}{c} X \\ 0 \end{array} \right)= O(\epsilon^4) X.$$
  Indeed, for the first identity, this follows again by setting
    $$ G(s) =   \left(P^\top (iB) \Sigma \right)_{(p, sn)} \left( \begin{array}{c} 0 \\ N \end{array} \right)$$
and by noticing that $G(0)= 0$ and $\nabla_s G_{|0}= 0$ since  $\nabla i=0$, $\nabla B= 0$ on $\mathcal{L}$, 
 $ \nabla_{n} P^\top = 0 $ (thanks to \eqref{macareux}),    and by using \eqref{psi''}. The second one can be obtained following the same lines.
 
\bigskip

Finally, we observe that  on $\mathcal{L}$, that is  to say,  for every $p \in \mathcal{L}$
 and for every $Y \in T_{p}\mathcal{L}$, $N \in N_{p} \mathcal{L}$, we have that
\begin{align}
& \label{iBTTL} \nabla^\top_{Y} (iB)_{\top \top} = 0, \quad \nabla^\top_{N}( iB)_{\top \top}=  \deux^\top(iB \cdot, N), \\
& \label{iBPPL} \nabla^\perp (iB)_{\perp \perp} = 0.
\end{align}
 Indeed, by using the definitions of these tensors, we obtain that on $ \mathcal{L}$, and for every
  $X, \, Y \in T_{p} \mathcal{L}$ and $N \in N_{p}\mathcal{L}$, prolonging $X$ to be a parallel vector field,
   $$  ( \nabla^\top_{Y} (iB)_{\top \top} )X =( \nabla_{Y}^\top P^\top)( iB X) =0$$
  and 
  $$(\nabla^\top_{N} (iB)_{\top \top} )X = (\nabla_{N}^\top P^\top) (iB X) = \deux^\top(iB Y, N)$$
   thanks to \eqref{macareux},  \eqref{psi''} and  Corollary \ref{BdeuxT}.
 
  The proof of the second identity follows from the same arguments.
  

\subsection{The hydrodynamical system}
\label{sectionhydrokahler} First of all, we will work under the

\medskip

\noindent \underline{Bootstrap hypothesis.} As in Section~\ref{sectioneuclid}, we work on an interval of time $[0, T^\epsilon]$ such that the estimate \eqref{hypapriori} is satisfied for some $r$ sufficiently small. Note that this ensures that  our exponential coordinate system provides a nice parametrization of the solution.

\medskip

Written in the $(\Phi,n)$ coordinates, \eqref{SMK} reads
\begin{multline}
\label{Keq1} \Sigma_{(\Phi, \epsilon^2n)} \Big(\partial_{t} \Phi - {c \over \epsilon^2} \partial_{x} \Phi, \epsilon^2 (\nabla_{t}^\perp n - {c \over \epsilon^2} \nabla_{x}^\perp n ) \Big)
 - {1 \over \epsilon^2} i B \Sigma_{(\Phi, \epsilon^2n)} (\partial_{x} \Phi, \epsilon^2 \nabla_{x}^\perp n ) \\
= i \Big( {1 \over 2 \epsilon} \nabla_{x}\big(\Sigma_{(\Phi, \epsilon^2 n )} \big( \partial_{x} \Phi, \epsilon^2 \nabla_{x}^\perp n \big) \big) - {1\over \epsilon^3} V'(u) \Big).
\end{multline}
In order to write a system, we rely on the decomposition of $\Sigma$ obtained in~(\ref{rossignol}).

First take the tangential component of the above to obtain
\begin{align*} 
& S_{\top \top}(\Phi,\epsilon^2 n) \big(\partial_{t} - {c \over \epsilon^2} \partial_{x} \big){ \Phi \over \epsilon}  + \epsilon^5 S_{\top \perp} (\Phi, n)  \big(\nabla_{t}^\perp n - {c \over \epsilon^2} \nabla_{x}^\perp n \big)  - {1 \over \epsilon^2} (iB)_{\top\top } {\partial_{x} \Phi \over \epsilon} - \epsilon^3 (iB)_{\top \perp} \nabla_{x}^\perp n \\
&  = i \Big( {1 \over 2  } \nabla_{x}^\perp \big( S_{\perp \perp}(\Phi,\epsilon^2 n)
\nabla_{x}^\perp n \big)+  { 1 \over 2 \epsilon^2 } \nabla_{x}^\perp \Big( P^\top \Sigma_{(\Phi, \epsilon^2n)} \cdot \big( \partial_{x} \Phi, \epsilon^2 \nabla_{x}^\perp n \big) \Big) \\
& \qquad \qquad \qquad \qquad  \qquad  \qquad   \qquad  \qquad    \qquad \qquad + \frac{\epsilon^3}{2} \nabla_x^{\perp} \left(  S_{\perp \top}(\Phi,\epsilon^2 n) \frac{\partial_x \Phi}{\epsilon} \right)
- {1 \over \epsilon^4}P^\perp V'(u) \Big).
\end{align*} 
Next, observe that
\begin{equation*}
\begin{split}
&{ 1 \over 2 \epsilon^2 } \nabla_{x}^\perp \Big( P^\top \Sigma_{(\Phi, \epsilon^2n)} \big( \partial_{x} \Phi, \epsilon^2 \nabla_{x}^\perp n \big) \Big) \\
& \qquad \qquad \qquad = {1 \over 2 } \deux^\perp \Big(\Sigma_{(\Phi, \epsilon^2n)} \cdot \big( {\partial_{x} \Phi \over \epsilon }, \epsilon \nabla_{x}^\perp n \big), 
P^\top \Sigma_{(\Phi, \epsilon^2n)} \cdot \big( {\partial_{x} \Phi \over \epsilon }, \epsilon \nabla_{x}^\perp n \big) \Big),
\end{split}
\end{equation*}
while formula~\eqref{formulaVprime} gives the expression of $P^\perp V'(u)$.
We thus get the equation for the tangential component:
\begin{align*}
& S_{\top \top}\big(\partial_{t} - {c \over \epsilon^2} \partial_{x} \big){ \Phi \over \epsilon}  - {1 \over \epsilon^2} (iB)_{\top\top} \partial_{x} {\Phi \over \epsilon} + \epsilon^5 S_{\top \perp}  \big(\nabla_{t}^\perp n - {c \over \epsilon^2} \nabla_{x}^\perp n \big) - \epsilon^3 (iB)_{\top \perp} \nabla_{x}^\perp n  \\
& \qquad  = i \left( {1 \over 2  } \nabla_{x}^\perp \big( S_{\perp \perp} \nabla_{x}^\perp n \big) - \frac{2\lambda}{\epsilon^2} S_{\perp \perp} n +  {1 \over 2} \deux^\perp\left(  \Sigma \left(\begin{array}{c} {\partial_{x} \Phi \over \epsilon} \\ \epsilon \nabla_{x}^\perp n 
\end{array} \right), P^\top   \Sigma \left(\begin{array}{c} {\partial_{x} \Phi \over \epsilon} \\ \epsilon \nabla_{x}^\perp n 
\end{array} \right) \right) \right.\\
& \left. \qquad \qquad \qquad \qquad  \qquad \qquad \qquad + \frac{\epsilon^3}{2} \nabla_x^{\perp} \left( S_{\perp \top} \frac{\partial_x \Phi}{\epsilon}\right) + \epsilon^2  R_V^\perp \right). 
\end{align*}

In  a similar way, we can get the equation for the normal component. Combining it with the above, we find the hydrodynamic form of \eqref{Keq1}:
\begin{equation}
\label{eqhydroK}
\left\{ 
\begin{array}{l} 
\displaystyle 
 S_{\top \top}\big(\partial_{t} - {c \over \epsilon^2} \partial_{x} \big){ \Phi \over \epsilon}  - {1 \over \epsilon^2} (iB)_{\top\top} \partial_{x} {\Phi \over \epsilon} + \epsilon^5 S_{\top \perp}  \big(\nabla_{t}^\perp n - {c \over \epsilon^2} \nabla_{x}^\perp n \big) - \epsilon^3 (iB)_{\top \perp} \nabla_{x}^\perp n  \\
\displaystyle  \qquad  = i \left( {1 \over 2  } \nabla_{x}^\perp \big( S_{\perp \perp} \nabla_{x}^\perp n \big) - \frac{2\lambda}{\epsilon^2} S_{\perp \perp} n +  {1 \over 2} \deux^\perp\left(  \Sigma \left(\begin{array}{c} {\partial_{x} \Phi \over \epsilon} \\ \epsilon \nabla_{x}^\perp n 
\end{array} \right), P^\top   \Sigma \left(\begin{array}{c} {\partial_{x} \Phi \over \epsilon} \\ \epsilon \nabla_{x}^\perp n 
\end{array} \right) \right) \right.\\
\displaystyle  \left. \qquad \qquad \qquad \qquad  \qquad \qquad \qquad + \frac{\epsilon^3}{2} \nabla_x^{\perp}  (S_{\perp \top} \frac{\partial_x \Phi}{\epsilon}) + \epsilon^2  R_{V}^\perp \right). 
 \\
\displaystyle S_{\perp \perp} \big( \nabla_{t}^\perp - {c \over \epsilon^2} \nabla_{x}^\perp \big) n -{1 \over \epsilon^2}(iB)_{\perp\perp} \nabla_{x}^\perp n   - \epsilon (iB)_{\perp \top} D\Phi \partial_{x}\phi + \epsilon^3 S_{\perp \top} (\partial_t - \frac{c}{\epsilon^2} \partial_x) \frac{\Phi}{\epsilon} \\
\displaystyle \quad  = i \left( {1 \over 2 \epsilon^2} \nabla_{x}^\top \left( S_{\top\top} {\partial_{x} \Phi \over \epsilon}
   + \epsilon^5   S_{\top \perp} \nabla_{x}^\perp n   \right) 
    + {1 \over 2} \deux^\top \left( \Sigma  \left( \begin{array}{cc} {\partial_{x} \Phi \over \epsilon} \\ \epsilon \nabla_{x}^\perp n \end{array} \right) , 
     P^\perp\Sigma \left( \begin{array}{cc} {\partial_{x} \Phi \over \epsilon} \\ \epsilon \nabla_{x}^\perp n \end{array} \right)  \right) + \epsilon  R_{V}^\top \right). 
\end{array}
\right.
\end{equation}

As a first consequence of this hydrodynamic formulation and \eqref{hypapriori}, we easily  get the following lemma.

\begin{lem} For $s \geq 2$, we have  on $[0, T^\epsilon]$ the estimates
\begin{align}
& \label{dtphikahl}
\epsilon^2 \| \partial_{t} \phi(t) \|_{s-1} = O(\mathcal{E}_s) \\
& \label{dxxxphikahl}
\epsilon \| \partial_{xxx} \phi \|_{s-1} =  O(\mathcal{E}_s) 
\end{align}
\end{lem}

\subsection{Estimates on $u$} 

\begin{prop}
\label{propenergieK1}
The following a priori estimate holds on $[0, T^\epsilon]$:
$$
\mathcal{E}_{s,1}(u,t) \lesssim \mathcal{E}_{s,1}(u,0) + \int_0^t O(\mathcal{E}_{s}(u,\mathcal{\tau}) \,d\tau.
$$
\end{prop}

\begin{proof}
\underline{Step 1: first decomposition.}
Start with the equation satisfied by $u$
$$
(\epsilon^2 \partial_t - (c+iB) \partial_x ) u = i \left( \frac{1}{2} \epsilon \nabla_x \partial_x u - \frac{1}{\epsilon} V'(u) \right).
$$
which we write
$$
LHS = RHS1 + RHS2.
$$
The plan is the following: apply $(\epsilon^2 \nabla_t - (c+Bi) \nabla_x) \nabla^m$ to the above, take the scalar product with 
$\frac{1}{\epsilon^2}(\epsilon^2 \nabla_t - c \nabla_x ) \nabla^{m-1} \partial u$,
and then estimate the resulting terms. The first such term is the following; we will use integration by parts and commutation of derivatives to estimate it.
\begin{equation}
\label{pinson1}
\begin{split}
& \frac{1}{\epsilon^2} \int (\epsilon^2 \nabla_t - (c+Bi) \nabla_x) \nabla^m LHS \cdot (\epsilon^2 \nabla_t - c \nabla_x ) \nabla^{m-1} \partial u\,dx \\
& \qquad \qquad = \frac{1}{\epsilon^2} \int (\epsilon^2 \nabla_t - (c+Bi) \nabla_x) \nabla^m (\epsilon^2 \partial_t - (c+iB) \partial_x ) u \cdot (\epsilon^2 \nabla_t - c \nabla_x ) \nabla^{m-1} \partial u\,dx \\
& \qquad \qquad =  \frac{d}{dt} \left[\frac{1}{2} \int |(\epsilon^2 \nabla_t - c \nabla_x )\nabla^{m-1} \partial u|^2 \,dx - \frac{\mu}{2} \int |\nabla_x \nabla^{m-1} \partial u|^2\,dx \right] + O(\mathcal{E}_s^2).
\end{split}
\end{equation}
Here, the term $O(\mathcal{E}_s^2)$ results from commutators arising when commuting $\nabla_t$ and $\nabla_x$ derivatives. Typical instances are
\begin{align*}
& \frac{1}{\epsilon^2} \int (\epsilon^2 \nabla_t - (c+Bi) \nabla_x) \nabla^{m-2} R(\epsilon^2 \partial_t u , \partial_x u) \partial u \cdot (\epsilon^2 \nabla_t - c \nabla_x ) \nabla^{m-1} \partial u\,dx\\
&  \frac{1}{\epsilon^2} \int (\epsilon^2 \nabla_t - (c+Bi) \nabla_x) \nabla^{m-1} i \nabla B (\nabla u,\partial_x u) \cdot (\epsilon^2 \nabla_t - c \nabla_x ) \nabla^{m-1} \partial u\,dx;
\end{align*}
simply counting derivatives and using~\eqref{macareux}, it becomes clear that it can be estimated by $O(\mathcal{E}_s^2)$. 
Next, we turn to the term involving $RHS1$ and transform it by first commuting derivatives, and then using the
equation satisfied by $u$:
\begin{equation*}
\begin{split}
(\epsilon^2 \nabla_t - (c+Bi) \nabla_x) \nabla^m RHS1 & = (\epsilon^2 \nabla_t - (c+Bi) \nabla_x) \nabla^m \frac{i}{2} \epsilon \nabla_x \partial_x u \\
& = \frac{i \epsilon}{2} \nabla_x^2 \nabla^m (\epsilon^2 \partial_t - (c+iB) \partial_x ) u+ \epsilon^3 O_{L^2} (\mathcal{E}_s)\\
& = \frac{i \epsilon}{2} \nabla_x^2 \nabla^m i \left( \frac{1}{2} \epsilon \nabla_x \partial_x u - \frac{1}{\epsilon} V'(u) \right) + \epsilon^3 O_{L^2} (\mathcal{E}_s).
\end{split}
\end{equation*}
(once again, commutators are easily dealt with). We now take the scalar product with $\frac{1}{\epsilon^2}(\epsilon^2 \nabla_t - c \nabla_x ) \nabla^{m-1} \partial u$ and once again commute derivatives and use the
equation satisfied by $u$ to obtain
\begin{equation}
\begin{split}
\label{pinson2}
& \frac{1}{\epsilon^2} \int (\epsilon^2 \nabla_t - (c+Bi) \nabla_x) \nabla^m RHS1 \cdot (\epsilon^2 \nabla_t - c \nabla_x ) 
\nabla^{m-1} \partial u\,dx \\
& \qquad \qquad = -\frac{1}{4} \int \nabla_x^2 \nabla^m \nabla_x^2 u \cdot (\epsilon^2 \nabla_t - c \nabla_x ) \nabla^{m-1} \partial u\,dx \\
& \qquad \qquad \qquad \qquad + \frac{1}{2\epsilon^2} \int \nabla^m \nabla_x^2 V'(u) \cdot (\epsilon^2 \nabla_t - c \nabla_x ) \nabla^{m-1} \partial u + O(\mathcal{E}_s^2) \\
& \qquad \qquad = -\frac{\epsilon^2}{8} \frac{d}{dt} \int \left|\nabla^m \nabla_x \partial_x u\right|^2\,dx
+ \frac{1}{2\epsilon^2} \int \nabla^m \nabla_x^2 V'(u) \cdot (\epsilon^2 \nabla_t - c \nabla_x ) \nabla^{m-1} \partial u \,dx 
+ O(\mathcal{E}_s^2).
\end{split}
\end{equation}
Finally, the contribution of $RHS2$ is simply
\begin{equation*}
\begin{split}
\frac{1}{\epsilon^2}  \int (\epsilon^2 \nabla_t &- (c+Bi) \nabla_x)  \nabla^m RHS2 \cdot (\epsilon^2 \nabla_t - c \nabla_x ) 
\nabla^{m-1} \partial u\,dx \\
& \qquad \qquad = - \frac{1}{\epsilon^3} \int (\epsilon^2 \nabla_t - (c+Bi) \nabla_x) \nabla^m i V'(u) \cdot 
(\epsilon^2 \nabla_t - c \nabla_x ) \nabla^{m-1} \partial u\,dx.
\end{split}
\end{equation*}
Putting everything together gives
\begin{equation*} \begin{split}
& \frac{d}{dt} \left[\frac{1}{2} \int |(\epsilon^2 \nabla_t - c \nabla_x )\nabla^{m-1} \partial u|^2 \,dx 
- \frac{\mu}{2} \int |\nabla^{m} \partial_x u|^2 \,dx
+ \frac{\epsilon^2}{8} \int \left|\nabla^m \nabla_x \partial_x u\right|^2\,dx \right] \\
& \;\;\;\;\;\;= \underbrace{\frac{1}{2\epsilon^2} \int \nabla^m \nabla_x^2 V'(u) 
\cdot (\epsilon^2 \nabla_t - c \nabla_x ) \nabla^{m-1} \partial u \,dx}_{I} \\
& \;\;\;\;\;\;\;\;\;\;\;\; - \underbrace{\frac{1}{\epsilon^3} \int (\epsilon^2 \nabla_t - (c+Bi) \nabla_x) \nabla^m i V'(u) \cdot 
(\epsilon^2 \nabla_t - c \nabla_x ) \nabla^{m-1} \partial u\,dx}_{II} + O(\mathcal{E}_s^2)
\end{split} \end{equation*}
We decompose further $I$ and $II$ by distinguishing for each scalar product between the tangent and normal part:
\begin{equation*} \begin{split}
& I = \underbrace{\frac{1}{2\epsilon^2} \int P^\top \nabla^m \nabla_x^2 V'(u) 
\cdot P^\top (\epsilon^2 \nabla_t - c \nabla_x ) \nabla^{m-1} \partial u \,dx}_{Ia}
+ \underbrace{\frac{1}{2\epsilon^2} \int P^\perp \dots \; \cdot P^\perp \dots \,dx}_{Ib} \\
& II = \underbrace{\frac{1}{\epsilon^3} \int P^\top(\epsilon^2 \nabla_t - (c+Bi) \nabla_x) \nabla^m i V'(u) \cdot 
P^\top (\epsilon^2 \nabla_t - c \nabla_x ) \nabla^{m-1} \partial u\,dx}_{IIa}
+ \underbrace{\frac{1}{\epsilon^3} \int P^\perp \dots \; \cdot P^\perp \dots \,dx}_{IIb}.
\end{split}
\end{equation*}
We will in the following examine separately $Ia$, $Ib$, $IIa$ and $IIb$.

\bigskip

\noindent
\underline{Step 2: estimating $Ia$.} First, ~\eqref{formulaVprime} gives
$$
\frac{1}{\epsilon^2} \nabla^m \nabla_x^2 V'(u) = 2 \lambda \nabla^m \nabla_x^2 \Sigma \left( \begin{array}{l} 0 \\ n \end{array} \right)+ O_{L^2} (\mathcal{E}_s)  .
$$
Due to the vanishing of the second fundamental form on the normal bundle~(\ref{rougegorge}),
\begin{equation*}
\begin{split}
P^\top \nabla^m \nabla_x^2 \Sigma \left( \begin{array}{l} 0 \\ n \end{array} \right)
& = \deux^\top \left( \nabla^m \nabla_x \partial_x u
, \Sigma \left( \begin{array}{l} 0 \\ n \end{array} \right) \right) + O_{L^2} (\mathcal{E}_s) \\
& = \deux^\top \left( \nabla^{\top m} \nabla_x^\top P^\top \partial_x u ,
 \Sigma \left( \begin{array}{l} 0 \\ n \end{array} \right) \right) + O_{L^2} (\mathcal{E}_s).
\end{split}
\end{equation*}
Combining the two previous equalities, we obtain
$$
P^\top \frac{1}{\epsilon^2} \nabla^m \nabla_x^2 V'(u) = 2 \lambda \deux^\top \left( \nabla^{\top m} \nabla_x^\top P^\top \partial_x u ,
 \Sigma \left( \begin{array}{l} 0 \\ n \end{array} \right) \right) + O_{L^2} (\mathcal{E}_s).
$$
On the other hand,
$$
P^\top (\epsilon^2 \nabla_t - c \nabla_x) \nabla^m u 
= (\epsilon^2 \nabla_t^\top - c \nabla_x^\top)\nabla^{\top \; m-1} P^\top \partial u
+ \epsilon O_{\mathcal H^1} (\mathcal{E}_s).
$$
Taking the scalar product of the two last equalities,
\begin{equation*} \begin{split}
Ia & = \lambda \int \deux^\top \left( \nabla^{\top m} \nabla_x^\top \partial_x u , \Sigma \left( \begin{array}{l} 0 \\ n \end{array} \right) \right) \cdot (\epsilon^2 \nabla_t^\top - c \nabla_x^\top)\nabla^{\top \; m-1} P^\top \partial u \,dx 
+ O(\mathcal{E}_s^2) \\
& = - \frac{\epsilon^2 \lambda}{2} \frac{d}{dt} \int \deux^\top \left( \nabla^{\top m} P^\top \partial_x u , \Sigma \left( \begin{array}{l} 0 \\ n \end{array} \right) \right) \cdot \nabla^{\top \; m} P^\top \partial_x u \,dx
+ O(\mathcal{E}_s^2)
\end{split}
\end{equation*}
where we used for the last equality the almost symmetry property~(\ref{almostsym}).

\bigskip

\noindent
\underline{Step 3: estimating $Ib$.} First observe that, on the one hand,~\eqref{formulaVprime} gives
\begin{equation*}
\begin{split}
P^\perp \nabla^m \nabla^2_x V'(u) & = 2\lambda \epsilon^2  P^\perp \nabla^m \nabla_x^2 \Sigma \left( \begin{array}{l} 0 \\ n \end{array} \right) + O_{L^2}(\mathcal{E}_s) \\ 
& = 2\lambda \epsilon^2 \nabla^{\perp m} \nabla_x^{\perp 2}  \Sigma \left( \begin{array}{l} 0 \\ n \end{array} \right) + O_{L^2} (\mathcal{E}_s),
\end{split}
\end{equation*}
while on the other hand, by~\eqref{psi''} and~\eqref{rossignol},
\begin{equation*}
\begin{split}
P^\perp (\epsilon^2 \nabla_t - c \nabla_x) \nabla^m u & =  (\epsilon^2 \nabla^\perp_t - c \nabla^\perp_x) \nabla^{\perp m-1}
P^\perp \Sigma \left( \begin{array}{l} \epsilon \partial \phi \\ \epsilon^2 \nabla^\perp n \end{array} \right) + 
\epsilon O_{\mathcal H^1} (\mathcal{E}_s) \\
& =  (\epsilon^2 \nabla^\perp_t - c \nabla^\perp_x) \nabla^{\perp m-1} \Sigma \left( \begin{array}{l} 0 \\ \epsilon^2 \nabla^\perp n \end{array} \right) + \epsilon O_{\mathcal H^1} (\mathcal{E}_s) + \epsilon^3 O_{L^2} (\mathcal{E}_s) \\
& = (\epsilon^2 \nabla^\perp_t - c \nabla^\perp_x) \nabla^{\perp m} \Sigma \left( \begin{array}{l} 0 \\ \epsilon^2 n \end{array} \right) + \epsilon O_{\mathcal H^1} (\mathcal{E}_s) + \epsilon^3 O_{L^2} (\mathcal{E}_s).
\end{split}
\end{equation*}
Taking the scalar product of the two gives, after integrating by parts in $x$, and commuting derivatives,
\begin{equation*} \begin{split}
Ib & = \lambda \int \nabla^{\perp m} \nabla_x^{\perp 2} \Sigma \left( \begin{array}{l} 0 \\ n \end{array} \right)
\cdot (\epsilon^2 \nabla^\perp_t - c \nabla^\perp_x) \nabla^{\perp m} \Sigma \left( \begin{array}{l} 0 \\ \epsilon^2 n \end{array} \right) \,dx + O(\mathcal{E}_s^2)\\
& = - \frac{\lambda \epsilon^4}{2} \frac{d}{dt} \int \left| \nabla^{\perp m} \nabla_x^\perp
\Sigma \left( \begin{array}{l} 0 \\  n \end{array} \right) \right|^2 \, dx+ O(\mathcal{E}_s^2).
\end{split}
\end{equation*}

\bigskip

\noindent
\underline{Step 4: estimating $IIa$.} On the one hand, it is easy to see from formula~\eqref{formulaVprime} that
\begin{equation}
\label{mallard1}
P^\top (\epsilon^2 \nabla_t - (c+Bi) \nabla_x) \nabla^m i V'(u) = 2 \lambda i ( \epsilon^2 \nabla_t^\top - (c + iB) \nabla_x^\perp ) \nabla^{\perp m} \Sigma \left( \begin{array}{l} 0 \\  \epsilon^2n \end{array} \right) + \epsilon^3 O_{L^2}(\mathcal{E}_s).
\end{equation}
On the other hand, it follows from the equation~\eqref{SMK} satisfied by $u$ that
\begin{align*}
\nabla^{\perp m} \Sigma \left( \begin{array}{c} 0 \\ (\epsilon^2 \nabla_t^\top - (c+iB) \nabla_x^\top) \epsilon^2 n \end{array} \right) = \frac{i \epsilon}{2} \nabla^{\top m}  \nabla_x^\top P^\top \partial_x u + \epsilon^3 O_{L^2} (\mathcal{E}_s).
\end{align*}
If $p\in \mathcal{L}$, and $N,N' \in N\mathcal L$, denote $G(s) = \Sigma_{(p,sN)} \left( \begin{array}{c} 0 \\ iB N' \end{array} \right) - iB \Sigma_{(p,sN)} \left( \begin{array}{c} 0 \\ N' \end{array} \right)$, one checks thanks to~\eqref{psi'} and~\eqref{psi''} that $G(0) = 0$ and $\nabla_s G(s)_{|s=0} = 0$. Therefore, $\Sigma_{(p,\epsilon^2 n)} \left( \begin{array}{c} 0 \\ iB N \end{array} \right) = iB \Sigma_{(p,\epsilon^2 n)} \left( \begin{array}{c} 0 \\  N \end{array} \right) + O(\epsilon^4)$. Combined with the last equality, this leads to
\begin{equation}
\label{mallard2}
(\epsilon^2 \nabla_t^\top - (c+iB) \nabla_x^\top) \nabla^{\perp m} \Sigma \left( \begin{array}{c} 0 \\  \epsilon^2 n \end{array} \right) = \frac{i \epsilon}{2} \nabla^{\top m}  \nabla_x^\top P^\top \partial_x u + \epsilon^3 O_{L^2} (\mathcal{E}_s).
\end{equation}
The equalities~\eqref{mallard1} and~\eqref{mallard2} lead to
$$
P^\top (\epsilon^2 \nabla_t - (c+Bi) \nabla_x) \nabla^m i V'(u) = - \lambda \nabla^{\top m}  \nabla_x^\top P^\top \partial_x u + \epsilon^3 O_{L^2} (\mathcal{E}_s).
$$
Using this last equality together with
$$
P^\top ( \epsilon^2 \nabla_t - c \nabla_x) \nabla^m u = ( \epsilon^2 \nabla_t^\top - c \nabla_x^\top) \nabla^{\top (m-1)} P^\top \partial u + \epsilon^3 O_{L^2} (\mathcal{E}_s).
$$
leads to
\begin{align*}
IIa & = - \frac{\lambda}{\epsilon^2} \int  ( \epsilon^2 \nabla_t^\top - c \nabla_x^\top) \nabla^{\top (m-1)} P^\top \partial u \cdot \nabla^{\top m}  \nabla_x^\top P^\top \partial_x u \,dx + O_{L^2} (\mathcal{E}_s^2) \\
& = \frac{\lambda}{2} \frac{d}{dt}  \int | \nabla^{\top m} P^\top \partial_x u|^2\,dx +  O_{L^2} (\mathcal{E}_s^2).
\end{align*}

\bigskip

\noindent
\underline{Step 5: estimating $IIb$.} Proceeding as in the previous steps,
\begin{equation*}
\begin{split}
\frac{i}{\epsilon^3} P^\top (\epsilon^2 \nabla_t - (c+iB) \nabla_x) \nabla^m V'(u) 
& = \frac{ 2 \lambda i}{\epsilon} P^\top (\epsilon^2 \nabla_t - (c+iB) \nabla_x) \nabla^m \Sigma \left( \begin{array}{l} 0 \\ n \end{array} \right) + O_{L^2} (\mathcal{E}_s) \\
& = \frac{ 2 \lambda i}{\epsilon}  \deux^\top \left( \nabla^{\top m} (\epsilon^2 \nabla_t^\top - c \nabla_x^\top) u , 
\Sigma \left( \begin{array}{l} 0 \\ n \end{array} \right) \right) + O_{L^2} (\mathcal{E}_s).
\end{split}
\end{equation*}
Pairing this identity with
$$
P^\perp (\epsilon^2 \nabla_t - c \nabla_x) \nabla^{m-1} \partial u = \frac{i \epsilon}{2} \nabla^{\top m} \nabla_x^\top P^\top \partial_x u + \epsilon O_{L^2}(\mathcal{E}_s)
$$
(where we used the equation satisfied by $u$) yields the desired estimate for $IIb$:
\begin{equation*}
\begin{split}
IIb & = \lambda \int \deux^\top \left(  \nabla^{\top m} (\epsilon^2 \nabla_t^\top - c \nabla_x^\top) u ,  \Sigma \left( \begin{array}{l} 0 \\ n \end{array} \right) \right) \cdot \nabla^{\top m} \nabla_x^\top P^\top \partial_x u\,dx + O(\mathcal{E}_s^2) \\
& = - \frac{\lambda \epsilon^2}{2} \frac{d}{dt} \int \deux^\top \left(\nabla^{\top m} P^\top \partial_x u
, \Sigma \left( \begin{array}{l} 0 \\ n \end{array} \right) \right) \cdot \nabla^{\top m} P^\top \partial_x u\,dx
+ O(\mathcal{E}_s^2),
\end{split}
\end{equation*}
where we used in the last line the almost symmetry of $\deux^\top$~\eqref{almostsym}.

\bigskip

\noindent
\underline{Step 6: conclusion.} The above estimates lead to the differential inequality
$$
\frac{d}{dt} E_m =O(\mathcal{E}_s^2)
$$
where
\begin{equation*}
\begin{split}
E_m = & \int \left[ \frac{1}{2} |(\epsilon^2 \nabla_t - c \nabla_x) \nabla^m u|^2 
+ \frac{\epsilon^2}{8} |\nabla^m \nabla_x \partial_x u|^2
- \frac{\mu}{2} |\nabla^m \partial_x u|^2 
+ \frac{\lambda}{2} |\nabla^{\top m} P^\top \partial_x u|^2 \right.\\
& \qquad \qquad \qquad \qquad \qquad \left. + \frac{\epsilon^4 \lambda}{2} \left| \nabla^{\perp m} 
\nabla_x^\perp \Sigma \left( \begin{array}{l} 0 \\  n \end{array}\right) \right|^2
\right]\,dx.
\end{split}
\end{equation*}
This gives the desired result since
$$
\sum_{|m|\leq s} E_m \gtrsim \mathcal{E}_{s,1}.
$$
\end{proof}
   
\subsection{Differentiating the hydrodynamical system}
  We shall again use the notation
     $$ \nabla^{\top m} \Phi^\epsilon = {1 \over \epsilon} \nabla^{\top m} \Phi.$$
 \begin{prop}
 \label{Kpropdiff}
  For $1 \leq |m |\leq s$ and $s \geq 2, $ we obtain on $ [0, T^\epsilon]$ the following system for  $( \nabla^{\top m} \Phi^\epsilon, \nabla^{\perp m} n)$.
 \beq
 \label{hydrodiffkahl}
 \left\{
 \begin{array}{llll}
 \displaystyle{  S_{\top \top}\big( \nabla_{t}^\top - {c \over \epsilon^2} \nabla_{x}^\top \big) \nabla^{\top m} \Phi^\epsilon - {1 \over \epsilon^2} (iB)_{\top\top} \nabla_{x}^\top \nabla^{\top m} \Phi^\epsilon + \epsilon^5S_{\top \perp} \big( \nabla_{t}^\perp - {c \over \epsilon^2} \nabla_{x}^\perp
  \big) \nabla^{\perp m} n } 
     \\
   \hspace{0.5cm}  \displaystyle{ =   i \Big[ {1 \over 2} \nabla_{x}^\perp \big( S_{\perp \perp} \nabla_{x}^\perp \nabla^{\perp m}n 
    + \epsilon^3 S_{\perp \top} \nabla_{x}^\top \nabla^{\top m} \Phi^\epsilon\big)
   - { 2 \lambda  \over \epsilon^2} S_{\perp \perp} \nabla^{\perp m}n}  - 2\lambda i \deux^\top \big( i S_{\perp \perp} n, S_{\perp \perp} \nabla^{\perp m} n \big) \\
   \hspace{1.5cm} \displaystyle{+ {1 \over 2 } \deux^\perp\big(S_{\top \top} \nabla_{x}^\top \nabla^{\top m} \Phi^\epsilon, S_{\top \top} D\Phi \partial_{x} \phi)
    + {1 \over 2 }   \deux^\perp\big(S_{\top \top} D\Phi \partial_{x}  \phi, S_{\top \top} \nabla_{x}^\top  \nabla^{\top m} \Phi^\epsilon) \Big] + O_{H^1}(\mathcal{E}_{s})  } \\
  \displaystyle{ S_{\perp \perp} \big( \nabla_t^\perp -{c \over \epsilon^2} \nabla_{x}^\perp \big) \nabla^{\perp m} n - {1 \over \epsilon^2}(iB)_{\perp \perp} \nabla_{x}^\perp \nabla^{\perp m} n
   +  \epsilon^3 S_{\perp \top} (\nabla_{t}^\top - {c \over \epsilon^2}\nabla_{x}^\top) \nabla^{\top m}
    \Phi^\epsilon } \\  
  \hspace{0.5cm} \displaystyle{   = i \Big[ {1 \over 2 \epsilon^2}  \nabla_{x}^\top \big( S_{\top \top} \nabla_{x}^\top \nabla^{\top m} \Phi^\epsilon + \epsilon^5 S_{\top \perp} \nabla_{x}^\perp
      \nabla^{\perp m } n \big)   }  \\
    \hspace{1cm}  \displaystyle{ + \deux^\top (S_{\top \top}D\Phi \partial_{x} \phi, S_{\perp \perp} \nabla_{x}^\perp \nabla^{\perp m}n \big)
       + {1 \over 2 } \deux^\top \big( S_{\top \top } \nabla_{x}^\top \nabla^{\top m} \Phi^\epsilon, S_{\perp \perp} \nabla_{x}^\perp n\big) \Big]
        + O_{H^1_{\epsilon}}(\mathcal{E}_{s}) }
        \end{array} \right.
       \eeq
        \end{prop}
   
   \begin{proof}
    We  shall apply   $(\nabla^\top)^m$ to  the first equation of~\eqref{eqhydroK}, and $(\nabla^\perp)^m$ to the second one.
    
    We first note that \eqref{nablaPhim1}, \eqref{nablaPhim2}, \eqref{nablaPhim3} still hold. We shall use these estimates intensively again.
    We already noticed note that the equivalents of \eqref{dtphi}, \eqref{dxxxphi}, namely~\eqref{dtphikahl} and~\eqref{dxxxphikahl}, are still true. 
 
\bigskip

\noindent     
\underline{Step 1: Left hand side of  $\eqref{eqhydroK}_{1}$.}
Let us apply $\nabla^{\top m}$ to the left hand side of  $\eqref{eqhydroK}$ and let us expand the resulting line $L^{\phi}$ as 
\begin{align*}
 & L_{1}^\phi =   S_{\top \top}  \nabla^{\top m} \big(\partial_{t}  - {c \over \epsilon^2} \partial_{x} \big){ \Phi \over \epsilon}  - {1 \over \epsilon^2} (iB)_{\top\top} \nabla^{\top m} {\partial_{x} \Phi \over \epsilon},\\
& L_{2}^\phi=   \Big[ \nabla^{\top m}, S_{\top\top} \Big] (\partial_{t}  - {c \over \epsilon^2} \partial_{x} \big){ \Phi \over \epsilon }   -  \Big[ \nabla^{\top m}, (iB)_{\top\top} \Big] \partial_{x}{ \Phi \over \epsilon }
 = \mathcal{C}_{1} - \mathcal{C}_{2} \\ 
& L_{3}^\phi =  \nabla^{\top m}\Big( \epsilon^5 S_{\top \perp}  \big(\nabla_{t}^\perp n - {c \over \epsilon^2} \nabla_{x}^\perp n \big)  \Big) -  \nabla^{\top m} \big( \epsilon^3 (iB)_{\top \perp} \nabla_{x}^
\perp n \big).
\end{align*}
Note that for $L_{2}^\phi$,  we still  use commutator notations, but that in the definition of 
$$ \Big[\nabla^{\top m}, S_{\top\top} \Big] X: = \nabla^{\top m}  \big( S_{\top \top} X  \big) - S_{\top \top } \nabla^{\top m} X,$$
the operator  $\nabla^\top$ is not the same in the first and in the second term.  Indeed, in the first term $\nabla^\top = P^\top \nabla$
 with $P^\top$ the orthogonal projector on $\tilde T_{u}\mathcal{L}$ while in the second term $\nabla^\top = P^\top \nabla$ with
  $P^\top$ the orthogonal projection on $T_{\Phi} \mathcal{L}.$ We shall keep on using this notation in the following.
To estimate $L_{1}^\phi$, we can  proceed as in \eqref{L1phi11} because of the identity \eqref{comT}, we thus get as in \eqref{L1phi}
\begin{equation}
\label{L1phiK}
 L_{1}^\phi= S_{\top\top} \big( \nabla_{t}^\top - {c \over \epsilon^2} \nabla_{x}^\top \big) \nabla^{\top m } \Phi^\epsilon - {1 \over \epsilon^2}  (iB)_{\top\top}  \nabla_{x}^\top \nabla^{\top m} \Phi^\epsilon + O_{H^1}(\mathcal{E}_{s}).
 \end{equation}
To estimate $L_2^\phi$,  let us start with the commutator $\mathcal{C}_1$.  We can expand it as a sum of terms under the form
 $$ \nabla^{\top \alpha} [ \nabla^\top, S_{\top \top}] \nabla^{\top \beta} \big( \partial_{t}  - {c \over \epsilon^2} \partial_{x} \big){ \Phi \over \epsilon }  , \quad | \alpha| + |\beta| = |m|-1.$$
Note that the $\nabla^\top$ in the middle commutator can be either $\nabla_{x}^\top$ or $\epsilon^2 \nabla_{t}^\top$. We thus obtain that
the commutator can be expanded into a sum of terms of the type
$$
\nabla^{ \top \alpha}  \Big((\nabla^\top S_{\top \top})   \big( \big( \nabla^{\top} \Phi, \epsilon^2 \nabla^{\perp} n),  \nabla^{\top \beta} \big(\partial_{t}  - {c \over \epsilon^2 } \partial_{x} \big){\Phi \over \epsilon} \big) \Big) 
$$  
with $| \alpha | + | \beta | = |m| -1$. 
 
 Thanks to \eqref{devDPsi}, we first observe that for every $X \in T_{\Phi} \mathcal{L}$, the tensor
  $ (\nabla_{X}^\top S_{\top\top})= \nabla^\top S_{\top \top}(X, \cdot)$ vanishes on $\mathcal{L}$
  thus we can write that
  \begin{equation}
  \label{macareuxK3} (\nabla^\top S_{\top \top})(X,  Y)  = \epsilon^2 \mathcal{T}_{(\Phi, n)}(X, Y)
  \end{equation}
  for some smooth tensor $ \mathcal{T}$.  This yields
 \begin{multline*} \nabla^{ \top \alpha} \Big(  (\nabla^\top S_{\top\top})_{ (\Phi , \epsilon^2n)}   \big( ( \nabla^{\top } \Phi, 0), \nabla^{\top \beta} \big( \partial_{t}  - {c \over \epsilon^2} \partial_{x} \big){\Phi \over \epsilon} \big) \big) \Big)   \\= \nabla^{ \top \alpha}  \Big(\mathcal{T}_{( \Phi , n)}    \big( \nabla^{\top } \Phi,  \nabla^{\top \beta} \big(\epsilon^2\partial_{t}  - c\partial_{x} \big){\Phi \over \epsilon} \big) \big) \Big) = O_{H^1}(\mathcal{E}_{s}).
\end{multline*}
 Next, we also observe  thanks to \eqref{devDPsi} that  for every  $Z\in N_{\Phi} \mathcal{L}$, $Y \in T_{\Phi} \mathcal{L}$, the tensor
 $ (\nabla^\top S_{\top \top}) ( (0, Z),  Y)
  - \deux^\top_{u}\big( S_{\top \top} Y,  S_{\perp \perp} Z \big)$ vanishes on $\mathcal{L}$. Thus, we can write
 \begin{equation}
 \label{macareuxK4}
  ( \nabla^\top S_{\top \top} )( (0, Z),  Y) =  \deux^\top_{u}\big( S_{\top \top}  Y,  S_{\perp \perp} Z \big) + \epsilon^2 \mathcal{T}_{(\Phi, n)}\left( Z, Y\right)
   \end{equation}
  for some smooth tensor $\mathcal{T}$  (in this proof, we shall always use the notation $\mathcal{T}(\Phi,n)$ for a harmless smooth tensor).
   We thus easily obtain that 
   \begin{multline}
 \nabla^{ \top \alpha}  \Big( (\nabla^\top S_{\top \top})_{( \Phi , \epsilon^2n)}   \big( \big( 0, \epsilon^2 \nabla^{\perp} n),  \nabla^{\top \beta} \big(\partial_{t}  - {c \over \epsilon^2 } \partial_{x} \big){\Phi \over \epsilon} \big) \Big) \\
  =     \nabla^{ \top \alpha}  \left[ \deux^\top_{u} \left( S_{\top\top}  \nabla^{\top \beta}(\epsilon^2 \partial_{t} - c \partial_{x}) {\Phi \over \epsilon}, S_{\perp \perp} \nabla^\perp n\right) \right]
   + O_{H^1}(\mathcal{E}_{s})
   \end{multline}
Observe then that all the above  terms are $O_{H^1}(\mathcal{E}_s)$ as long as $\beta \neq 0$. We now focus on the term for which $\beta = 0$. It can be written as
\begin{multline*}
 \nabla^{ \top m-1} \Big( \deux^\top_u \left( S_{\top \top} \big(\epsilon^2 \partial_{t}  - c \partial_{x} \big){\Phi \over \epsilon},  S_{\perp \perp} \nabla^\perp n \right)\Big) \\
 =  \deux^\top_u \left( S_{\top \top} \big(\epsilon^2 \partial_{t}  - c \partial_{x} \big){\Phi \over \epsilon},  S_{\perp \perp} \nabla^{\perp m} n \right)
 + O_{H^1}(\mathcal{E}_s) .
\end{multline*}
We have thus proven that
\begin{equation}
\label{C1Khydro}
\mathcal{C}_{1} =  \deux^\top_u \left( S_{\top \top}\big( \epsilon^2 \partial_{t}  - c \partial_{x} \big){\Phi \over \epsilon},  S_{\perp \perp} \nabla^{\perp m} n \right)
 + O_{H^1}(\mathcal{E}_s).
 \end{equation}
  We can  proceed in a similar way for the commutator $\mathcal{C}_2$. We thus have to estimate
  \begin{equation}
  \label{C2Khydro1} {1 \over \epsilon^2} \nabla^{\top \alpha} \left(  (\nabla^\top (iB)_{\top\top}) \left(( \nabla^{\top} \Phi, \epsilon^2 \nabla^{\perp}n), \nabla^{\top \beta} \partial_{x} {\Phi \over \epsilon} \right)\right)
  \end{equation}
  with $(\alpha, \beta)$ as above.
   Thanks to \eqref{iBTTL}, we first observe that
  $$ \nabla^{\top \alpha}  \left(  {1 \over \epsilon^2}
     (\nabla^\top (iB)_{\top\top}) \left(( \nabla^{\top } \Phi, 0), \nabla^{\top \beta} \partial_{x} {\Phi \over \epsilon} \right)\right)
  =    \left(  \nabla^{\top \alpha}  \mathcal{T}_{(\Phi, n)}   \left( \nabla^{\top } \Phi, \nabla^{\top \beta} \partial_{x} {\Phi \over \epsilon} \right)\right) =  O_{H^1}(\mathcal{E}_s).$$
   Thus, it remains to estimate
  $$ \nabla^{\top \alpha} \left(  (\nabla^\top (iB)_{\top\top}) \left((0, \nabla^{\perp }n), \nabla^{\top \beta} \partial_{x} {\Phi \over \epsilon} \right)\right).$$
    We note that for every $X\in T_{\Phi} \mathcal{L}$, $Z \in N_{\Phi} \mathcal{L}$, we have  that
      $$   \left(\nabla_{Z}^\top(iB)_{\top\top} \right)_{(\Phi, \epsilon^2 n)} X =  \deux^\top_{u}\left( (iB)_{\top\top}X,  S_{\perp \perp}Z \right) + \epsilon^2  \mathcal{T}(\Phi, n) (Z,X).$$
       Indeed, this follows from the fact that the tensor
     $   \left(\nabla_{Z}^\top(iB)_{\top\top} \right) - \deux^\top_{u}\left( (iB)_{\top\top}\cdot,  S_{\perp \perp}Z \right)$
     vanishes on $\mathcal{L}$ thanks to \eqref{iBTTL}.
  By plugging this identity in the terms \eqref{C2Khydro1}, we finally  obtain that
 \begin{equation}
 \label{C2Khydro}
\mathcal{C}_{2} =  \deux^\top_{u}\left( (iB)_{\top\top}\partial_{x}{\Phi \over \epsilon},  S_{\perp \perp} \nabla^{\perp m}n  \right)  + O_{H^1}(\mathcal{E}_{s}).
 \end{equation}
 By collecting \eqref{C1Khydro} and \eqref{C2Khydro}, we get
 $$ L_{2}^\phi= \deux^\top_u \left(  \Bigl(S_{\top \top} (\epsilon^2\partial_{t}  - c \partial_{x}) - (iB)_{\top \top } \partial_{x} \Bigr){\Phi \over \epsilon},  S_{\perp \perp} \nabla^{\perp m} n \right)
 + O_{H^1}(\mathcal{E}_s)$$
 and thanks to the first line of  \eqref{eqhydroK}, we can simplify its expression as
 \begin{equation}
 \label{L2phiK}
  L_{2}^\phi=  - 2 \lambda \deux^\top_{u} \left(   i S_{\perp \perp} n, S_{\perp \perp} \nabla^{\perp m}n  \right) +  O_{H^1}(\mathcal{E}_{s}).
  \end{equation}
It remains to deal with $L_{3}^\phi$ which is  easy to handle. Just by counting derivatives,  we get
\begin{equation}
\label{L3phiK}L_{3}^\phi =   \epsilon^5 S_{\top \perp} (\Phi,  \epsilon^2 n)  \big(\nabla_{t}^\perp 
- {c \over \epsilon^2} \nabla_{x}^\perp  \big)  \nabla^{\perp m} n  + O_{H^1}(\mathcal{E}_{s}).
\end{equation}
  Note that we cannot consider the term involving $\nabla_{t}^\perp \nabla^{\perp m}n$ as a remainder term since it  
    is not controlled in $H^1$   by $\mathcal{E}_{s,1}$
    (at least when $ \nabla^m$ involves only time derivatives). We also keep the term involving $  \nabla_{x}^\perp \nabla^{\perp m}n$ for symmetry reasons.
    
Looking at \eqref{L3phiK}, \eqref{L2phiK}, \eqref{L1phiK}, we have thus proven that
\begin{multline}
\label{LHS1kahl}
\nabla^{\top m } LHS\eqref{eqhydroK}_{1}=  S_{\top \top} \big( \nabla_{t}^\top - {c \over \epsilon^2} \nabla_{x}^\top \big) \nabla^{\top m} \Phi^\epsilon
 - {1 \over \epsilon^2} (iB)_{\top\top} \nabla_{x}^\top \nabla^{\top m } \Phi^\epsilon \\
+ \epsilon^5 S_{\top \perp} \big(\nabla_{t}^\perp - {1 \over \epsilon^2} \nabla_{x}^\perp  \big)  \nabla^{\perp m} n  - 2\lambda \deux^\top\Big( i S_{\perp \perp} n, S_{\perp \perp} \nabla^{\perp m} n \Big) + O_{H^1}(\mathcal{E}_{s}).
\end{multline}
 We shall again omit to specify that $\deux^\top$ is evaluated at $u$.
 
\bigskip
\noindent
\underline{Step 2: Right-hand side of  $\eqref{eqhydroK}_{1}$.}
Simply counting derivatives,
\begin{align*} 
&  \nabla^{\perp m } \Big( \deux^\perp\Big(  \Sigma \left(\begin{array}{c} {\partial_{x} \Phi \over \epsilon} \\ \epsilon \nabla_{x}^\perp n 
\end{array} \right), P^\top   \Sigma \left(\begin{array}{c} {\partial_{x} \Phi \over \epsilon} \\ \epsilon \nabla_{x}^\perp n 
\end{array} \right) \Big)  \Big) \\ 
& =   \deux^\perp\Big(  \Sigma \left(\begin{array}{c}  \nabla_{x}^\top \nabla^{\top m} \Phi^\epsilon \\ \epsilon \nabla_{x}^\perp  \nabla^{\perp m}n 
      \end{array} \right), P^\top   \Sigma \left(\begin{array}{c} {D\Phi \partial_{x} \phi } \\ \epsilon \nabla_{x}^\perp n 
      \end{array} \right) \Big) 
         +  \deux^\perp\Big(  \Sigma \left(\begin{array}{c}  D\Phi \partial_{x}  \phi\\ \epsilon \nabla_{x}^\perp n 
      \end{array} \right), P^\top   \Sigma \left(\begin{array}{c} \nabla_{x}^\top \nabla^{\top m} \Phi^\epsilon\\ \epsilon \nabla_{x}^\perp \nabla^{\perp m } n 
      \end{array} \right) \Big)  + O_{H^1}(\mathcal{E}_{s}) \\
      & =  T_{1}+ T_{2}  + O_{H^1}(\mathcal{E}_{s})
       \end{align*}
  as in the proof of \eqref{rightphi1}. Using the properties of $\Sigma$ and $\deux^\perp$, we can further simplify the expression of these terms: by~(\ref{rossignol}) and~\eqref{rougegorge},
\begin{equation*}
\begin{split}
& T_{1}= \deux^\perp\big( S_{\top\top}\nabla_{x}^\top \nabla^{\top m} \Phi^\epsilon, S_{\top \top} D\Phi \partial_{x} \phi \big)
       +  O_{H^1}(\mathcal{E}_{s}) \\
&  T_{2}=  \deux^\perp\big( S_{\top \top} D\Phi\partial_{x} \phi,  S_{\top \top} \nabla_{x}^\top \nabla^{\top m} \Phi^\epsilon) +   O_{H^1}(\mathcal{E}_{s})
\end{split}
\end{equation*}
Therefore,
\begin{multline}
\label{transport1kahl}
\nabla^{\perp m } \Big( \deux^\perp\Big(  \Sigma \left(\begin{array}{c} {\partial_{x} \Phi \over \epsilon} \\ \epsilon \nabla_{x}^\perp n 
\end{array} \right), P^\top   \Sigma \left(\begin{array}{c} {\partial_{x} \Phi \over \epsilon} \\ \epsilon \nabla_{x}^\perp n 
\end{array} \right) \Big)  \Big)   \\
=  \deux^\perp\big( S_{\top\top} \nabla_{x}^\top \nabla^{\top m} \Phi^\epsilon, S_{\top \top} D\Phi \partial_{x} \phi \big)
+  \deux^\perp\big( S_{\top \top} D\Phi\partial_{x} \phi,  S_{\top \top}  \nabla_{x}^\top \nabla^{\top m} \Phi^\epsilon) +   O_{H^1}(\mathcal{E}_{s}).
 \end{multline}
To treat
$$ {1 \over \epsilon^2}  \nabla^{\perp m } \big( S_{\perp \perp} n \big),$$
we use \eqref{devDPsi}  which gives that
\begin{equation}
\label{avocette}
\left(\nabla^\perp  S_{\perp \perp}  \right)_{(\Phi, \epsilon^2n)}= \epsilon^2 \mathcal{T}(\Phi, n).
\end{equation}
This equality enables to write the commutation of one derivative with $S_{\perp \perp}$ as follows
\begin{align*}
\frac{1}{\epsilon^2} \nabla^{\perp m} ( S_{\perp \perp} n )& = \frac{1}{\epsilon^2} \nabla^{\perp m-1} ( S_{\perp \perp}  \nabla^\perp n )
+ \frac{1}{\epsilon^2}  \nabla^{\perp m-1} ( \nabla^\perp S_{\perp \perp} ((\nabla^\perp \Phi,\epsilon^2 \nabla^\perp n),n) ) \\
& =  \frac{1}{\epsilon^2} \nabla^{\perp m-1}  (S_{\perp \perp} \nabla^\perp n) +O_{H^1(\mathcal{E}_s)}
\end{align*}
Repeating this operation gives the desired identity:
\begin{equation} \label{transport2kahl}
\frac{1}{\epsilon^2} \nabla^{\perp m} ( S_{\perp \perp} n) = \frac{1}{\epsilon^2} S_{\perp \perp}  \nabla^{\perp m}n  +O_{H^1}(\mathcal{E}_s).
\end{equation}
It remains to consider
$$  \nabla^{\perp m } \big( \nabla_{x}^\perp (S_{\perp \perp} \nabla_{x}^\perp n \big)\big).$$
The first step is to commute $ \nabla^{\perp m}$ with the first $\nabla_{x}^\perp$ using \eqref{riemperpkahl}. 
Observe that if $\nabla^{\perp m}$ does not contain time derivatives, there is no commutator, we can thus
write $ \nabla^{\perp m}= \nabla^{\perp \tilde m} \epsilon^2 \nabla_{t}^\perp$ and get
\begin{align*}   \nabla^{\perp m } \big( \nabla_{x}^\perp (S_{\perp \perp} \nabla_{x}^\perp n \big)\big)
& =    \nabla^{\perp \tilde m } \big( \nabla_{x}^\perp  \epsilon^2 \nabla_{t}^\perp  (S_{\perp \perp} \nabla_{x}^\perp n \big)\big) +
\nabla^{\perp \tilde  m } \big(  \epsilon^2 R^\perp( {\epsilon^2 \partial_{t}u \over \epsilon },  {\partial_{x} u \over \epsilon} )  S_{\perp \perp} \nabla_{x}^\perp n \big) \\
&  =    \nabla^{\perp \tilde  m } \big( \nabla_{x}^\perp \epsilon^2 \nabla_{t}^\perp  (S_{\perp \perp} \nabla_{x}^\perp n \big)\big)  + O_{H^1}(\mathcal{E}_{s}).
\end{align*}
Iterating this manipulation yields
$$
\nabla^{\perp m } \big( \nabla_{x}^\perp (S_{\perp \perp} \nabla_{x}^\perp n \big)\big)=
\nabla_{x}^\perp \Big( \nabla^{\perp m}   \big(  S_{\perp \perp} \nabla_{x}^\perp n \big)\Big) + O_{H^1}(\mathcal{E}_{s}).
$$
Commuting one derivative with $S_{\perp \perp}$ gives
\begin{align*}
\nabla_{x}^\perp \Big( \nabla^{\perp m}   \big(  S_{\perp \perp} \nabla_{x}^\perp n \big)\Big) 
& = \nabla_{x}^\perp \Big( \nabla^{\perp m-1}   \big(  S_{\perp \perp} \nabla^\perp \nabla_{x}^\perp n \big)\Big) + \nabla_x^\perp \nabla^{\perp m-1} \left[ \nabla^\perp S_{\perp \perp} \right] \left( (\nabla^\top \Phi , \epsilon^2 \nabla^\perp n),\nabla_x^\perp n \right) \\
& = \nabla_{x}^\perp \Big( \nabla^{\perp m-1}   \big(  S_{\perp \perp} \nabla^\perp \nabla_{x}^\perp n \big)\Big)  + O_{H^1}(\mathcal{E}_{s});
\end{align*}
here, the last equality follows from~(\ref{avocette}). Repeating this argument leads to
$$
 \nabla^{\perp m} \Big(  \nabla_{x}^\perp \big(  S_{\perp \perp} \nabla_{x}^\perp n \big)\Big) = \nabla_{x}^\perp \Big(  S_{\perp \perp} \nabla^{\perp m} \nabla_{x}^\perp n \Big) + O_{H^1}(\mathcal{E}_{s}).
$$
Finally, using once again the commutation relation \eqref{riemperpkahl} enables us to write
\begin{equation}
\label{transport3kahl}
 \nabla^{\perp m} \Big(  \nabla_{x}^\perp \big(  S_{\perp \perp} \nabla_{x}^\perp n \big)\Big) = \nabla_{x}^\perp \Big(  S_{\perp \perp} \nabla_{x}^\perp \nabla^{\perp m} n \Big) + O_{H^1}(\mathcal{E}_{s}).
\end{equation}
 The term  $\epsilon^3  \nabla^{\perp m}  \nabla_{x}^\perp ( S_{\perp \top}  {\partial_{x} \Phi \over \epsilon} )$ can be easily expanded as
 $$  \epsilon^3  \nabla^{\perp m}  \nabla_{x}^\perp ( S_{\perp \top}  {\partial_{x} \Phi \over \epsilon} )
  = \epsilon^3 \nabla_{x}^\perp ( S_{\perp \top} \nabla_{x}^\top  \nabla^{\top m} \Phi^\epsilon)  + O_{H^1}(\mathcal{E}_{s})$$
   just by counting the derivatives of the commutator. 
Gathering~(\ref{transport1kahl}), (\ref{transport2kahl}),  (\ref{transport3kahl}) and the last identity, we obtain
\begin{align}
\label{RHS1kahl}
\nabla^{\top m} RHS&\eqref{eqhydroK}_1  = 
i \left[ \frac{1}{2} \nabla_{x}^\perp \Big(  S_{\perp \perp} \nabla^{\perp m} \nabla_{x}^\perp n \Big) + { 1 \over 2} \epsilon^3 \nabla_{x}^\perp ( S_{\perp \top} \nabla_{x}^\top  \nabla^{\top m} \Phi^\epsilon) - \frac{2\lambda}{\epsilon^2}  S_{\perp \perp}   \nabla^{\perp m}n \right. \\
& \left. + \frac{1}{2} \deux^\perp\big( S_{\top\top}  \nabla_{x}^\top \nabla^{\top m} \Phi^\epsilon, S_{\top \top} D\Phi \partial_{x} \phi \big)
+ \frac{1}{2} \deux^\perp\big( S_{\top \top} D\Phi\partial_{x} \phi,  S_{\top \top}  \nabla_{x}^\top \nabla^{\top m} \Phi^\epsilon) \right] +   O_{H^1}(\mathcal{E}_{s}).
\end{align}

\bigskip

\noindent
\underline{Step 3: Left-hand side of  $\eqref{eqhydroK}_{2}$.}
 By using similar arguments as above, we can first write
\begin{multline*}
  \nabla^{\perp m}  LHS\eqref{eqhydroK}_{2}=  S_{\perp \perp} \big( \nabla_{t}^\perp - {c \over \epsilon^2} \nabla_{x}^\perp \big) \nabla^{\perp m} n 
  - {1 \over \epsilon^2}\nabla^{\perp m}\left(  (iB)_{\perp\perp} \nabla_{x}^\perp n\right)  \\
   + \epsilon^3 S_{\perp \top}( \nabla_{t}^\top - { c \over \epsilon^2} \nabla_{x}^\top) \nabla^{\top m} \Phi^\epsilon + O_{H^1_{\epsilon}}(\mathcal{E}_{s}).
  \end{multline*}          
To handle  ${1 \over \epsilon^2} \nabla^{\perp m}\left(  (iB)_{\perp\perp}  \nabla_{x}^\perp n \right)$, we can expand
 this term as previously and use \eqref{iBPPL} to handle the commutator.
This yields
  $$  {1 \over \epsilon^2} \nabla^{\perp m}\left(  (iB)_{\perp\perp}  \nabla_{x}^\perp n \right)= {1 \over \epsilon^2}  (iB)_{\perp \perp} \nabla_{x}^\perp \nabla^{\perp m} n + O_{H^1_{\epsilon}}
  (\mathcal{E}_{s}).$$ 
  We have thus proven that
\begin{multline}
\label{LHS2kahl}  \nabla^{\perp m}  LHS\eqref{eqhydroK}_{2}=  S_{\perp \perp} \big( \nabla_{t}^\perp - {c \over \epsilon^2} \nabla_{x}^\perp \big) \nabla^{\perp m} n 
  - {1 \over \epsilon^2}  (iB)_{\perp\perp} \nabla_{x}^\perp  \nabla^{\perp m} n   \\+ \epsilon^3 S_{\perp \top}( \nabla_{t}^\top - { c \over \epsilon^2} \nabla_{x}^\top) \nabla^{\top m} \Phi^\epsilon + O_{H^1_{\epsilon}}(\mathcal{E}_{s}).\end{multline} 
\bigskip

\noindent
\underline{Step 4: Right -hand side of  $\eqref{eqhydroK}_{2}$.}  
Again simply using commutator estimates, we obtain
\beq
\label{RHS2kahl1} \nabla^{\perp m } \nabla_{x}^\perp \Big(  i \epsilon^3  S_{\top \perp}\nabla_{x}^\perp n \Big) = i  \nabla_{x}^\top
\big(  \epsilon^3 S_{\top \perp} \nabla_{x}^\perp  \nabla^{\perp m} n \big) + O_{H^1_\epsilon}(\mathcal{E}_s).
\eeq
In a similar way, there is no new difficulty in commuting $\nabla^{\top m}$ and the term involving $\deux^\top$ in order to get
\begin{multline*} \nabla^{\top m} \Big( 
\deux^\top \Big( \Sigma \left( \begin{array}{ll}  D\Phi \partial_{x}  \phi \\ \epsilon \nabla_{x}^\perp n 
\end{array} \right), P^\perp \Sigma \left( \begin{array}{ll}  D\Phi \partial_{x}  \phi \\ \epsilon \nabla_{x}^\perp n 
\end{array} \right) \Big) \Big)
\\
= \deux^\top \Big( \Sigma \left( \begin{array}{ll} \nabla_{x}^\top  \nabla^{\top m } \Phi^\epsilon \\ \epsilon \nabla_{x}^\perp \nabla^{\perp m } n 
\end{array} \right),  P^\perp \Sigma \left( \begin{array}{ll}  D\Phi \partial_{x}  \phi \\ \epsilon \nabla_{x}^\perp n 
\end{array} \right)\Big)
+ \deux^\top \Big( \Sigma \left( \begin{array}{ll}  D\Phi \partial_{x}  \phi \\ \epsilon \nabla_{x}^\perp n 
\end{array} \right),  P^\perp \Sigma \left( \begin{array}{ll} \nabla_x^\top \nabla^{\top m} \Phi^\epsilon \\ \epsilon \nabla_{x}^\perp \nabla^{\perp m}  n 
\end{array} \right) \Big) \Big) + O_{H^1_\epsilon}(\mathcal{E}_s) .
\end{multline*}
Next, proceeding as in the proof of \eqref{transport1kahl}, we obtain 
\begin{multline}
\label{RHS2kahl2}
\nabla^{\top m} \Big(  \deux^\top \Big( \Sigma \left( \begin{array}{ll}  D\Phi \partial_{x}  \phi \\ \epsilon \nabla_{x}^\perp n 
\end{array} \right), S_{\perp \perp} \nabla_{x}^\perp n \Big) \Big) \\
=  \deux^\top \big( S_{\top\top}\nabla_{x}^\top  \nabla^{\top m }\Phi^\epsilon, S_{\perp \perp}\nabla_{x}^\perp n \big)
+ \deux^\top \big( S_{\top\top} D\Phi \partial_{x } \phi , S_{\perp \perp} \nabla_{x}^\perp \nabla^{\perp m} n \big).
\end{multline}

We shall now  study more precisely the term ${ 1 \over \epsilon^2} \nabla^{\top m } \nabla_{x}^\top\big(S_{\top \top} D\Phi \partial_{x} \phi \big)$. By using  \eqref{riemtopkahl}, we first get
$$
{ 1 \over \epsilon^2} \nabla^{\top m } \nabla_{x}^\top\big(S_{\top \top} D\Phi \partial_{x} \phi \big)=  { 1 \over \epsilon^2} \nabla_{x}^\top \big( \nabla^{\top m} \big( S_{\top \top} D\Phi \partial_{x} \phi \big)\big).
$$
We now commute $ \nabla^{\top m}$ with $S_{\top \top}$. It yields a sum of terms of the type
\begin{align*}
& \frac{1}{\epsilon^2} \nabla^\top_x \nabla^{\top \alpha} \left[ \nabla^\top S_{\top \top} \right] \left( (\nabla^{\top } \Phi,\epsilon^2 \nabla^{\perp } n), \nabla^{\top \gamma} D\Phi \partial_x \phi \right)\\
& \quad =  \frac{1}{\epsilon^2} \nabla^\top_x \nabla^{\top \alpha} \left[ \nabla^\top S_{\top \top} \right] \left( (0,\epsilon^2 \nabla^{\perp } n), \nabla^{\top \gamma} D\Phi \partial_x \phi \right) + \frac{1}{\epsilon^2} \nabla^\top_x \nabla^{\top \alpha} \left[ \nabla^\top S_{\top \top} \right] \left( (\nabla^{\top } \Phi,0), \nabla^{\top \gamma} D\Phi \partial_x \phi \right).
\end{align*}
where $|\alpha| + |\gamma| = m-1$. The second term on the above right-hand side is always $O_{H^1_\epsilon}(\mathcal{E}_s)$ by~(\ref{devDPsi}); as for the first term on the above-right hand side, due in particular to~\eqref{dxxxphikahl}, it is also $O_{H^1_\epsilon}(\mathcal{E}_s)$ unless $\gamma=0$. In other words,
\begin{align*}
& { 1 \over \epsilon^2} \nabla^{\top m } \nabla_{x}^\top\big(S_{\top \top} D\Phi \partial_{x} \phi \big) \\
& \qquad = { 1 \over \epsilon^2} \nabla_{x}^\top\big(S_{\top \top} \nabla^{\top m } D\Phi \partial_{x} \phi \big)
+  \frac{1}{\epsilon^2}  \left[ \nabla^\top S_{\top \top} \right] \left( (0,\epsilon^2 \nabla^{\perp}_x \nabla^{\perp m} n), D\Phi \partial_x \phi \right) + O_{H^1_\epsilon}(\mathcal{E}_s) .
\end{align*}

Using once again~(\ref{devDPsi}), the second term on the above right-hand side can be replaced as follows
\begin{align}
 \nonumber& { 1 \over \epsilon^2} \nabla^{\top m } \nabla_{x}^\top\big(S_{\top \top} D\Phi \partial_{x} \phi \big) \\
\nonumber& \qquad = { 1 \over \epsilon^2} \nabla_{x}^\top\big(S_{\top \top} \nabla^{\top m } D\Phi \partial_{x} \phi \big)
+  \deux^\top \left(S_{\top \top} D\Phi \partial_x \phi , S_{\perp \perp} \nabla_x^\perp \nabla^{\perp m} n \right) + O_{H^1_\epsilon}(\mathcal{E}_s)\\
\label{RHS2kahl3}
& \qquad = { 1 \over \epsilon^2} \nabla_{x}^\top\big(S_{\top \top} \nabla_x^{\top} \nabla^{\top m } \Phi^\epsilon \big)
+  \deux^\top \left(S_{\top \top} D\Phi \partial_x \phi , S_{\perp \perp} \nabla_x^\perp \nabla^{\perp m} n \right)
+ O_{H^1_\epsilon}(\mathcal{E}_s) .
\end{align}
Combining \eqref{RHS2kahl1}, \eqref{RHS2kahl2}, \eqref{RHS2kahl3}, we have thus proven that
\begin{multline}
\label{RHS2kahl}
\nabla^{\perp m} RHS~\eqref{eqhydroK}_{2}= i \Big[ {1 \over 2 \epsilon^2}  \nabla_{x}^\top \big( S_{\top \top} \nabla_{x}^\top \nabla^{\top m} \Phi^\epsilon  + \epsilon^5 S_{\top \perp} \nabla_{x}^\perp
\nabla^{\perp m } n \big)   +  \\ \deux^\top (S_{\top \top}  D\Phi \partial_{x}  \phi, S_{\perp \perp} \nabla_{x}^\perp \nabla^{\perp m}n \big)
+ {1 \over 2 } \deux^\top \big( S_{\top \top } \nabla_x^\top \nabla^{\top m} \Phi^\epsilon, S_{\perp \perp} \nabla_{x}^\perp n\big) \Big]
+ O_{H^1_{\epsilon}}(\mathcal{E}_{s}).
\end{multline}
The proof of Proposition \ref{Kpropdiff} follows by combining \eqref{LHS1kahl},  \eqref{RHS1kahl},  \eqref{LHS2kahl}, \eqref{RHS2kahl}.
\end{proof}

\subsection{Estimates on the hydrodynamical system}
The aim of this section is to use the hydrodynamical system derived in Proposition  \ref{Kpropdiff} to prove the following
\begin{prop}
\label{esthydroK}
The following a priori estimate holds
$$
\mathcal{E}_{s,2}^2(u,t) \lesssim \mathcal{E}_{s}^2(u,0) + \epsilon O( \mathcal{E}_{s} (u,t)) + \int_0^t O(\mathcal{E}_{s}(u,\tau))\,d\tau, \quad \forall t \in [0, T^\epsilon].
$$
\end{prop}

We shall use the same idea as in the proof of Proposition \ref{esthydroplat}. The additional difficulty is that the coupling
between the two equations of the system \eqref{hydrodiffkahl} is slightly stronger than before.
It is thus more convenient to deal directly with the whole system and use vectors and block matrix notations.

We introduce some notations before stating the proposition: let $ U^m \overset{def}{=} (\nabla^{\top m} \Phi^\epsilon, \nabla^{\perp m}n )^t= (U^m_{1}, U^m_{2})^t \in T_{\Phi}\mathcal{L} \times N_{\Phi}\mathcal{L}$.

For any vector $U \in \widetilde T_{u}\mathcal{L} \times  \widetilde N_{u}\mathcal{L}$ along $u \in \mathcal{M}$, it will be convenient to use the notation  
$$ \mathcal{D}_{i} U\overset{def}{=} (\nabla^\top_{i} U_{1}, \nabla^{\perp}_{i} U_{2}), \quad i=t, \, x.$$
Note that  in this notation in the definition of the connection, $P^\top$ and $P^\perp$ are to be taken at $u$.  
Define the  matrix $J_{u} \in \mathscr{L}(\widetilde T_{u} \mathcal{L} \times \widetilde N_{u} \mathcal{L}, \widetilde T_{u} \mathcal{L} \times \widetilde N_{u} \mathcal{L})$ by  
$$
J_{u} \overset{def}{=} \left( \begin{array}{ll} 0 & i_{u} \\ i_{u} & 0 \end{array} \right).
$$
Note that $J$ is skew symmetric for the scalar product  on $\widetilde T_{u} \mathcal{L} \times \widetilde N_{u} \mathcal{L}$ defined by 
\begin{equation}
\label{scalartilde1}
 \langle U, V \rangle_{u}=\langle U_{1}, V_{1} \rangle_{u} +  \langle U_{2}, V_{2} \rangle_{u}.
 \end{equation}
 In a similar way, we define $\mathcal{B}_{u} \in  \mathscr{L}(\widetilde T_{u} \mathcal{L} \times \widetilde N_{u} \mathcal{L}, \widetilde T_{u} \mathcal{L} \times \widetilde N_{u} \mathcal{L})$
  defined by 
 $$ \mathcal{B}_{u}=  \left( \begin{array}{cc}  0 & P^\top_{u} B(u) P^\perp_{u} \\ P^\perp_{u} B(u)  P^\top_{u}& 0 \end{array} \right).$$
  Again $\mathcal{B}$ is skew-symmetric for the above scalar product. 
We shall also use $ \mathcal{M}(\Phi, \epsilon^2n) \in \mathscr{L}(T_{\Phi}\mathcal{L} \times N_{\Phi} \mathcal{L}, \widetilde T_{u} \mathcal{L} \times \widetilde N_{u} \mathcal{L})$
defined as 
$$\mathcal{M}\overset{def}{=}  \left( \begin{array}{cc} S_{\top \top} & \epsilon^5 S_{\top \perp} \\  \epsilon^3 S_{\perp\top}  & S_{\perp \perp} \end{array} \right).$$
Note that the only difference between $\mathcal{M}$ and $\Sigma$ defined in \eqref{rossignol} is  in the  power of $\epsilon$
 in front of the off-diagonal terms  
  that come from the anisotropic scaling of the KdV limit.

We now rewrite in a more concise form the system~\eqref{hydrodiffkahl}: first the left hand-side of \eqref{hydrodiffkahl}, which, thanks to \eqref{Baction}, we can write
\begin{equation}
LHS~\eqref{hydrodiffkahl}= \mathcal{M} \left( \mathcal{D}_{t} U^m - {c \over \epsilon^2} \mathcal{D}_{x} \right) U^m - \frac{1}{\epsilon^2} J \mathcal{B} \mathcal{M} \mathcal{D}_{x} U^m+ \left( \begin{array}{cc} O_{H^1}(\mathcal{E}_{s}) \\
 O_{H^1_{\epsilon}}(\mathcal{E}_{s}) \end{array}\right).
\end{equation}
Next, the right hand-side of \eqref{hydrodiffkahl}, which can be symmetrized a little more by observing that 
 \begin{multline}
 \label{petitesym}
  { 1 \over 2} \deux^\perp_{u}\left( S_{\top \top}\nabla_{x}^\top \nabla^{\top m} \Phi^\epsilon, S_{\top \top } D\Phi \partial_{x} \phi\right)
  + {1 \over 2 } \deux^\perp_{u} \left(  S_{\top \top } D\Phi \partial_{x} \phi, S_{\top \top}\nabla_{x}^\top \nabla^{\top m} \Phi^\epsilon\right) \\
   = \deux^\perp_{u}  \left(  S_{\top \top } D\Phi \partial_{x} \phi, S_{\top \top}\nabla_{x}^\top \nabla^{\top m} \Phi^\epsilon\right) + O_{H^1}(\mathcal{E}_{s}).
   \end{multline}
   Indeed, our generalized second fundamental form $\deux^\perp$ is not symmetric in its arguments, but if we define the tensor
   \begin{multline*} T(u)( D\Phi \partial_{x} \phi,  X)=  { 1 \over 2} \deux^\perp_{u}\left( S_{\top \top} (u) X, S_{\top \top }(u) D\Phi \partial_{x} \phi\right)
  + {1 \over 2 } \deux^\perp_{u} \left(  S_{\top \top }(u) D\Phi \partial_{x} \phi, S_{\top \top}(u) X
 \right) \\
 - \deux^\perp_{u}  \left(  S_{\top \top }(u) D\Phi \partial_{x} \phi, S_{\top \top}(u)X\right)
 \end{multline*}
 acting on $T_{\Phi}\mathcal{L} \times \mathcal{T}_{\Phi}\mathcal{L}$, it vanishes if $u \in \mathcal{L}$ by the symmetry in its arguments of the second
  fundamental form of the tangent bundle of $\mathcal{L}$. We thus obtain that
  $$ T(u)(D\Phi \partial_{x} \phi, X)= \epsilon^2 \mathcal{T}(\Phi, n)(D\Phi \partial_{x} \phi, X)$$
  and hence we get \eqref{petitesym}.
  
Using this simplification, define the tensor $ \mathcal{B}_{\deux}(u, \partial_{x} \phi) \in \mathscr{L}(\widetilde T_{u} \mathcal{L} \times \widetilde N_{u} \mathcal{L}, \widetilde T_{u} \mathcal{L} \times \widetilde N_{u} \mathcal{L})$ as
 $$ \mathcal{B}_{\deux}(u, \partial_{x} \phi)=  \left( \begin{array}{cc} 0  & \deux^\top_{u}(S_{\top\top} D\Phi \partial_{x} \phi, \cdot) \\  \deux^\perp_{u} (S_{\top \top}D\Phi \partial_{x} \phi, \cdot) & 0 \end{array}\right)$$
  so that
 $$
  \left( \begin{array}{cc}
  i    \deux^\perp\big(S_{\top \top} D\Phi \partial_{x}  \phi, S_{\top \top} \nabla_{x}^\top  \nabla^{\top m} \Phi^\epsilon)   \\
   i  \deux^\top (S_{\top \top}D\Phi \partial_{x} \phi, S_{\perp \perp} \nabla_{x}^\perp \nabla^{\perp m}n  ) \end{array} \right)
   =  J \mathcal{B}_{\deux} \mathcal{M} \mathcal{D}_{x}U^m +  \left( \begin{array}{cc} O_{H^1}(\mathcal{E}_{s}) \\
 O_{H^1_{\epsilon}}(\mathcal{E}_{s}) \end{array}\right).
$$
By definition of $\deux^\top$ and $\deux^\perp$, we have that
 $$ \langle \deux^\top_{u}(X, N), Y \rangle + \langle N, \deux^\perp_{u}(X,Y) \rangle= 0$$ for every vector fields along $u$ such that
  $X \in \widetilde T_{u} \mathcal{L}, \, Y\in \widetilde{T}_{u} \mathcal{L}, \, N \in \widetilde N_{u} \mathcal{L}.$
   This yields that $\mathcal{B}_{\deux}$ is skew symmetric on $\widetilde T_{u}\mathcal{L} \times \widetilde N_{u} \mathcal{L}$  for the scalar product defined by \eqref{scalartilde1}.
   
To include the zero order terms in the right hand-side of \eqref{hydrodiffkahl},  we introduce  the tensor $\mathcal{V}(u,n) \in  \mathscr{L}(\widetilde T_{u} \mathcal{L} \times \widetilde N_{u} \mathcal{L}, \widetilde T_{u} \mathcal{L} \times \widetilde N_{u} \mathcal{L})$
      defined by 
    $$ \mathcal{V} (u,n)= \left( \begin{array}{cc}  0 &  0 \\ 0 & 2 \lambda   + 2 \epsilon^2  \lambda i \deux^\top(i S_{\perp \perp} n, \cdot )  \end{array} \right)
     $$
   to obtain   that
  $$
  {1 \over \epsilon^2}  \left( \begin{array}{cc} i \left( 
  2 \lambda  S_{\perp \perp} \nabla^{\perp m}n + 2\epsilon^2 \lambda i \deux^\top \big( i S_{\perp \perp} n, S_{\perp \perp} \nabla^{\perp m} n \big) \right) \\
   0 
   \end{array}\right)  =  {1 \over \epsilon^2} J \mathcal{V} \mathcal{M} U^m
   +  \left( \begin{array}{cc} O_{H^1}(\mathcal{E}_{s}) \\
 O_{H^1_{\epsilon}}(\mathcal{E}_{s}) \end{array}\right)
   .$$
   Moreover, let us notice that  $ i \deux^\top_{u}(X, \cdot)$ remains symmetric on $\widetilde N_{u} \mathcal{L}$. Consequently, we obtain that
    $\mathcal{V}$ is symmetric on $\widetilde T_{u}\mathcal{L} \times \widetilde N_{u} \mathcal{L}$ for the scalar product defined by \eqref{scalartilde1}. 
    
Finally, we will denote
$$ A_{\delta}\overset{def}{=} \left( \begin{array}{cc}  Id  & 0  \\ 0 & \delta Id \end{array} \right).$$ 
Then the right hand side of \eqref{hydrodiffkahl} reads
\begin{align*}
RHS\eqref{hydrodiffkahl} & = \frac{1}{\epsilon^2}  J \left( \frac{1}{2} \mathcal{D}_x \mathcal{A}_{\epsilon^2} \mathcal M \mathcal{D}_x U^m + \epsilon^2 \mathcal{B}_{II} \mathcal{M} \mathcal{D}_x U^m - \mathcal{V} \mathcal{M} U^m \right) \\
& \qquad \qquad + \left( \begin{array}{l} 0 \\ \frac{i}{2} \deux^\top (S_{\top \top} \nabla_x^\top \nabla^{\top m} \Phi^\epsilon, S_{\perp \perp} \nabla_x^\perp n ) \end{array} \right)+  \left( \begin{array}{cc} O_{H^1}(\mathcal{E}_{s}) \\
 O_{H^1_{\epsilon}}(\mathcal{E}_{s}) \end{array}\right).
\end{align*}
Gathering the expressions for the left and right hand sides of \eqref{hydrodiffkahl}, we proved the following proposition
\begin{prop}
On $[0, T^\epsilon]$, we have that $U^m =  (\nabla^{\top m} \Phi^\epsilon, \nabla^{\perp m}n )^t$ solves for $1 \leq |m| \leq s$ the system
  \begin{multline}
  \label{hydrosimple}
  ( \mathcal{D}_{t}- {c \over \epsilon^2} \mathcal{D}_{x}) U^m= {1 \over \epsilon^2} \mathcal{M}^{-1} J \Bigl( {1 \over 2} \mathcal{D}_{x}(\mathcal{A}_{\epsilon^2} \mathcal{M} \mathcal{D}_{x} U^m) + ( \mathcal{B}  + \epsilon^2 \mathcal{B}_{\deux}) \mathcal{M}\mathcal{D}_{x} U^m - \mathcal{V} \mathcal{M}U^m\Bigr)  \\
    +   \mathcal{M}^{-1} \left( \begin{array}{ll} 0  \\  {1 \over 2 } i  \deux^\top \big( S_{\top\top} \nabla_{x}^\top \nabla^{\top m} \Phi^\epsilon, S_{\perp \perp} \nabla_{x}^\perp n \big)  \end{array} \right) +  \left( \begin{array}{cc} O_{H^1}(\mathcal{E}_{s}) \\
 O_{H^1_{\epsilon}}(\mathcal{E}_{s}) \end{array}\right)
\end{multline}
  where $\mathcal{B}$, $\mathcal{B}_{\deux}$ are skew symmetric and $\mathcal{V}$  is symmetric on $\widetilde T_{u}\mathcal{L} \times \widetilde N_{u} \mathcal{L}$ for the scalar product defined by \eqref{scalartilde1} 
  \end{prop}

  As in the proof of Proposition \ref{esthydroplat}, we shall prove Proposition \ref{esthydroK} by taking the scalar product  of \eqref{hydrosimple} with $LU^m$ for a well chosen operator $L$. Set
$$ L U^m = - \mathcal{D}_{x}( \mathcal{M}^*  \mathcal{A}_{\epsilon^2} \mathcal{M} \mathcal{D}_{x} U^m) + 2 \mathcal{M}^* \mathcal{V} \mathcal{M} U^m  - 2 \mathcal{M}^*( \mathcal{B}
     + \epsilon^2 \mathcal{B}_{\deux}) \mathcal{M} \mathcal{D}_{x} U^m$$
     and the $L^2$ scalar product
     $$ (U,V)\overset{def}= \int_{\mathbb{R}} \langle U_{1}, V_{1}\rangle_u + \langle U_{2}, V_{2}\rangle_u \,dx.$$
     We shall prove the following crucial result for the proof of Proposition \ref{esthydroK}:
   \begin{prop}
\label{estimationhydroK} For any $m$ such that $1 \leq |m|\leq s$, we have that on $[0, T^\epsilon]$, 
\begin{multline*}
\frac{d}{dt} \left({ 1 \over 2 }\left(\mathcal{M} \mathcal{D}_{x} U^m, \mathcal{A}_{\epsilon^2} \mathcal{M} \mathcal{D}_{x} U^m \right)
 +   \big( \mathcal{V}  \mathcal{M} U^m , \mathcal{M} U^m \big) \right.  \\
  \left. -  2\left(  U_{2}^m, S_{\perp\perp}^* P^\perp B P^\top S_{\top\top} \nabla_{x}^\top U_{1}^m \right)
  -2 \epsilon^2\left(  U_{2}^m, S_{\perp\perp}^* \deux^\perp\left( S_{\top\top}D\Phi \partial_{x} \phi, S_{\top\top} \nabla_{x}^\top U_{1}^m\right) \right) \right) 
   \\= O(\mathcal{E}_s^2).
   \end{multline*}
\end{prop}
   
\begin{proof}
Taking the scalar product of~(\ref{hydrosimple}) with $ LU^m$, we get
\begin{equation}
   \label{dauphin20}
  \big( LHS\eqref{hydrosimple}, LU^m \big)= \big( RHS\eqref{hydrosimple}, LU^m \big).  
   \end{equation}
Split the left hand side into 
\begin{align*}
 &I =  \left(  \left(\mathcal{D}_{t} - {c \over \epsilon^2} \mathcal{D}_{x}\right) U^m, - \mathcal{D}_{x}(\mathcal{M}^* \mathcal{A}_{\epsilon^2} \mathcal{M}  \mathcal{D}_{x} U^m) \right), \\
 &II=  2   \left(  \left(\mathcal{D}_{t} - {c \over \epsilon^2} \mathcal{D}_{x}\right) U^m, \mathcal{M}^* \mathcal{V} \mathcal{M} U^m \right), \\
  &II_{B}=  - 2  \left(  \left(\mathcal{D}_{t} - {c \over \epsilon^2} \mathcal{D}_{x}\right) U^m, \mathcal{M}^* \mathcal{B} \mathcal{M} \mathcal{D}_{x} U^m \right), \\
  & II_{\deux}=   - 2  \epsilon^2 \left(  \left(\mathcal{D}_{t} - {c \over \epsilon^2} \mathcal{D}_{x}\right) U^m, \mathcal{M}^* \mathcal{B}_{\deux} \mathcal{M} \mathcal{D}_{x} U^m \right),
   \end{align*}
and the right hand side into
  \begin{align*}
   & III =  { 1 \over 2 \epsilon ^2} \left(  \mathcal{M}^{-1} J \left( \mathcal{D}_{x}(\mathcal{A}_{\epsilon^2} \mathcal{M} \mathcal{D}_{x} U^m) +  2( \mathcal{B}  + \epsilon^2 \mathcal{B}_{\deux}) \mathcal{M}\mathcal{D}_{x} U^m -  2 \mathcal{V}  \mathcal{M}U^m \right), LU^m \right), \\
   & IV =  \left( \mathcal{M}^{-1} \left( \begin{array}{ll} 0  \\  {1 \over 2 } i  \deux^\top \big( S_{\top\top} \nabla_{x}^\top \nabla^{\top m} \Phi^\epsilon, S_{\perp \perp} \nabla_{x}^\perp n \big)  \end{array} \right) ,  LU^m \right),
\end{align*}
so that \eqref{dauphin20} becomes
\begin{equation}
\label{dauphin2}
I+II+II_{B}+II_{\deux}= III + IV.
\end{equation}
 Note that the above terms have very similar properties to the ones that were defined in the proof of Proposition \ref{esthydroplat}.
  To handle $I$, we can rely on  \eqref{riemtopkahl}, \eqref{riemperpkahl} and \eqref{devDPsi} which yield $\mathcal{D}_{x} \mathcal{M} = O(\epsilon^2)$, $\mathcal{D}_{t} \mathcal{M} = O(1)$ to obtain
    that 
    \begin{equation}
    \label{unK} 2 I=  {d \over dt} \big(\mathcal{M} \mathcal{D}_{x} U^m, \mathcal{A}_{\epsilon^2} \mathcal{M} \mathcal{D}_{x} U^m \big)+ O(\mathcal{E}_s^2).
    \end{equation}
     To handle $II$, we use that $\mathcal{D}_{t}( \mathcal{M}^* \mathcal{V} \mathcal{M}) = O(1)$, $\mathcal{D}_{x}(\mathcal{M}^* \mathcal{V} \mathcal{M})= O(\epsilon^2)$
      and that $ \mathcal{M}^* \mathcal{V} \mathcal{M}$ is symmetric. This yields
     \begin{equation}
     \label{deuxK} II = {d \over dt} \big( \mathcal{V}  \mathcal{M} U^m , \mathcal{M} U^m \big).
     \end{equation}
     The term $II_{B}$ requires some more care. We first note that since $\mathcal{M}^* \mathcal{B} \mathcal{M}$ is skew-symmetric, we have that 
     $$-  II_{B} =  2   \left(\mathcal{D}_{t} U^m, \mathcal{M}^* \mathcal{B} \mathcal{M} \mathcal{D}_{x} U^m \right).$$
       Next, by using the definition of $\mathcal{M}$ and $\mathcal{B}$, we get that
     $$ -II_{B}= 2 \left(  (\nabla_{t}^\top U_{1}^m, S_{\top\top}^* P^\top B P^\perp S_{\perp \perp} \nabla_{x}^\perp U_{2}^m )+  (\nabla_{t}^\perp U_{2}^m,  S_{\perp \perp}^*
      P^\perp B P^\top  S_{\top\top}\nabla_{x}^\top U_{1}^m) \right)+ O(\mathcal{E}_s^2)$$
      and we shall manipulate the second term. We write
     \begin{multline*}  \left(\nabla_{t}^\perp U_{2}^m,  S_{\perp \perp}^*
      P^\perp B P^\top S_{\top\top} \nabla_{x}^\top U_{1}^m \right) = {d \over dt} \left(  U_{2}^m, S_{\perp\perp}^* P^\perp B P^\top S_{\top\top} \nabla_{x}^\top U_{1}^m \right)
      \\ - \left( U_{2}^m, \nabla_{t}^\perp \left(  S_{\perp\perp}^* P^\perp B P^\top S_{\top\top} \nabla_{x}^\top U_{1}^m \right)\right)
       \end{multline*}
       and observing that $ \nabla_X^\perp( S_{\perp\perp}^* P^\perp B P^\top S_{\top\top}) = O(\epsilon^2)X$ (using again \eqref{devDPsi} and the fact that $\nabla B= 0$ on $\mathcal{L}$), we obtain
       $$  \left( U_{2}^m, \nabla_{t}^\perp \left(  S_{\perp\perp}^* P^\perp B P^\top S_{\top\top} \nabla_{x}^\top U_{1}^m \right)\right) 
       =   \left( U_{2}^m,\ S_{\perp\perp}^* P^\perp B P^\top S_{\top\top} \nabla_{t}^\top  \nabla_{x}^\top U_{1}^m \right) + O(\mathcal{E}_s^2).$$
       Note that to get this estimate, it is crucial that the $x$ derivative is applied to $U_{1}^m$  and not $U_{2}^m$.
        Next, using \eqref{riemtopkahl},  the skew symmetry of $B$ and similar arguments as above, we obtain
     \begin{multline*}  \left(\nabla_{t}^\perp U_{2}^m,  S_{\perp \perp}^*
      P^\perp B P^\top S_{\top\top} \nabla_{x}^\top U_{1}^m \right) = {d \over dt} \left(  U_{2}^m, S_{\perp\perp}^* P^\perp B P^\top S_{\top\top} \nabla_{x}^\top U_{1}^m \right) \\
       -   (\nabla_{t}^\top U_{1}^m, S_{\top\top}^* P^\top B P^\perp S_{\perp \perp} \nabla_{x}^\perp U_{2}^m ) + O(\mathcal{E}_s^2).
       \end{multline*}
       This finally yields
    \begin{equation}
    \label{deuxBK}- II_{B}=   2  {d \over dt} \left(  U_{2}^m, S_{\perp\perp}^* P^\perp B P^\top S_{\top\top} \nabla_{x}^\top U_{1}^m \right) + O(\mathcal{E}_s^2)
      + O(\mathcal{E}_s^2).
      \end{equation}
      The term $II_{\deux}$ can be handled by similar arguments (note in particular that $\mathcal{M}^* \mathcal{B}_{\deux}\mathcal{M}$ is also skew symmetric)
       and is actually slightly easier due to the factor $\epsilon^2$  in front. We obtain
     \begin{equation}
     \label{deuxdeuxK} -II_{\deux}= 2 \epsilon^2  {d \over dt}    \left(  U_{2}^m, S_{\perp\perp}^* \deux^\perp\left( S_{\top\top}D\Phi \partial_{x} \phi, S_{\top\top} \nabla_{x}^\top U_{1}^m\right) \right).
     \end{equation}
    Let us now turn to the terms from the right hand-side. By the skew symmetry of $J$ and the choice of $L$, 
$$
   III=  { 1 \over 2 \epsilon ^2} \left(  \mathcal{M}^{-1}J \left( \mathcal{D}_{x}(\mathcal{A}_{\epsilon^2} \mathcal{M} \mathcal{D}_{x} U^m) +  2( \mathcal{B}  + \epsilon^2 \mathcal{B}_{\deux^\top}) \mathcal{M}\mathcal{D}_{x} U^m -  2 \mathcal{V} \mathcal{M}U^m \right),  - [ \mathcal{D}_{x}, \mathcal{M}^*]  \mathcal{M} \mathcal{A}_{\epsilon^2} \mathcal{D}_{x} U^m\right).
$$
For the commutator in the right hand side, by using the rough estimate $[\nabla_{x}^\top, S_{\top\top}^*]= O(\epsilon^2)$, $ [\nabla_{x}^\perp, S_{\perp \perp}^*]= O(\epsilon^3)$,  that comes from \eqref{devDPsi}, we obtain that
  $$ III= -  {1 \over 2 \epsilon^2}\left(S_{\top \top}^{-1} i S_{\perp \perp} (\nabla_{x}^\perp)^2 U_{2}^m, [\nabla_{x}^\top, S_{\top\top}^*] S_{\top\top}\nabla_{x}^\top U_{1}^m\right) +  O(\mathcal{E}_s^2).$$
  Since $[\nabla_{x}^{\top}, S_{\top\top}^*] = [ \nabla_{x}^\top, S_{\top\top}]^*$, we can use \eqref{macareuxK3},  \eqref{macareuxK4} and \eqref{almostsym} to get that
 \begin{equation}
 \label{troisK}
 III=   -  {1 \over 2}\left( i S_{\perp \perp} (\nabla_{x}^\perp)^2 U_{2}^m, \deux^\top_{u}\left( S_{\top\top} \nabla_{x}^\top U_{1}^m, S_{\perp \perp}\nabla_{x}^\perp n \right)
 \right)+O(\mathcal{E}_s^2).
 \end{equation}
  Finally it remains to handle $IV$. By counting powers of $\epsilon$, we can easily simplify it into
  \begin{multline}
  \label{quatreK} IV= {1 \over 2} \left(  S_{\perp\perp}^{-1} i \deux^\top\left( S_{\top\top}\nabla_{x}^\top U_{1}^m,  -S_{\perp \perp} \nabla_{x}^\perp n,\right), S_{\perp\perp}^{*} S_{\perp\perp} (\nabla_{x}^\perp)
  ^2 U_{2}^m \right) + O(\mathcal{E}_s^2) \\=  -  {1 \over 2} \left(  i \deux^\top_{u}\left( S_{\top\top}\nabla_{x}^\top U_{1}^m, S_{\perp \perp} \nabla_{x}^\perp n,\right), S_{\perp\perp} (\nabla_{x}^\perp)
  ^2 U_{2}^m \right) + O(\mathcal{E}_s^2).
  \end{multline}
  Note that again there is a crucial cancellation when we sum up $III$ and $IV$ thanks to the skew symmetry of $i$.
  
To conclude the proof of Proposition \ref{estimationhydroK}, it suffices to collect \eqref{unK}, \eqref{deuxK}, \eqref{deuxBK}, \eqref{deuxdeuxK}, \eqref{troisK}, \eqref{quatreK}. 

\bigskip
\noindent
\textit{Proof of Proposition \ref{esthydroK}:}
  It suffices to integrate in time the estimate of Proposition \ref{estimationhydroK}  and to prove that the left hand side gives a control of $\mathcal{E}_{s}^2$.
Using  rough expansions, we can first write that
   \begin{multline*}
   \left({ 1 \over 2 }\left(\mathcal{M} \mathcal{D}_{x} U^m, \mathcal{A}_{\epsilon^2} \mathcal{M} \mathcal{D}_{x} U^m \right)
 +   \big( \mathcal{V}  \mathcal{M} U^m , \mathcal{M} U^m \big) \right.  \\
  \left. -  2\left(  U_{2}^m, S_{\perp\perp}^* P^\perp B P^\top S_{\top\top} \nabla_{x}^\top U_{1}^m \right)
  -2 \epsilon^2\left(  U_{2}^m, S_{\perp\perp}^* \deux^\perp\left( S_{\top\top}D\Phi \partial_{x} \phi, S_{\top\top} \nabla_{x}^\top U_{1}^m\right) \right) \right)  \\
 =   \int_{\mathbb{R}}\Big( {1 \over 2} |\nabla_{x}^\top \nabla^{\top m} \Phi^\epsilon |^2  + {1 \over 2} \epsilon^2 | \nabla_{x}^\perp \nabla^{\perp m} n |^2 + 2\lambda| \nabla^{\perp m}n|^2 - 2  \nabla^{\perp m}n \cdot B_{\Phi}  \nabla^{\top m} \Phi^\epsilon \Big) \, dx + \epsilon O(\mathcal{E}_{s}(u,t)).
\end{multline*}
Moreover, we can also easily find a slightly modified energy for the case $m=0$.
We can thus conclude as in the end of the proof of Proposition \ref{esthydroplat}.
\end{proof}

\subsection{Reduction of \eqref{eqhydroK}}
The aim is to prove that  the control of $\mathcal{E}_{s}$, $s \geq 2$ provided on $[0, T^\epsilon]$  by  Proposition
\ref{propenergieK1} and Proposition \ref{esthydroK} allows to reduce \eqref{eqhydroK} to a hydrodynamical system set on $T_{\Phi}\mathcal{L} \times N_{\Phi}\mathcal{L}$
which is very similar at leading order to \eqref{eqhydro}. As a preliminary, we shall establish that
\begin{lem}
\label{lemreduction}
The following expansions hold on $[0, T^\epsilon]$ for the following tensors acting from  $T_{\Phi} \mathcal{L}$ or $N_{\Phi} \mathcal{L}$ to $T_{\Phi}\mathcal{L}$ or  $N_{\Phi}\mathcal{L}$, 
\begin{align*}
&  \left(S_{\top\top}^{-1} (iB)_{\top\top}\right)_{(\Phi, \epsilon^2 n)}=  (iB)_{\Phi} + O(\epsilon^4), \\
&   \left(S_{\perp \perp}^{-1}  (iB)_{\perp\perp }\right)_{(\Phi, \epsilon^2 n)}=  (iB)_{\Phi} + O(\epsilon^4), \\
&  \left( S_{\top\top}^{-1} i S_{\perp\perp}\right)_{(\Phi, \epsilon^2 n)}=  (S_{0}^{-1} i)_{\Phi} +  O(\epsilon^4), \\
 & \left( S_{\perp\perp}^{-1} i S_{\top\top}\right)_{(\Phi, \epsilon^2 n)} =  (i S_{0})_{\Phi} + O(\epsilon^4), \\
 &     \left( S_{\perp\perp}^{-1} i ( \nabla^\top_{(Y,N)} S_{\top\top}) \right)_{(\Phi, \epsilon^2 n)}= i  \deux_{\Phi}^\top(\cdot, N) + O(\epsilon^2), \quad \forall \;\; Y\in T_{\Phi}\mathcal{L}, \, N \in N_{\Phi}\mathcal{L}, \\
  &  \left( S_{\top\top} i (S_{\perp\perp}^*)^{-1}\right)_{(\Phi, \epsilon^2 n)} =  (S_{0}^{-1} i)_{\Phi} + O(\epsilon^4)
   \end{align*}
(recall that $S_{0}$ is defined by $S_{0}= Id + \epsilon^2 \deux^\top(\cdot, n)$ on $T\mathcal{L}$). 
     \end{lem}
   \begin{proof}
   Let us prove the first expansion. Fix $X \in T_{\Phi}\mathcal{L}$, $n  \in N_{\Phi }\mathcal{L}$  and define 
    $$ f(s) =\left(S_{\top\top}^{-1}(iB)_{\top\top}\right)_{(\Phi, sn)} X =  \left(S_{\top\top}^{-1} P^\top iB S_{\top\top}\right)_{(\Phi, sn)} (X,0)  + O(s^2) X$$ 
    by definition of $(iB)_{\top\top}$.
    Then, we have $f(0)= (iB)_{\Phi} X$. Moreover,  by using that $ \nabla i= 0$,  $\nabla B= 0$ on $\mathcal{L}$, and that
      $ (\nabla_n S_{\top \top })_{\Phi}= \deux^\top(\cdot, n)$ thanks to \eqref{devDPsi} and \eqref{macareux}, we obtain that
     $$ f'(0)=  -  \deux^\top_\Phi(iB X, n) + iB \deux_\Phi^\top(X,n) = 0$$
     where the final cancellation comes from Corollary \ref{BdeuxT}. Thus by Taylor expansion $f(s)= f(0)+ O(s^2)$.
     
     The other expansions follow from the same arguments and \eqref{devDPsi}. 
          
   \end{proof}
   
    We shall now simplify the expression of the system \eqref{eqhydroK}.
Start with the first line: multiplying it by $S_{\top\top}^{-1}$ and dropping higher order terms in $\epsilon$ gives
\begin{multline*}
     {  \partial_{t}\Phi \over \epsilon} - { 1 \over \epsilon^2}\left(c+ S_{\top\top}^{-1}(iB)_{\top\top} \right) {\partial_{x} \Phi \over \epsilon}
      \\ = S_{\top\top}^{-1} i  \left[  
      {1 \over 2  } \nabla_{x}^\perp \big( S_{\perp \perp}
     \nabla_{x}^\perp n \big) - {2\lambda  \over \epsilon^2} (S_{\perp \perp}^*)^{-1} n  
      + {1 \over 2} \deux^\perp\Big(  \Sigma \left(\begin{array}{c} {\partial_{x} \Phi \over \epsilon} \\ 0
      \end{array} \right),   \Sigma \left(\begin{array}{c} {\partial_{x} \Phi \over \epsilon} \\ 0 
      \end{array} \right) \Big)  \right] +  \epsilon O_{H^2}(\mathcal{E}_{s}).
        \end{multline*}  
    Next, note that
    $$  S_{\top\top}^{-1} i \nabla_{x}^\perp( S_{\perp \perp}  \nabla_{x}^\perp n )= 
     S_{\top \top}^{-1} i S_{\perp \perp}  (\nabla_{x}^\perp)^2 n +  (\nabla^\perp_{ (\epsilon D\Phi \partial_{x} \phi, \epsilon^2 \nabla_{x}^\perp n)} S_{\perp\perp}) \nabla_{x}^\perp n
      =   S_{\top \top}^{-1} i S_{\perp \perp}  (\nabla_{x}^\perp)^2 n  + \epsilon  O_{H^2}(\mathcal{E}_{s})$$
        and hence by using Lemma \ref{lemreduction} we get
    $$ S_{\top\top}^{-1} i \nabla_{x}^\perp( S_{\perp \perp}  \nabla_{x}^\perp n )= i_\Phi ( \nabla_{x}^\perp)^2 n  +   \epsilon  O_{H^2}(\mathcal{E}_{s}).$$
      We can also  use  Lemma  \ref{lemreduction}, to expand
     $ {1 \over \epsilon^2} S_{\top\top}^{-1}(iB)_{\top\top} $ and $  {2\lambda  \over \epsilon^2} S_{\top\top}^{-1} i(S_{\perp \perp}^*)^{-1} n $. Moreover, an immediate expansion gives
           $$  S_{\top\top}^{-1} i  \deux^\perp\Big(  \Sigma \left(\begin{array}{c} {\partial_{x} \Phi \over \epsilon} \\ 0
      \end{array} \right),   \Sigma \left(\begin{array}{c} {\partial_{x} \Phi \over \epsilon} \\ 0 
      \end{array} \right)\Big) = i_\Phi  \deux^\perp_{\Phi} \left( D\Phi \partial_{x} \phi,  D\Phi \partial_{x} \phi \right) + \epsilon O_{H^2}(\mathcal{E}_{s}).$$
       This yields  the following equation  on $T_{\Phi} \mathcal{L}$ where all the involved tensors are evaluated on $\mathcal{L}$ at $\Phi$: 
     \begin{multline*}
      D\Phi \partial_{t} \phi - {1 \over \epsilon^2} (c+iB ) D\Phi \partial_{x} \phi= 
       {1 \over 2 }  i (\nabla_{x}^\perp)^2 n  - {2\over \epsilon^2} \lambda  S_{0}^{-1} i n + {1 \over 2 } i \deux^\perp \left( D\Phi \partial_{x} \phi, D\Phi \partial_{x} \phi\right) +
         \epsilon O_{H^2}(\mathcal{E}_{s}). 
     \end{multline*} 
 By Corollary \ref{BdeuxT}, we can  also write it as
   \begin{multline}
   \label{eqphireduite}
    S_{1}  \partial_{t} \phi - {1 \over \epsilon^2} (c+iB) S_1 \partial_{x} \phi = 
        i  {1 \over 2 } \left( (\nabla_{x}^\perp)^2 n  - {2\over \epsilon^2} \lambda   n + {1 \over 2 }  \deux^\perp \left( D\Phi \partial_{x} \phi, D\Phi \partial_{x} \phi\right)  \right)+
         \epsilon O_{H^2}(\mathcal{E}_{s}). 
     \end{multline} 
Note that up to $O(\epsilon)$ term, we get back the same expression as in the first line of \eqref{eqhydro} (here we have simplified a little bit by assuming that
 $F_{1}= R^V= 0$).
  
   We can proceed in the same way  for the second line of \eqref{eqhydroK}, we multiply it by $S_{\perp \perp}^{-1}$  and we use lemma \ref{lemreduction} again.
    Let us give some details for the most involved  term in the right hand side: 
    $$ {1 \over \epsilon^2} S_{\perp \perp}^{-1} i \nabla_{x}^\top \left( S_{\top \top} {\partial_{x} \Phi \over \epsilon}\right)
     =  {1 \over \epsilon^2} S_{\perp \perp}^{-1} i S_{\top\top} \nabla_{x}^\top  {\partial_{x} \Phi \over \epsilon} + 
      {1 \over \epsilon^2} S_{\perp \perp}^{-1} i( \nabla^\top_{( \epsilon D\Phi \partial_{x} \phi, \epsilon^2 \nabla_{x}^\perp n) } S_{\top\top}){\partial_{x} \Phi \over \epsilon},$$
       and hence by using  Lemma \ref{lemreduction} again, we obtain
       $$ {1 \over \epsilon^2} S_{\perp \perp}^{-1} i \nabla_{x}^\top \left( S_{\top \top} {\partial_{x} \Phi \over \epsilon}\right) =  {1 \over \epsilon^2}(iS_{0})_{\Phi} 
        \nabla_{x}^\top {\partial_{x} \Phi \over \epsilon}+  i \deux^\top\left(\nabla_{x}^\perp n, {\partial_{x} \Phi \over \epsilon} \right)
         + \epsilon O_{H^2}(\mathcal{E}_{s}),$$
which can also be written, using the definition of $S_{0}$ and Proposition \ref{deuxsym}, under the form 
  $${1 \over \epsilon^2} S_{\perp \perp}^{-1} i   \nabla_{x}^\top \left( S_{\top \top} {\partial_{x} \Phi \over \epsilon}\right) =  {1 \over \epsilon^2}  i \left(
        \nabla_{x}^\top \left( S_{0} {\partial_{x} \Phi \over \epsilon} \right) \right)
         +  i \deux^\top\left(\nabla_{x}^\perp n, {\partial_{x} \Phi \over \epsilon} \right) + \epsilon O_{H^2}(\mathcal{E}_{s}).$$
We thus obtain for the second line of \eqref{eqhydroK}
\begin{multline}
\label{eqnreduite}
\nabla_{t}^\perp n  - {1 \over \epsilon^2} (c+iB) \nabla_{x}^\perp n = i \left(  {1 \over 2  \epsilon^2}  \nabla_{x}^\top \left( S_{0} {\partial_{x} \Phi \over \epsilon} \right)
 + {1 \over 2}  \deux^\top( D\Phi \partial_{x} \phi, \nabla_{x}^\perp n) \right) + \epsilon \, O_{H^2}(\mathcal{E}_{s})
\end{multline}
where again all the tensors are evaluated at $\Phi (\epsilon \phi) \in \mathcal{L}$.

Looking at \eqref{eqphireduite}, \eqref{eqnreduite}, we can use again~\eqref{H2} and Corollary  \ref{BdeuxT} to obtain that
\begin{prop}
 On  $ [0, T^\epsilon]$, we have that the solution $u= \Psi (\Phi (\epsilon \phi), \epsilon^2 n)$ of \eqref{SMK}  verifies the system 
 \begin{equation}
\label{eqhydrohamiltonK}
\left\{ 
\begin{array}{l} 
\displaystyle
S_1 \partial_t \phi= \frac{1}{  \epsilon^2}   i \Bigl[- 2 \lambda n  - i (c +  i B) S_{1} \partial_{x}   \phi   \\
 \displaystyle
   \mbox{\hspace{5cm}}  + \epsilon^2  \frac{1}{2} \deux^{\perp}(D\Phi \partial_x \phi, D\Phi \partial_{x} \phi) + \frac{1}{2} \epsilon^2 
(\nabla^{\perp}_x)^2 n ) \Bigr] + \epsilon  O_{H^2}(\mathcal{E}_{s}) \\
\displaystyle
\nabla_t^\perp n    = {1 \over \epsilon^2} i \left[ \frac{1}{2} \nabla_x^\top (S_1 \partial_x \phi) 
 - i  (c+ i B) \nabla_{x}^\perp n 
 + \epsilon^2 \frac{1}{2} \deux^{\top} \left(
 D\Phi \partial_x \phi \, , \,\nabla_x^\perp n  \right) \right] + \epsilon  O_{H^2}(\mathcal{E}^s)
\end{array}
\right.
\end{equation}
where all the tensors are evaluated at $\Phi \in \mathcal{L}$.
\end{prop}
     Note that up to the $O(\epsilon)$ remainders, this is a system on $T_{\Phi}\mathcal{L} \times N_{\Phi}\mathcal{L}$!

\subsection{Proof of Theorem \ref{theounif} in the case of a general K\"ahler manifold $\mathcal M$}
    The local well-posedness of smooth solutions for   \eqref{SMK} is classical
    (see for example \cite{Kenigandco}, \cite{Ding}) 
      and  we can proceed as in section \ref{prooftheounif1}. We define $T^*$ in the same way
      and use Propositions \ref{propenergieK1}, \ref{esthydroK} for a bootstrap argument.
       The only difficulty as before is to estimate $\|\epsilon \phi\|_{L^\infty}$, but this can be done
        as previously by using the system \eqref{eqhydrohamiltonK}. Note that here $i$ and $B$ are not necessarily constant tensors
         but since $\nabla i= \nabla B = 0$ on $\mathcal{L}$ , we can indeed proceed in the same way when we apply $\nabla^\top$ to 
         the first equation of    \eqref{eqhydrohamiltonK}.

\subsection{Proof of Theorem \ref{theoKdV} in the case of a general K\"ahler manifold $\mathcal M$}
     Again, by using  the system  \eqref{eqhydrohamiltonK}, we can proceed exactly as in section  
      \ref{sectionKdVlimit}.

\section{Remarks on the limit KdV system}
\label{ROTLKS}

The Cauchy problem for our KdV system on $T_{0}\mathcal{L}$ reads
\beq
\label{KdVfin}
\left\{ \begin{array}{ll}
2  c\partial_{t} A  = \frac{1}{4} \partial_{xxx}  A  +  \big({3 \over 2}  -\frac{2\mu}{\lambda} - {2c \over \lambda} (iB)_{0} \big) i_{0} \deux^\perp_{0} \left(  \partial_{x} A, A\right)  - {i_0 \over 2} F_{1,0}(i_0 \partial_{x} A, i_0 A) \\
A(t=0)=A_0.
\end{array} \right.
\eeq
Rescaling coordinates appropriately, and taking into account of the symmetry properties of $ i_{0}\deux_{0}^\perp$ and $F_{1}(0)$ (see \eqref{F1def}),  this can be written under the general form
\begin{equation}
\label{vectorKdV}
\left\{ \begin{array}{l}
\partial_t u - \partial_{xxx} u + \partial_x \mathcal{Q}(u,u) = 0 \\
u(t=0) = u_0,
\end{array} \right.
\end{equation}
where the unknown $u$ is valued in $\mathbb{R}^d$, and $\mathcal{Q}$ is a bilinear tensor $\mathbb{R}^d \times \mathbb{R}^d \rightarrow \mathbb{R}^d$ such that
\begin{equation}
\label{alouette}
(u,v,w) \mapsto \mathcal{Q}(u,v) \cdot w \;\; \mbox{is symmetric in $(u,v,w)$}.
\end{equation}
Indeed, $ (u,v,w) \mapsto i_{0} B_{0} i_{0} \deux_{0}^\perp(u,v) \cdot w=  B_{0} \deux^\perp_{0} (u,v) \cdot w = - i_{0} \deux^\perp_{0}(u,v) \cdot i_{0} B_{0} w $ is symmetric in all its arguments
  thanks to Proposition \ref{deuxsym} and Corollary \ref{BdeuxT}. 
\subsection{Hyperbolic structure}
Since  the matrix $\mathcal Q(u,\cdot)$ is symmetric for any $u$,
\begin{equation}
\label{hyper}\partial_t u + \partial_x\mathcal Q (u, u)=0\end{equation}
is a symmetric hyperbolic system. From the classical theory, we therefore get
\begin{thm}
\label{KdVtheo}
For $ s> 3/2$,  $ A_{0} \in H^s(\mathbb{R})$,   there exists  $T>0$ such that for every $\delta \in \mathbb{R}$ there is a unique
solution of
$$
\partial_t u - \delta \partial_{xxx} u + \partial_x \mathcal Q(u,u) = 0 ,
$$
 in $\mathcal{C}([0, T],  H^s) \cap \mathcal{C}^1([0, T], H^{s-1} \big).$ 
\end{thm}
Notice that this result does not use any dispersive properties. 

\medskip

The system \eqref{hyper} is a system of conservation laws with flux $f$ given by
$$ f(u)= \mathcal Q (u,u).$$
Let us restrict to the region where it is strictly hyperbolic and consider $\lambda (u)$ an eigenvalue of $Df(u)$, and $r(u)$ an associated unitary eigenvector.
 We have 
$$  Df(u) \cdot r(u)= \lambda(u) r(u), \quad \forall u \in \mathcal{U}.$$
Let us differentiate this identity in the direction $r(u)$.  We find
$$ D^2f(u) (r, r) + Df(u) \cdot\big(Dr(u) \cdot r(u)\big)= D\lambda(u) \cdot r(u) r(u)+ \lambda (u) Dr(u) \cdot r(u).$$
Taking the scalar product of this identity with $r(u)$,  and using that  $\big(Dr(u)\cdot r(u)\big) \cdot r(u) = 0$ (since $r(u)$ is unitary)
and that $D^2 f= 2 \mathcal Q$, we obtain
$$   2\mathcal Q (r, r) \cdot r +\big( Df \cdot ( Dr\cdot r) \big) \cdot r= D\lambda\cdot  r.$$
By symmetry of $Df$,
$$\big( Df \cdot ( Dr\cdot r) \big) \cdot r =\big( Dr \cdot r \big) \cdot  \big( Df\cdot r \big) = \lambda (Dr \cdot r) \cdot r = 0, $$
therefore, we finally obtain that 
$$ D\lambda \cdot r =  2 \mathcal Q (r,r) \cdot r.$$
Consequently the characteristic fied $\lambda$ is  genuinely nonlinear in  an open set $\mathcal{U}$  if  and only if $\mathcal Q (r(u), r(u)) \cdot r(u) \neq 0, \quad \forall u \in \mathcal{U}$.

In particular, assuming that in an open set $\mathcal{U}$, the system is strictly hyperbolic and such that for any eigenvector $r$, $\mathcal Q (r,r) \cdot r  \neq 0$, the result of John \cite{John}, for example, gives that singularities occur in finite time.

\subsection{Hamiltonian structure}
The equation~\eqref{vectorKdV} derives formally from the Hamiltonian
\begin{equation}
\label{grebe}
H(u) =\int_{\mathbb{R}} \left[\frac{1}{2} \left|\partial_x u\right|^2 + \frac{1}{3} \mathcal Q (u,u) \cdot u \right]\,dx
\end{equation}
given the symplectic form
$$
\omega(f,g) = \int_{\mathbb{R}} \partial_x^{-1} f g \,dx.
$$
Other conserved quantities are the mass
(which is related via the Noether theorem to the invariance by translation of~(\ref{vectorKdV}))
$$
M(u) = \int_{\mathbb{R}} |u|^2\,dx.
$$
and the ``momentum''
$$
P(u) = \int_{\mathbb{R}} u\,dx.
$$
The class of Hamiltonians given by~\eqref{grebe} is equal to the 
class of Hamiltonians of the type 
$$\int \left[ \frac{1}{2} \left|\partial_x u\right|^2 + P(u) \right]\,dx,$$ where $P$ is a 
trilinear  form.

A more general class consists of Hamiltonians of the type $\int \left[ Q(\partial_x u) + P(u) \right]\,dx$, with $Q$ a
bilinear form, and $P$ a trilinear form. In the case $d=2$, this class of Hamiltonians gives the Gear-Grimshaw equations, which were derived in 
the context of water waves~\cite{Gear-Grimshaw}. A mathematical investigation of their properties was conducted in~\cite{BPST}.

\subsection{Local and global well-posedness} Let us first mention the central result in~\cite{BPST}. As already mentioned, this paper focuses on
Hamiltonians of the type $\int \left[ Q(\partial_x u) + P(u) \right]\,dx$, with $d=2$, $P$ a trilinear and $Q$ a bilinear form. 
Relying on methods introduced in~\cite{KPV0}, global well-posedness is established for data in $H^1$ when the bilinear form $Q$ is coercive 
(which is automatically the case for $H$ as in~(\ref{grebe})).

Focusing now on the equation~\eqref{vectorKdV}, we observe that the linearized problem is simply $\partial_t u - \partial_{xxx} u=0$, therefore
only one dispersion relation is present in the problem, namely $\xi^3$. It is then clear that the argument in~\cite{KPV} applies, giving the following
theorem (we refer to~\cite{KPV} for a definition of the $X^{s,b}$ space).

\begin{thm}
Let $s \in (-\frac{3}{4},0]$ and $b \in (\frac{1}{2},1)$. If $u_0 \in H^s$, there exists $T>0$ and a unique solution of~\eqref{vectorKdV} belonging 
to  $X^{s,b}(0,T)$ and 
$\mathcal{C}([0, T],  H^s)$.
\end{thm}

Combining this theorem with the conservation of the $L^2$ norm, one obtains global well-posedness for $L^2$ data.

\subsection{Solitary waves}

In the case of scalar KdV ($d=1$, $\mathcal Q (u,u)=u^2$), solitary waves are of the form $cq(\sqrt{c}(x+ct))$, where $c>0$, and $q$ is real-valued and solves the ODE
\begin{equation}
\label{mesangecharbonniere}
q' - q''' + (q^2)' = 0.
\end{equation}
It is known since the original paper of Korteweg and De Vries~\cite{KdV} that periodic wave solutions are given by Jacobi elliptic functions, while
a finite energy solitary wave on $\mathbb{R}$ is given by
$$
q = - 3 \operatorname{sech}^2.
$$
For the equation~\eqref{vectorKdV} in the general case, solitary waves are also of the type $cQ(\sqrt{c}(x+ct))$, where $c>0$, and $Q$ is valued in $\mathbb{R}^d$ and solves the ODE
\begin{equation}
\label{mesangebleue}
Q' - Q''' + \mathcal Q (Q,Q)' = 0.
\end{equation}
The following lemma allows to reduce partially~\eqref{mesangebleue} to~\eqref{mesangecharbonniere} 

\begin{lem}
For $\mathcal Q$ non zero satisfying~(\ref{alouette}), there exists $z \in \mathbb{R}^d$, $z \neq 0$, such that $\mathcal Q (z,z) = z$.
\end{lem}

\begin{proof}
A solution of $\mathcal Q (z,z) = z$ is a critical point of the functional $K(z) \overset{def}{=} \frac{1}{3} \mathcal Q(z,z)\cdot z - \frac{1}{2} |z|^2$.
The existence of a nonzero critical point can be deduced from the (finite dimensional) mountain pass lemma: it suffices to observe first 
that $K(0)=0$; that there exists $\delta,\epsilon>0$ such that $K(z)<-\delta$ if $|z|=\epsilon$; and that $K(\lambda y) \overset{\lambda \rightarrow \infty}{\rightarrow}\infty$
if $y$ is such that $B(y,y)\cdot y>0$. 

It remains to check the (finite dimensional) Palais-Smale condition: suppose that $z^n$ is a sequence such that
$K(z^n)$ is bounded and $K'(z^n) = \mathcal{Q} (z^n,z^n) - z^n \rightarrow 0$. But then
$$
|3K(z^n) - K'(z^n)\cdot z^n| \lesssim 1 + o(z^n) \quad \mbox{and} \quad 3K(z^n) - K'(z^n)\cdot z^n =- \frac{1}{2}|z^n|^2.
$$
This implies that $(z^n)$ is bounded, hence the desired conclusion.
\end{proof}

This lemma implies that if $Q = q z$, with $z$ given by the previous lemma, and $q \in \mathbb{R}$, $Q$ solves~\eqref{mesangebleue} if and only if
$q$ solves~\eqref{mesangecharbonniere}. This results in the following proposition.

\begin{prop}
If $z$ is given by the previous lemma, and $q = - 3 \operatorname{sech}^2$, then $cq(\sqrt{c}(x+ct))z$ is a solution of~\eqref{vectorKdV} for any $c>0$. 
\end{prop}

An interesting question is whether there are other stationary waves than the above; one might even think that the set of traveling waves of finite energy
is of dimension $d+1$, since the stable and unstable manifolds of~\eqref{mesangebleue} at $(Q,Q')=(0,0)$ have dimension $d$.

\subsection{Miura transform and integrability} The classical Miura transform $u = \partial_x v + \frac{1}{6}v^2$ turns a solution of the mKdV equation $\partial_t v - \partial_x^3 v + \frac{1}{6} v^2 \partial_x v = 0$ into a solution of the KdV equation $\partial_t u - \partial_x^3 u + u \partial_x u =0$.

We wish to generalize this transformation to the vector case~\eqref{vectorKdV}. We therefore look for a symmetric quadratic form $B$ such that the transformation
$$
u = \partial_x v + B(v,v)
$$
maps solutions of the vector mKdV
$$
\partial_t v - \partial_{xxx} v + \mathcal{T}(v,v,\partial_x v)
$$
to solutions of~\eqref{vectorKdV}. We assume here that $\mathcal{T}$ is symmetrical in its arguments, but lifting this restriction does not lead to an improvement of the conditions on $\mathcal{Q}$. 

A short computation gives the equalities
\begin{align*}
 \partial_t u - \partial_{xxx} u  &= - 2 \mathcal{T}(v,\partial_x v,\partial_x v) - \mathcal{T}(v,v,\partial_{xx} v) - 4 B (v,\mathcal{T}(v,v,\partial_x v) - 6 B(\partial_x v,\partial_{xx} v) \\
 - \partial_x \mathcal{Q}(u,u) &= - 2 \mathcal{Q}(\partial_x v,\partial_{xx} v) - 4 \mathcal{Q}(\partial_x v,B(v,\partial_x v)) - 2 \mathcal{Q}(\partial_{xx} v,B(v,v)) - 4 \mathcal{Q}(B(v,v),B(v,\partial_x v)).
\end{align*}\
Identifying similar terms on the right-hand sides leads to the necessary condition: for any $X,Y,Z \in \mathbb{R}^d$, $\mathcal{Q}(X,\mathcal{Q}(Y,Z)) = \mathcal{Q}(Y,\mathcal{Q}(X,Z))$, and further
$$
B = \frac{1}{3} \mathcal{Q} \quad \mbox{and} \quad \mathcal{T}(X,Y,Z) = \frac{2}{3} \mathcal{Q}(X,\mathcal{Q}(Y,Z)).
$$
Summarizing, we proved the following lemma.

\begin{lem}
\label{lemmiura}
If $\mathcal{Q}$ satisfies $\mathcal{Q}(X,\mathcal{Q}(Y,Z)) = \mathcal{Q}(Y,\mathcal{Q}(X,Z))$ for all $X,Y,Z \in \mathbb{R}^d$, then
the Miura transform
$$
u = \partial_x v + \frac{1}{3} \mathcal{Q}(v,v)
$$
maps solutions of the equation
$$
\partial_t v - \partial_{xxx} v + \frac{2}{3} \mathcal{Q}(v,\mathcal{Q}(v,\partial_x v)) = 0
$$
to solutions of~\eqref{vectorKdV}.
\end{lem}

Observe that the $v$ equation above derives from the Hamiltonian 
$$
\int_\mathbb{R} \left[ \frac{1}{2} |\partial_x v|^2 + \frac{1}{6} \mathcal{Q}(v,\mathcal{Q}(v,v)) \cdot v \right] \,dx.
$$
When the above lemma applies, this bi-Hamiltonian structure results in an infinite number of conserved quantities~\cite{Olver}.

\subsection{The case $d=2$}

To illustrate the above, we consider the first non-trivial case: $d=2$. It is then convenient to identify $\mathbb{R}^2$ with $\mathbb{C}$. The condition~\eqref{alouette} is satisfied for 
$$
\mathcal{Q}(x,y) = \alpha xy + \beta \overline{xy} + \overline{\alpha} (\overline{x}y+x\overline{y}) \qquad \mbox{with $\alpha, \beta \in \mathbb{C}$}.
$$
The condition for the existence of a Miura transform in Lemma~\ref{lemmiura} becomes then $|\alpha| = |\beta|$. 

\bigskip

\noindent
{\bf Acknowledgements:} The authors are grateful to Jalal Shatah for suggesting the question at the heart of this article, to Herbert Koch for very helpful discussions during the writing of this article, and to Or Hershkovits for his geometric insight.

\end{document}